\documentclass[12pt]{amsart}
\usepackage[english]{babel}
\usepackage{amsmath}
\usepackage{amsthm}
\usepackage{mathtools}
\usepackage{mathrsfs}
\usepackage{amssymb}
\usepackage{xcolor}
\usepackage{xfrac}
\usepackage{esint}
\usepackage{graphicx}
\usepackage{float}
\usepackage{array}
\usepackage{enumitem}
\usepackage[new]{old-arrows}
\usepackage[all,cmtip]{xy}
\usepackage{tikz-cd}
\usepackage{centernot}
\usepackage{stmaryrd}
\usepackage{comment}
\excludecomment{confidential}

\numberwithin{equation}{section}

\newtheorem{theorem}{Theorem}[section]
\newtheorem{corollary}[theorem]{Corollary}
\newtheorem{lemma}[theorem]{Lemma}
\newtheorem{proposition}[theorem]{Proposition}
\newtheorem{definition}[theorem]{Definition}

\DeclareMathOperator{\Alg}{Alg}
\DeclareMathOperator{\ord}{ord}
\DeclareMathOperator{\alg}{alg}

\DeclareMathOperator{\codim}{codim}

\DeclareMathOperator{\Deck}{Deck}

\newcommand{\N}{\mathbb{N}}
\newcommand{\Z}{\mathbb{Z}}

\newcommand{\R}{\mathbb{R}}
\newcommand{\C}{\mathbb{C}}

\newcommand{\field}[1] {\mathbb{#1}}
\newcommand{\K}{\field{K}}

\def\a{\alpha}
\def\b{\beta}
\def\e{\varepsilon}
\def\D{\Delta}
\def\d{\delta}
\def\g{\gamma}
\def\G{\Gamma}
\def\l{\lambda}

\def\o{\omega}
\def\O{\Omega}
\def\p{\partial}
\def\r{\rho}

\def\s{\sigma}

\def\v{\varphi}

\def\ua{\uparrow}
\def\da{\downarrow}

\newcommand{\mc}{\mathcal}
\newcommand{\mf}{\mathfrak}

\DeclareMathOperator{\Int}{int}

\begin{document}
	\title{Bifurcation Theory for Fredholm Operators}
	\author{Juli\'an L\'opez-G\'omez, Juan Carlos Sampedro} \thanks{The authors have been supported by the Research Grant PID2021--123343NB-I00 of the Spanish Ministry of Science and Innovation and by the Institute of Interdisciplinar Mathematics of Complutense University.}
	\address{Institute of Interdisciplinary Mathematics \\
Department of Mathematical Analysis and Applied Mathematics \\
		Complutense University of Madrid \\
		28040-Madrid \\
		Spain.}
	\email{julian@mat.ucm.es, juancsam@ucm.es}

\begin{abstract}
This paper consists of four parts. It begins by using the authors's generalized Schauder formula, \cite{JJ}, and the algebraic multiplicity, $\chi$, of Esquinas and L\'{o}pez-G\'{o}mez \cite{ELG,Es,LG01} to package and sharpening all existing results in local and global bifurcation theory for Fredholm operators through the recent author's axiomatization of the Fitzpatrick--Pejsachowicz--Rabier degree, \cite{JJ2}. This facilitates reformulating and refining  all existing results in a compact and unifying way. Then, the local structure of the solution set of analytic nonlinearities $\mathfrak{F}(\l,u)=0$ at a simple degenerate eigenvalue is ascertained by means of some concepts and devices of Algebraic Geometry and Galois Theory, which establishes a bisociation between Bifurcation Theory and Algebraic Geometry. Finally, the unilateral theorems of \cite{LG01,LG02}, as well as the refinement of Xi and Wang \cite{XW}, are substantially generalized. This paper also analyzes two important examples to illustrate and discuss the relevance  of the abstract theory. The second one studies the regular positive solutions of a multidimensional quasilinear boundary value problem of mixed type related to the mean curvature operator.
\end{abstract}


\maketitle
\section{Introduction}

\noindent This paper is a natural continuation of \cite{JJ,JJ2,JJ3}, where the authors axiomatized
the degree for Fredholm operators of  Fitzpatrick, Pejsachowicz and Rabier \cite{FPRa,FPRb,PR} through the generalized algebraic multiplicity, $\chi$, of Esquinas and L\'{o}pez-G\'{o}mez \cite{Es,ELG,LG01}, which had been axiomatized by Mora-Corral \cite{MC}. The monographs \cite{LG01} and \cite{LGMC} present a rather
complete synthesis of the crucial role played by $\chi$ in Bifurcation Theory. The axiomatization of the degree for Fredholm maps carried out by the authors in \cite{JJ2} follows similar patterns as the uniqueness theorems of F\"{u}hrer \cite{Fu}  and Amann and Weiss \cite{AW} for the degrees of Brouwer \cite{Br}  and  Leray--Schauder \cite{LS}, respectively.
\par
The first part of this paper invokes these, rather recent, developments to reformulate all existing local and global bifurcation results for nonlinear Fredholm maps at the light of the algebraic multiplicity $\chi$ through the generalized Schauder formula delivered by the authors in \cite{JJ}. In particular, it tidies up and sharpens considerably the local bifurcations theorems of Fitzpatrick, Pejsachowicz and Rabier \cite{FPRb}, Pejsachowicz and Rabier \cite{PR}, as well as \cite[Th. 6.2.1]{LG01} and the global theorem of L\'{o}pez-G\'{o}mez and Mora-Corral \cite{LGMC}, which was originally proven by
using the  degree of Benevieri and Furi \cite{BF1,BF2,BF3}; we deliver another
proof of \cite[Th. 5.4]{LGMC} by means of the Fitzpatrick--Pejsachowicz--Rabier degree. Actually, this paper grew from the germinal idea of bringing together, by the first time, all these results, scattered in a series of independent monographs, at the light of our most recent developments concerning $\chi$ and the Fitzpatrick--Pejsachowicz--Rabier degree.
\par
Roughly speaking,  our global theorem (collected in Theorems \ref{T6.3.5,2} and \ref{C6.3.6}) establishes that any compact component $\mathfrak{C}$ of the set of non-trivial solutions of a nonlinear equation $\mathfrak{F}(\l,u)=0$ must bifurcate from the trivial solution $(\l,u)=(\l,0)$ respecting the topological  property that the sum of the parities of the compact components
of the generalized spectrum where $\mathfrak{C}$ bifurcates from  $(\l,0)$ must be zero. In the proof of the global alternative of Rabinowitz \cite{Ra}, Nirenberg \cite[p. 87]{Ni} emphasized this property in the context of  the Leray--Schauder degree for nonlinear maps $\mf{F}(\l,u)$ such that $\mf{F}(\l,0)=0$ for all $\l\in\R$ and $\mf{L}(\l)\equiv D_u\mf{F}(\l,0)=\l I-K$ for some compact operator $K$. Theses findings were later sharpened by Ize \cite{Iz} and Magnus \cite{Ma}, though in a rather limited way, because of the absence of a versatile theory of algebraic multiplicities for arbitrary compact perturbations of the identity map, $\mf{L}(\l)=I-\mc{K}(\l)$ with $\mc{K}(\l)$ compact.  Examples of compact components bifurcating from the trivial solution in the context of Reaction-Diffusion systems are well known to arise in a huge number of applications. The interested reader is sent to Chapter 2 of \cite{LG01}, and to Fencl and L\'{o}pez-G\'{o}mez \cite{FLG} for some recent intriguing applications.
\par
Naturally, as a byproduct of the global bifurcation theorem, the \emph{global alternative of Rabinowitz}, \cite{Ra}, holds. It establishes that any component bifurcating from a nonlinear eigenvalue must be either unbounded, or it bifurcates from $(\l,0)$ at two different values of the parameter $\l$ with opposite parities. Possibly by the simplicity of this formulation and the sophisticated topological technicalities
necessary to state rigourously and prove the global theorem, even the simplest version of Nirenberg \cite{Ni} has fallen into oblivion for almost 50 years. As a result, experts use to invoke exclusively the global alternative of Rabinowitz in the context of global bifurcation theory, which is really harmful from the point of view of the applications, as the global alternative cannot provide with any useful information on the value of the degree of the $\l$-slices of the component, which in many applications coincides with the exact number of solutions of $\mf{F}=0$ (see, e.g., L\'{o}pez-G\'{o}mez and C. Mora-Corral \cite{LGMC05}, \cite{LGM}). Actually, some authors did not
focus attention into the deeper insights of the underlying theory, as, e.g., Shi and Wang \cite{XW}, whose observation that the local theorem of Crandall and Rabinowitz \cite{CR} is global was  actually a straightforward direct consequence of Theorem 5.4 of L\'{o}pez-G\'{o}mez and Mora-Corral \cite{LGMC}.
\par
Once tidied up, polished and sharpened the main bifurcation theorems at the light of the  degree of Fitzpatrick, Pejsachowicz and Rabier, the second part of this paper studies the bifurcation from simple degenerate eigenvalues for analytic $\mathfrak{F}$'s, where  $N[D_{u}\mathfrak{F}(\lambda_{0},0)]=\text{span}[\varphi_{0}]$ for some $\varphi_{0}\in U\backslash\{0\}$, which goes back to the pioneering work of Dancer \cite{Da,Da73, Da732} and Kielh\"ofer \cite{Ki}. By Theorem 4.4.3 of L\'{o}pez-G\'{o}mez and Mora-Corral \cite{LGM}, the main result of \cite{Ki} can be stated by simply saying that in a neighborhood of $(\lambda_{0},0)$,  $\mathfrak{F}^{-1}(0)$ consists of
at most $2\chi+2$ branches of analytic functions intersecting at $(\lambda_{0},0)$. In this article, we use a rather different approach by using a number of technical tools in Algebraic and Analytic Geometry to ascertain the fine local structure of $\mathfrak{F}^{-1}(0)$ at $(\lambda_{0},0)$. Precisely, after using a
Lyapunov--Schmidt decomposition to reduce the original problem to another one with a finite-dimensional character, the analysis of the local structure of the associated solution set will be carried over through a careful study of the roots of the Weierstrass polynomial associated to the reduced bifurcation equation. This allows us to use some techniques of algebraic monodromy theory to study them by means of the theory of Riemann surfaces. Thus, essentially, we translate the underlying analytic problem into another one of algebraic nature that can be treated computationally. In the complex case, using these technical devices, one can characterize the precise local structure of the zero set in terms of a sequence of Riemann surfaces that can be analyzed through an algebraic field extension of finite degree. At a later stage,  we will use Galois Theory to ascertain whether, or not, the zero set can be expressed through a composition of radicals and meromorphic functions. In the real case, this problem has a rather computational nature, since one has to determine how many complexified Riemann surfaces are  \emph{real-to-real}, i.e., such that its intersection with the real plane lies into the real plane. Finally, we are able to determine the exact number of real branches that emanate from $(\lambda_{0},0)$ through the celebrated Sturm criteria, that is used in Numerical Analysis to ascertain the exact number of real zeroes of polynomials. This analysis complements the one given by Dancer and Kielh\"ofer. In a further step we focus attention on the global behavior of the zero set $\mf{F}^{-1}(0)$ of these type of nonlinearities. In this setting, we generalize the Dancer \cite[Th. 4]{Da732} and Buffoni--Toland \cite[Th. 9.1.1]{BT} global alternatives up to cover the degenerate case $\chi\geq 2$ 
where the theorem of Crandall and Rabinowitz \cite{CR} cannot be applied. This is a significantly important advance for the study of degenerate problems as illustrated in the example of Sections \ref{SDODP} and \ref{SNS}. Moreover,
inspired by a novel idea of Dancer \cite[Th. 3]{Da732}, we will prove that, under these assumptions, $\mf{F}^{-1}(0)$ is an \textit{analytic graph}. As a very special example, by the local theorem of Crandall and Rabinowitz \cite{CR} and the global alternative of Rabinowitz \cite{Ra}, it is folklore that,  for every $a\in \mc{C}[0,\pi]$ and any integers $n\geq 1$ and $p\geq 2$, the set of solutions of the semilinear  boundary value problem
\begin{equation}
	\label{1.1}
	\left\{\begin{array}{l}
		-u''=\lambda u +a(x) u^p \quad \hbox{in}\;\, (0,\pi), \\
		u(0)=u(\pi)=0,
	\end{array}
	\right.
\end{equation}
admits a component, $\mathscr{C}_n$, with $(\l,u)=(n^2,0)\in \mathscr{C}_n$, which is unbounded in
$\R\times \mc{C}[0,\pi]$. Moreover, by the maximum principle, since the number of nodes of
the solutions along $\mathscr{C}_n$ is constant, it turns out that
$
\mathscr{C}_n\cap \mathscr{C}_m=\emptyset$, $ n\neq m
$.
Our main result in this part  shows that actually each of the components $\mathscr{C}_n$, $n\geq 1$, consists of  a discrete set of
analytic arcs of curve plus a discrete set of branching points.
\par
The third part of this paper generalizes substantially the unilateral bifurcation theorems of \cite{LG02} by substituting the norm in the underlying Banach space $U$ by some continuous functional $\psi(u)$. As a byproduct, our unilateral theorem does not require the differentiability of the norm of $U$, as in Shi and Wang \cite{XW}, but simply the compactness of the  imbedding $U\hookrightarrow V$, which is a rather common property in most of the existing applications. Another important feature of our refinement is that it can be applied to deal with the general case when $\chi\geq 1$. So, it also covers the case of bifurcation from simple degenerate eigenvalues.
\par
The main advantage of developing global bifurcation theory in a Fredholm scenario is that one can deal very easily with quasilinear elliptic equations and systems, even when they cannot be transformed into a semilinear boundary value problem though some  change of variables, tricky or not. As in this scenario one can directly deal with Fredholm operators, expressing the differential equation as an integral equation
is unnecessary. This paper ends by giving a non-trivial application to a quasilinear boundary value problem of mixed type involving the multidimensional mean curvature operator.
\par
This paper has been organized as follows. Section \ref{SGAM} collects the main features of the algebraic multiplicity $\chi$ used in this paper. Section \ref{STDFO} reviews the main ingredients of the degree of
 Fitzpatrick, Pejsachowicz and Rabier \cite{FPRa,FPRb,PR}, and invokes them to give a general version of the Lereay--Schauder continuation theorem for nonlinear Fredholm operators. Sections \ref{SLBT} and \ref{SGBT} deliver the main local bifurcation theorem and the main global bifurcation theorem of this paper, respectively.
Section \ref{SLBA}  consists of a sharp analysis, at the light of Analytic Geometry, of the problem of bifurcation in the special case of analytic $\mf{F}$'s and
$\dim N[D_{u}\mathfrak{F}(\lambda_0,0)]=1$. In particular, it delivers our refinements of the main findings of Dancer \cite{Da,Da73, Da732} and  Kielh\"{o}fer \cite{Ki}. Section \ref{SAGA} generalizes the Dancer \cite{Da732} and Buffoni--Toland \cite{BT} global alternatives up to cover the degenerate case $\chi\geq 2$. Section \ref{SDODP} invokes the previous theory to analyze the local and global structure of the solution set of
\begin{equation}
\label{1.2}
\left\{\begin{array}{l}
-u''=\lambda u' +u+(\lambda-u^2)u^2 \quad \hbox{in}\;\, (0,\pi), \\
u(0)=u(\pi)=0.
\end{array}
\right.
\end{equation}
In particular, it is shown that the set of positive solutions of \eqref{1.2} contains a closed loop
bifurcating from $(\l,u)=(0,0)$. This example seems to be the first of this type constructed, analytically, in the literature in a situation where $\chi=2$. The existence of such loops in problems with weights is well documented in the literature (see, e.g., Fencl and L\'{o}pez-G\'{o}mez \cite{FLG} and the references there in). Actually, introducing an additional parameter in the model setting one can get, numerically, these closed loops, as in L\'{o}pez-G\'{o}mez and Molina-Meyer \cite{LGMM}, but our result for \eqref{1.2} is the first existing analytical result. Section \ref{SUBGSE} delivers the refinement of the unilateral theorems of \cite{LG02} and \cite{XW}, and Section \ref{SNS} uses these refinements to analyze the global structure of the negative solutions of \eqref{1.2}; so, completing the analysis of Section \ref{SDODP}. Finally, Section \ref{SQPMT} applies the new unilateral theorem to analyze the global structure of the set of positive solutions of the quasilinear problem
\begin{equation}
\label{1.3}
\left\{\begin{array}{ll}
-\text{div}\left(\frac{\nabla u}{\sqrt{1+|\nabla u|^{2}}}\right)=\lambda a(x) u +g(x,u)u & \text{ in } \Omega, \\[1ex]
\mathcal{B}u=0 & \text{ on } \partial\Omega,
\end{array}
\right.
\end{equation}
where $\mc{B}$ is a general boundary operator of mixed type.
\par Along this paper, given a pair $(U,V)$ of $\K$-Banach spaces, $\K\in\{\R,\C\}$, the space of linear bounded operators $T:U\to V$ is denoted by $\mc{L}(U,V)$. Naturally, we set $\mc{L}(U):=\mc{L}(U,U)$.  We denote by $GL(U,V)$ the space of topological isomorphisms and $GL(U):=GL(U,U)$.
Given $T\in \mc{L}(U,V)$, we denote by $N[T]$ and $R[T]$, the kernel and the range of $T$, respectively.
Finally, $\mc{K}(U)$ stands for the set of linear compact endomorphisms of $U$, and
$GL_c(U)=GL(U)\cap \mc{K}(U)$ is the compact linear group of $U$.

\section{Generalized algebraic multiplicity}\label{SGAM}

\noindent Throughout this section, $\mathbb{K}\in\{\mathbb{R},\mathbb{C}\}$, $\Omega$ is a subdomain of $\mathbb{K}$, and, for any given finite dimensional curve $\mathfrak{L}\in\mathcal{C}(\Omega,\mathcal{L}(\mathbb{K}^{N}))$,
a point $\l\in\Omega$ is said to be a \textit{generalized eigenvalue} of $\mathfrak{L}$ if $\mathfrak{L}(\l)\notin GL(\mathbb{K}^{N})$, i.e., $\mathrm{det\,}\mf{L}(\l)=0$. Then, the
\textit{generalized spectrum} of $\mathfrak{L}\in\mathcal{C}(\Omega,\mathcal{L}(\mathbb{K}^{N}))$, denoted by 
$\Sigma(\mathfrak{L})$, consists of the set of $\lambda\in\Omega$ such that $\mathfrak{L}(\lambda)\notin GL(\mathbb{K}^{N})\}$. For analytic curves $\mathfrak{L}\in\mc{C}^{\omega}(\O,\mathcal{L}(\mathbb{K}^{N}))$, since $\mathrm{det\,}\mf{L}(\l)$ is analytic in $\l\in\O$, either $\Sigma(\mathfrak{L})=\Omega$,
or $\Sigma(\mathfrak{L})$ is discrete. Thus, $\Sigma(\mf{L})$ consists of isolated generalized
eigenvalues if $\mf{L}(\mu)\in GL(\mathbb{K}^N)$ for some $\mu\in\O$. In such case, the \textit{algebraic multiplicity} of the curve $\mathfrak{L}\in\mc{C}^{\omega}(\Omega,\mathcal{L}(\mathbb{K}^{N}))$ at $\lambda_{0}$ is defined through
\begin{equation}
	\label{2.1}
	\mathfrak{m}_{\alg}[\mathfrak{L},\lambda_{0}]:=\ord_{\l=\lambda_{0}}\det\mathfrak{L}(\lambda).
\end{equation}
Although the multiplicity is defined for all $\l_0\in\R$, it equals zero if
$\l_0\in\R\setminus\Sigma(\mf{L})$. This concept extends the classical notion of algebraic multiplicity in linear algebra. Indeed,
if $\mathfrak{L}(\lambda)=\lambda I_{N}-T$ for some linear operator $T\in\mathcal{L}(\mathbb{K}^{N})$, then $\mathfrak{L}\in\mc{C}^{\omega}(\mathbb{K},\mathcal{L}(\mathbb{K}^{N}))$ and it is easily seen that $\mathfrak{m}_{\alg}[\mathfrak{L},\lambda_{0}]$ is well defined for all $\l_0\in\Sigma(\mf{L})$ and that
\eqref{2.1} holds. Note that, since $GL(\K^N)$ is open, $I_N-\l^{-1} T\in GL(\K^N)$ for sufficiently large $\l$. Thus, $\l I_N-T\in GL(\K^N)$ and $\Sigma(\mf{L})$ is discrete.
\par
This concept admits a natural (non-trivial) extension to an infinite-dimensional setting. To formalize it, we need to introduce some of notation. In this paper, for any given pair of $\mathbb{K}$-Banach spaces, say $U$ and $V$, we denote by $\Phi_0(U,V)$  the set of linear Fredholm operators of index zero between $U$ and $V$. Then, a \emph{Fredholm (continuous) path,} or curve, is any map $\mathfrak{L}\in \mathcal{C}(\Omega,\Phi_{0}(U,V))$.  Naturally, for any given $\mathfrak{L}\in \mathcal{C}(\Omega,\Phi_{0}(U,V))$, it is said that $\lambda\in\Omega$ is a \emph{generalized eigenvalue} of $\mathfrak{L}$ if $\mathfrak{L}(\lambda)\notin GL(U,V)$, and the \emph{generalized spectrum} of $\mathfrak{L}$, $\Sigma(\mathfrak{L})$,  is defined through   	
\begin{equation*}
	\Sigma(\mathfrak{L}):=\{\lambda\in\Omega: \mathfrak{L}(\lambda)\notin GL(U,V)\}.
\end{equation*}
The following concept, going back to \cite{LG01}, plays a pivotal role in the sequel.

\begin{definition}
	\label{de2.1}
	Let $\mathfrak{L}\in \mathcal{C}(\Omega, \Phi_{0}(U,V))$ and $\kappa\in\mathbb{N}$. A generalized eigenvalue $\lambda_{0}\in\Sigma(\mathfrak{L})$ is said to be $\kappa$-algebraic if there exists $\varepsilon>0$ such that
	\begin{enumerate}
		\item[{\rm (a)}] $\mathfrak{L}(\lambda)\in GL(U,V)$ if $0<|\lambda-\lambda_0|<\varepsilon$;
		\item[{\rm (b)}] there exists $C>0$ such that
		\begin{equation}
			\label{2.2}
			\|\mathfrak{L}^{-1}(\lambda)\|<\frac{C}{|\lambda-\lambda_{0}|^{\kappa}}\quad\hbox{if}\;\;
			0<|\lambda-\lambda_0|<\varepsilon;
		\end{equation}
		\item[{\rm (c)}] $\kappa$ is the minimal integer for which \eqref{2.2} holds.
	\end{enumerate}
\end{definition}
Throughout this paper, the set of $\kappa$-algebraic eigenvalues of $\mathfrak{L}$ is  denoted by $\Alg_\kappa(\mathfrak{L})$, and the set of \emph{algebraic eigenvalues} by $\Alg(\mathfrak{L}):=\cup_{\kappa\in\mathbb{N}}\Alg_\kappa(\mathfrak{L})$. 
As in the special case when $U=V=\K^N$, according to Theorems 4.4.1 and 4.4.4 of \cite{LG01}, when $\mathfrak{L}(\lambda)$ is analytic in $\Omega$, i.e., $\mathfrak{L}\in\mc{C}^{\omega}(\Omega, \Phi_{0}(U,V))$,  then, either $\Sigma(\mathfrak{L})=\Omega$,
or $\Sigma(\mathfrak{L})$ is discrete and $\Sigma(\mathfrak{L})\subset \Alg(\mathfrak{L})$.
Subsequently, we denote by $\mathcal{A}_{\lambda_{0}}(\Omega,\Phi_{0}(U,V))$ the set  of curves $\mathfrak{L}\in\mathcal{C}^{r}(\Omega,\Phi_{0}(U,V))$ such that $\lambda_{0}\in\Alg_{\kappa}(\mathfrak{L})$ with $1\leq \kappa \leq r$ for some $r\in\mathbb{N}$.
Next, we will construct an infinite dimensional analogue of the classical algebraic multiplicity $\mathfrak{m}_{\alg}[\mathfrak{L},\l_{0}]$ for the class  $\mathcal{A}_{\lambda_{0}}(\Omega,\Phi_{0}(U,V))$. It can be carried out through the theory of Esquinas and L\'{o}pez-G\'{o}mez
\cite{ELG},  where the following pivotal concept, generalizing the transversality condition of
Crandall and Rabinowitz \cite{CR},  was introduced. Throughout this paper, we set
$\mathfrak{L}_{j}:=\frac{1}{j!}\mathfrak{L}^{(j)}(\lambda_{0})$, $1\leq j\leq r$, should these derivatives exist.

\begin{definition}
	\label{de2.3}
	Let $\mathfrak{L}\in \mathcal{C}^{r}(\O,\Phi_{0}(U,V))$ and $1\leq \kappa \leq r$. Then, a given $\lambda_{0}\in \Sigma(\mathfrak{L})$ is said to be a $\kappa$-transversal eigenvalue of $\mathfrak{L}$ if
	\begin{equation*}
		\bigoplus_{j=1}^{\kappa}\mathfrak{L}_{j}\left(\bigcap_{i=0}^{j-1}N[\mathfrak{L}_{i}]\right)
		\oplus R[\mathfrak{L}_{0}]=V\;\; \hbox{with}\;\; \mathfrak{L}_{\kappa}\left(\bigcap_{i=0}^{\kappa-1}N[\mathfrak{L}_{i}]\right)\neq \{0\}.
	\end{equation*}
\end{definition}

For these eigenvalues, the algebraic multiplicity was introduced in \cite{ELG} by
\begin{equation}
	\label{ii.3}
	\chi[\mathfrak{L}, \lambda_{0}] :=\sum_{j=1}^{\kappa}j\cdot \dim \mathfrak{L}_{j}\left(\bigcap_{i=0}^{j-1}N[\mathfrak{L}_{i}]\right).
\end{equation}
In particular, when $N[\mf{L}_0]=\mathrm{span}[\v_0]$ for some $\v_0\in U$ such that $\mf{L}_1\v_0\notin R[\mf{L}_0]$, then
\begin{equation}
	\label{ii.4}
	\mf{L}_1(N[\mf{L}_0])\oplus R[\mf{L}_0]=V
\end{equation}
and hence, $\l_0$ is a 1-transversal eigenvalue of $\mf{L}(\l)$ with $\chi[\mf{L},\l_0]=1$. The transversality condition \eqref{ii.4} goes back to Crandall and Rabinowitz \cite{CR}. More generally, under condition \eqref{ii.4}, $\chi[\mf{L},\l_0]=\dim N[\mf{L}_0]$. According to Theorems 4.3.2 and 5.3.3 of \cite{LG01}, for every $\mathfrak{L}\in \mathcal{C}^{r}(\O, \Phi_{0}(U,V))$, $\kappa\in\{1,2,...,r\}$ and $\lambda_{0}\in \Alg_{\kappa}(\mathfrak{L})$, there exists a polynomial $\Phi: \O\to \mathcal{L}(U)$ with $\Phi(\lambda_{0})=I_{U}$ such that $\lambda_{0}$ is a $\kappa$-transversal eigenvalue of the path
\begin{equation}
	\label{ii.5}
	\mathfrak{L}^{\Phi}:=\mathfrak{L}\circ\Phi\in \mathcal{C}^{r}(\O, \Phi_{0}(U,V)),
\end{equation}
and $\chi[\mathfrak{L}^{\Phi},\lambda_{0}]$ is independent of the curve of \emph{trasversalizing local isomorphisms} $\Phi$ chosen to transversalize $\mathfrak{L}$ at $\lambda_0$ through \eqref{ii.5}. Therefore, the following concept of multiplicity
is consistent
\begin{equation}
	\label{ii.6}
	\chi[\mf{L},\l_0]:= \chi[\mathfrak{L}^{\Phi},\lambda_{0}],
\end{equation}
and it can be easily extended by setting
$\chi[\mathfrak{L},\lambda_0] =0$ if $\lambda_0\notin\Sigma(\mathfrak{L})$ and
$\chi[\mathfrak{L},\lambda_0] =+\infty$ if $\lambda_0\in \Sigma(\mathfrak{L})
\setminus \Alg(\mathfrak{L})$ and $r=+\infty$. Thus, $\chi[\mathfrak{L},\lambda]$ is well defined for all  $\lambda\in \O$ of any smooth path $\mathfrak{L}\in \mathcal{C}^{\infty}(\O,\Phi_{0}(U,V))$; in particular, for any analytical curve  $\mathfrak{L}\in\mc{C}^{\omega}(\O,\Phi_{0}(U,V))$.
\par The next uniqueness result, going back to Mora-Corral \cite{MC}, axiomatizes these concepts of algebraic multiplicity. Some refinements of them were delivered in \cite[Ch. 6]{LGMC}. In order to spate the result we need some preliminary definitions. Throughout this article, we denote by $\mathscr{C}^{\infty}_{\l_0}(U)$, the set of operator families $\mf{L}:\O_{\l_0}\to \Phi_{0}(U)$ of class $\mc{C}^{\infty}$ that are defined in a neighbourhood $\O_{\l_0}$ of $\l_0$, and introduce the space of \textit{germs} of smooth curves over $\l_0$ by $\mc{C}^{\infty}_{\l_0}(U):=\mathscr{C}^{\infty}_{\l_0}(U)/\sim$, 
where  we identify two families $\mf{L}_{1},\mf{L}_{2}\in\mathscr{C}^{\infty}_{\l_{0}}(U)$, $\mf{L}_{1}\sim\mf{L}_{2}$, if there exists a neighbourhood $\O_{\l_{0}}\subset\mathscr{D}(\mf{L}_{1})\cap\mathscr{D}(\mf{L}_{2})$ of $\l_{0}$ such that
$\mf{L}_{1}(\l)=\mf{L}_{2}(\l)$ for each $\l\in\O_{\l_{0}}$. Here, $\mathscr{D}(\mf{L})$ stands for the domain of the curve $\mf{L}$. Subsequently, given $\mathfrak{L}, \mathfrak{M} \in \mathcal{C}(\Omega, \Phi_{0}(U))$, we denote by $\mf{L}\mf{M}\in\mathcal{C}(\O,\Phi_{0}(U))$, the curve defined through $[\mf{L}\mf{M}](\l):=\mf{L}(\l)\circ\mf{M}(\l)$ for each $\l\in\O$.

\begin{theorem}
	\label{th24}
	Let $U$ be a non-trivial $\mathbb{K}$-Banach space and $\lambda_{0}\in\mathbb{K}$. Then, the algebraic multiplicity $\chi$ is the unique map 	$\chi[\cdot, \lambda_{0}]: \mc{C}^{\infty}_{\l_0}(U)\longrightarrow [0,\infty]$ such that
	\begin{enumerate}
		\item[{\rm (PF)}] For every pair $\mathfrak{L}, \mathfrak{P} \in \mc{C}^{\infty}_{\l_0}(U)$,
		$\chi[\mathfrak{L}\mathfrak{P}, \lambda_{0}]=\chi[\mathfrak{L},\lambda_{0}]+\chi[\mathfrak{P},\lambda_{0}]$. 
		\item[{\rm (NP)}] There exists a rank one projection $\Pi \in \mathcal{L}(U)$ such that
		\begin{equation*}
			\chi[(\lambda-\lambda_{0})\Pi +I_{U}-\Pi,\lambda_{0}]=1.
		\end{equation*}
	\end{enumerate}
\end{theorem}

The axiom (PF) is the  \emph{product formula} and (NP) is a \emph{normalization property}
for establishing the uniqueness of $\chi$. From these two axioms one can derive the remaining properties of  $\chi$; among them, that it equals the classical algebraic multiplicity when
$\mathfrak{L}(\lambda)= \lambda I_{U} - K$ for some compact operator $K$. Indeed, for every $\mathfrak{L}\in \mathcal{C}^{\infty}(\Omega,\Phi_{0}(U))$ and $\l_{0}\in\O$, the following properties are satisfied (see \cite{LGMC} for any further details):
\begin{itemize}
	\item $\chi[\mathfrak{L},\lambda_{0}]\in\mathbb{N}\uplus\{+\infty\}$;
	\item $\chi[\mathfrak{L},\lambda_{0}]=0$ if and only if $\mathfrak{L}(\lambda_0)
	\in GL(U)$;
	\item $\chi[\mathfrak{L},\lambda_{0}]<\infty$ if and only if $\lambda_0 \in\Alg(\mathfrak{L})$.
	\item If $U =\mathbb{K}^N$, then, in any basis, $		\chi[\mathfrak{L},\lambda_{0}]= \mathrm{ord}_{\lambda_{0}}\det \mathfrak{L}(\lambda)$.
	\item For every $K\in \mathcal{K}(U)$ and $\lambda_0\in \s(K)$,
	\begin{equation*}
		\label{1.1.90}
		\chi [\lambda I_U-K,\lambda_{0}]=\dim N[(\l_0 I_{U}-K)^{\nu(\l_0)}],
	\end{equation*}
	where $\nu(\l_0)$ is the \emph{algebraic ascent} of $\l_0$, i.e., the minimal integer, $\nu\geq 1$, such that
	\[
	N[(\l_0 I_{U}-K)^{\nu}]=N[(\l_0 I_{U}-K)^{\nu+1}].
	\]
\end{itemize}

\section{Topological degree for Fredholm operators}\label{STDFO}

\noindent A crucial feature that facilitates the construction of the Leray--Schauder degree is the fact that, for any real Banach space $U$, the space $GL_{c}(U)$ consists of two path-connected components, which fails to be true in the general context of Fredholm operators of index zero as a consequence of the Kuiper theorem \cite{K}.
Consequently, it is
not possible to introduce an orientation in $GL(U,V)$ for general real Banach spaces $U, V$, since in general, $GL(U,V)$ is path-connected. In 1991,  assuming that $(U,V)$ is a pair of real Banach spaces and $\mathfrak{L} :[a, b] \longrightarrow \Phi_0(U,V)$ is a continuous path of linear Fredholm operators of index zero with invertible endpoints, Fitzpatrick and Pejsachowicz \cite{FP1,FP2} introduced an homotopy invariant of $\mathfrak{L}$, the parity of $\mathfrak{L}$ on $[a,b]$, denoted in this paper  by $\sigma(\mathfrak{L},[a,b])$, which became a key technical device to overcome the difficulty of the orientation.
\par
In this section we begin by reviewing, very shortly, the concept of parity of an admissible curve of Fredholm operators and state some of its fundamental properties. Subsequently, it is said that a curve $\mf{L}\in\mc{C}([a,b],\Phi_{0}(U,V))$ is \textit{admissible} if $\mf{L}(a),\mf{L}(b)\in GL(U,V)$, and we denote by $\mathscr{C}([a,b],\Phi_{0}(U,V))$ the class of admissible curves. To define the parity via the Leray--Schauder degree it is necessary to introduce the concept of \textit{parametrix}. For any given $\mf{L}\in\mathscr{C}([a,b], \Phi_{0}(U,V))$, a parametrix of $\mf{L}$ is a family $\mf{P}\in\mc{C}([a,b], GL(V,U))$ such that
$$
  \mf{P}(\l)\mf{L}(\l)-I_{U}\in\mc{K}(U) \quad \hbox{for each}\;\; \l\in[a,b].
$$
The existence of a parametrix for every $\mf{L}\in\mathscr{C}([a,b], \Phi_{0}(U,V))$ is guaranteed by Theorem 2.1 of Fitzpatrick and Pejsachowicz \cite{FP2}. Then, for every $\mf{L}\in \mathscr{C}([a,b], \Phi_{0}(U,V))$, the \textit{parity} of the curve $\mf{L}$ is defined through
\begin{equation*}
		\sigma(\mf{L}, [a,b]):=\deg(\mf{P}(a)\mf{L}(a))\cdot \deg(\mf{P}(b)\mf{L}(b)),
\end{equation*}
where $\mf{P}\in \mc{C}([a,b], GL(V,U))$ is a parametrix of $\mf{L}$ and, for every $T\in GL_{c}(U)$, we are denoting $\deg(T):=\deg(T, B_{\varepsilon}(0))$, for sufficiently small $\varepsilon>0$, where $\deg$ is the Leray--Schauder degree.  This notion is consistent as it does not depend of the chosen parametrix.
\par
Throughout this paper, a homotopy $H:[0,1]\times [a,b]\to \Phi_{0}(U,V)$ is called \textit{admissible} if
$$
  H(\{a,b\}\times [0,1])\subset GL(U,V),
$$
and the class of admissible homotopies is denoted by $\mathscr{H}\equiv\mathscr{H}([a,b]\times [0,1],\Phi_{0}(U,V))$. Then, two admissible curves $\mf{L}_{1},\mf{L}_{2}\in\mathscr{C}([a,b],\Phi_{0}(U,V))$ are said to be $\mc{A}$-\textit{homotopic} if there exists an admissible homotopy $H\in\mathscr{H}$ such that $H(a,\cdot)=\mf{L}_1$ and $H(b,\cdot)=\mf{L}_2$.  The following result collects some properties of the parity after Fitzpatrick and Pejsachowicz \cite{FP2} that will used through this paper.
\begin{theorem}
\label{th3.1}
For every $\mf{L}\in\mathscr{C}([a,b],\Phi_{0}(U,V))$, the following properties hold:
\begin{itemize}
	\item \textbf{Stability:} If $\mf{L}(\lambda)\in GL(U,V)$ for all $\lambda \in [a,b]$, then $\sigma(\mf{L}, [a,b])=1$.

\item \textbf{Homotopy invariance:} If $\mf{L}_{1}, \mf{L}_{2} \in \mathscr{C}([a,b], \Phi_{0}(U,V))$ are $\mc{A}$-homotopic, then
\begin{equation*}
	\sigma(\mf{L}_{1},[a,b])=\sigma(\mf{L}_{2},[a,b]).
\end{equation*}
\item \textbf{Product formula:} For any tern of real Banach spaces,  $(U,V,W)$, and every   $\mf{L}_1\in \mathscr{C}([a,b],\Phi_{0}(U,V))$ and $\mf{L}_{2} \in \mathscr{C}([a,b],\Phi_{0}(V,W))$,
\begin{equation*}
	\sigma(\mf{L}_{2}\mf{L}_{1}, [a,b])=\sigma(\mf{L}_{2}, [a,b])\cdot\sigma(\mf{L}_{1},[a,b]).
\end{equation*}
\item \textbf{Additivity:} For any partition of the interval $[a,b]$, $[a,b]=\cup_{i=1}^{N}\  [m_{i-1},m_{i}]$, and every $\mf{L} \in \mathscr{C}([a,b], \Phi_{0}(U,V))$ admissible on $[m_{i-1},m_{i}]$ for each $1\leq i\leq N$,
\begin{equation*}
	\sigma(\mf{L}, [a,b])=\prod_{i=1}^{N}\sigma(\mf{L},[m_{i-1},m_{i}]).
\end{equation*}
\end{itemize}
\end{theorem}
\noindent Subsequently, for every $r\in\mathbb{N}\uplus\{\infty,\omega\}$,
we set
\[
\mathscr{C}^r([a,b],\Phi_{0}(U,V)) := \mathcal{C}^{r}([a,b],\Phi_{0}(U,V))\cap \mathscr{C}([a,b],\Phi_{0}(U,V)).
\]
\noindent The next result, proven by the authors in \cite{JJ}, shows how the parity of any admissible Fredholm path
$\mathfrak{L}\in\mathscr{C}([a,b], \Phi_{0}(U,V))$ can be computed though the algebraic multiplicity $\chi$.
\begin{theorem}
\label{th3.2}
	Any continuous admissible path $\mathfrak{L}\in\mathscr{C}([a,b],\Phi_{0}(U,V))$ is
	$\mathcal{A}$-homotopic to some analytic path  $\mathfrak{L}_{\omega}\in\mathscr{C}^{\omega}([a,b],\Phi_{0}(U,V))$. Moreover, for any of these paths,
\begin{equation*}
		\sigma(\mathfrak{L},[a,b])=(-1)^{\sum_{i=1}^{n}\chi[\mathfrak{L}_{\omega},\lambda_{i}]},
\end{equation*}
where $\Sigma(\mathfrak{L}_{\omega})=\{\lambda_{1},\lambda_{2},...,\lambda_{n}\}$.
\end{theorem}

As the main trouble to introduce a topological degree for Fredholm operators of index zero is the absence of orientation in $GL(U,V)\subset\Phi_{0}(U,V)$, the notion introduced in the next definition, going back to Fitzpatrick, Pejsachowicz and Rabier \cite{FPRa}, restricts the admissible maps for which the degree is defined to the ones where is possible to introduce a notion of \emph{orientability}. Let $X$ be a path-connected topological space and $h:X\to\Phi_{0}(U,V)$ a continuous function. A point $x\in X$ is said to be \emph{regular with respect to $h$} if $h(x)\in GL(U,V)$. Subsequently, the set of regular points with respect to $h$ will be denoted by $\mathcal{R}_{h}$.

\begin{definition}
\label{de3.3}
Let $X$ be a path-connected topological space, and consider a pair $(U,V)$ of real Banach spaces. A continuous map $h:X\to\Phi_{0}(U,V)$ is said to be orientable if there exists a function $\varepsilon: \mathcal{R}_{h}\to\mathbb{Z}_{2}$, called orientation, such that, for every continuous curve $\gamma\in\mathcal{C}([a,b],X)$ with $\gamma(a), \gamma(b)\in \mathcal{R}_{h}$,
	\begin{equation}
		\label{4.4.200}
		\sigma(h\circ \gamma,[a,b])=\varepsilon(\gamma(a))\cdot\varepsilon(\gamma(b)).
	\end{equation}
When $X$ is not path-connected, a map $h:X\to\Phi_{0}(U,V)$ is said to be orientable if it is orientable on each path-connected component of $X$.
\end{definition}

\par
Now, we will collect some important features going back to the seminal paper of Fitzpatrick, Pejsachowicz and Rabier \cite{FPRa}. If $h:X\to\Phi_{0}(U,V)$ is an orientable map with orientation $\varepsilon:\mc{R}_{h}\to\Z_{2}$, then $\varepsilon$ is constant on each path connected component of $\mc{R}_{h}$. Moreover, for every continuous map $h:X\to\Phi_{0}(U,V)$, the following three assertions are equivalent:
	\begin{enumerate}
		\item $h $ is orientable.
		\item For each  $\gamma\in \mc{C}([a,b],X)$ such that $\gamma(a),\gamma(b)\in\mc{R}_{h}$, the parity $\sigma(h\circ\gamma,[a,b])$ only depends on $\{\gamma(a),\gamma(b)\}$.
		\item $\sigma(h\circ\gamma,[a,b])=1$ for every closed path $\gamma\in \mc{C}([a,b],X)$  such that $\gamma(a)=\gamma(b)\in \mathcal{R}_h$.
	\end{enumerate}
When $h:X\to\Phi_{0}(U,V)$ is orientable and $\mc{R}_{h}\neq \emptyset$, then there are, exactly, two different orientations for $h$. Namely, for any given $p\in\mc{R}_{h}$, these orientations are defined by
\begin{equation}
	\label{2O}
	\varepsilon^{\pm}:\mc{R}_{h} \longrightarrow \mathbb{Z}_{2}, \quad \varepsilon^{\pm}(q):=\pm\sigma(h\circ \gamma_{pq},[a,b]),
\end{equation}
where $\gamma_{pq}\in \mc{C}([a,b],X)$ is an arbitrary path such that $\g(a)=p$ and $\g(b)=q$, and the sign $\pm$ determines the orientation of $p$, in the sense that $\varepsilon^{+}(p)=1$ and $\varepsilon^{-}(p)=-1$. Thus, \eqref{2O} can be expressed as
\begin{equation}
	\label{2O2}
	\varepsilon^{\pm}(q)=\varepsilon^{\pm}(p)\cdot \sigma(h\circ \gamma_{pq},[a,b]), \quad q\in\mc{R}_{h}.
\end{equation}

\subsection{Topological Degree for Fredholm Operators}
The main goal of this section is to introduce the degree for Fredholm operators that we are going to use in this paper. The best way to do it is through its axiomatization theorem. To state this fundamental result we need to introduce some previous notations and terminologies.
Let $(U,V)$ be a pair of real Banach spaces. For any open subset, $\mc{O}\subset U$ and integers $n\geq 0$, $r\geq 1$, an operator $f:\mc{O} \to V$ is said to be $\mathcal{C}^{r}$-\emph{Fredholm of index n} if $f\in \mathcal{C}^{r}(\mc{O},V)$ and $Df\in \mathcal{C}^{r-1}(\mc{O},\Phi_{n}(U,V))$. The set of all these operators is denoted in this paper by $\mathscr{F}^{r}_{n}(\mc{O},V)$. An operator $f\in\mathscr{F}^{r}_{0}(\mc{O},V)$ is said to be \emph{orientable} if  $Df:\mc{O}\to\Phi_{0}(U,V)$ is an orientable map, as discussed in Definition \ref{de3.3}. Moreover, for any open and bounded set $\Omega$ such that $\bar{\O}\subset\mc{O}\subset U$, and any operator $f:\mc{O} \to V$ satisfying 	
\begin{enumerate}
	\item $f\in \mathscr{F}^{1}_{0}(\mc{O},V)$ is \emph{orientable} with orientation $\varepsilon:\mathcal{R}_{Df}\to\mathbb{Z}_{2}$,
	\item $f$ is \emph{proper} on $\bar{\O}$,
	\item $0\notin f(\partial \Omega)$,
\end{enumerate}
it is said that $(f,\Omega,\varepsilon)$ is a \emph{Fredholm $\mc{O}$-admissible triple}. Subsequently, the class of Fredholm $\mc{O}$-admissible triples is denoted by $\mathscr{A}(\mc{O})$.
\par
Given a $\mc{C}^{r}$-Fredholm map $f:\mc{O}\subset U\to V$, a point $u\in \mc{O}$ is said to be a \textit{regular point} of $f$ if $Df(u)\in\mc{L}(U,V)$ is surjective, i.e., $R[Df(u)]=V$. Thanks to the open mapping theorem, if $f\in\mathscr{F}^{r}_{0}(\mc{O},V)$, $u\in \mc{O}$ is a regular point of $f$ if and only if $Df(u)\in GL(U,V)$. Naturally, $\mc{R}_{Df}$ stands for the set of regular points of $f$. On the other hand, for any given open or closed subset of $\mc{O}$, $\mathscr{O}\subset \mc{O}$, a point $v\in V$ is said to be  a \textit{regular value} of $f:\mathscr{O}\to V$ if $f^{-1}(v)\cap \mathscr{O}$ is empty or it consists on regular points, i.e.,  $Df(u)\in \mc{L}(U,V)$ is surjective for each $u\in f^{-1}(v)\cap \mathscr{O}$. In this paper, the set of regular values of $f:\mathscr{O}\to V$ is denoted by $\mc{RV}_{f}(\mathscr{O})$. By definition, the regular points and regular values of $f:\mathscr{O}\to V$ are related via the set identity
$$
  \mc{RV}_{f}(\mathscr{O})=V\backslash f(\mathscr{O}\backslash \mc{R}_{Df}).
$$
Obviously, $ \mc{RV}_{f}(\mc{O})\subset \mc{RV}_{f}(\mathscr{O}_{2})\subset\mc{RV}_{f}(\mathscr{O}_{1})$ for any open or closed subsets $\mathscr{O}_{1},\mathscr{O}_{2}\subset \mc{O}$ such that $\mathscr{O}_{1}\subset \mathscr{O}_{2}$.
Given  an open and bounded subset of $\mc{O}$, say $\O$, for the construction of the degree, we are mainly interested in regular values of the restriction map $f:\bar{\O}\to V$, i.e., in the set, $\mc{RV}_{f}(\bar{\O})$. Note that if $(f,\O,\varepsilon)\in\mathscr{A}(\mc{O})$ and $v\in \mc{RV}_{f}(\bar{\O})$, then $f^{-1}(v)\cap \bar{\O}$ is finite, possibly empty.
Given $(f,\Omega,\varepsilon)\in \mathscr{A}(\mc{O})$, it is said that $(f,\Omega,\varepsilon)$ is a \textit{$\mc{O}$-regular triple} if $0\in\mc{RV}_{f}(\O)$, i.e., if
$Df(x)\in GL(U,V)$ for all $x \in f^{-1}(0)\cap \O$. The set of regular triples is denoted by $\mathscr{R}(\mc{O})$.
\par
It is said that a map $H\in \mathcal{C}^{r}([0,1]\times \mc{O},V)$ is a $\mathcal{C}^{r}$-\textit{Fredholm homotopy} if $H\in\mathscr{F}^{r}_{1}([0,1]\times \mc{O}, V)$, i.e., if
\begin{equation*}
	D_{u}H(t,u)\in \Phi_{0}(U,V) \text{ for all } (t,u)\in [0,1]\times \mc{O}.
\end{equation*}
A $\mc{C}^{r}$-Fredholm homotopy $H\in\mathscr{F}^{r}_{1}([0,1]\times \mc{O},V)$ is called \textit{orientable} if $D_{u}H:[0,1]\times\mc{O}\to\Phi_{0}(U,V)$ is an orientable map. In such case, we denote by  $\varepsilon_{t}$ the restriction
\begin{equation}
	\label{Orient2}
	\varepsilon_{t}: \mathcal{R}_{DH_{t}}\longrightarrow\mathbb{Z}_{2}, \quad \varepsilon_{t}(x):=\varepsilon(t,x),
\end{equation}
for every $t\in[0,1]$. Given $H\in\mathscr{F}^{r}_{1}([0,1]\times\mc{O},V)$, the following statements hold:

\begin{enumerate}
	\item If $H$ is orientable with orientation $\varepsilon$,  then for every $t\in[0,1]$, the $t$-section $H_{t}\in \mathscr{F}_{0}^{r}(\mc{O}, V)$ is orientable with the orientation $\varepsilon_{t}$ defined in \eqref{Orient2}.
	
	\item If for some $t_{0}\in[0,1]$, the section $H_{t_0}\in\mathscr{F}_{0}^{r}(\mc{O},V)$ is non-degenerate and orientable, then $H$ is orientable. Furthermore, any orientation $\varepsilon_{t_{0}}$ of $H_{t_{0}}$ can be extended as an orientation $\varepsilon$ of $H$.
\end{enumerate}

We are ready to introduce  the class of $\mc{O}$-admissible homotopies. For any open and bounded subset $\O\subset U$ such that $\bar{\O}\subset \mc{O}$, it is said that $(H,\O,\varepsilon)$ is a \textit{Fredholm $\mc{O}$-admissible homotopy} if the following conditions are satisfied:
\begin{enumerate}
	\item $H\in\mathscr{F}^{1}_{1}([0,1]\times \mc{O},V)$ is \textit{orientable} with orientation $\varepsilon:\mc{R}_{D_u H}\to\Z_{2}$,
	\item $H$ is proper on $[0,1]\times \bar{\O}$,
	\item $0\notin H([0,1]\times \partial\O)$,
\end{enumerate}
The class of $\mc{O}$-admissible homotopies is denoted by $\mathscr{H}(\mc{O})$.
\par
Finally, the \textit{admissible class} is given by the set $\mc{A}:=\mathscr{A}/\sim$, where
$$
  \mathscr{A}:=\bigcup\{\mathscr{A}(\mc{O}): \mc{O}\subset U \ \text{open subset}\}
$$
is the class of all $\mc{O}$-admissible triples, and the binary relation $\sim$ relates two triples $(f_i,\O_i,\varepsilon_i)\in\mathscr{A}(\mc{O}_{i})$, $i\in\{1,2\}$, whenever:
\begin{enumerate}
	\item $\O_{1}=\O_{2}\equiv \O$.
	\item $f_{1}(u)=f_{2}(u)$ for each $u\in\bar{\O}$.
	\item $\varepsilon_{1}(u)=\varepsilon_{2}(u)$ for each $u\in\mc{R}_{Df_1}\cap\O=\mc{R}_{Df_2}\cap\O$.
\end{enumerate}
Once introduced these notations, we can estate the next axiomatization of the topological degree for Fredholm operators.

\begin{theorem}[\textbf{Axiomatization of the degree}]
	\label{thiii.1}
	There exists a unique integer valued map $\deg: \mc{A}\to \mathbb{Z}$
	satisfying the next three properties:
	\begin{enumerate}
		\item[{\rm (N)}] \textbf{Normalization:} For every $L\in GL(U,V)$ with orientation $\varepsilon$ and each open and bounded subset $\O\subset U$ such that $0\in\O$, one has that
$$
    \deg(L,\Omega,\varepsilon)=\varepsilon(0).
$$

\item[{\rm (A)}] \textbf{Additivity:} For every $(f,\Omega,\varepsilon)\in\mc{A}$  and any
		pair of disjoint open subsets $\Omega_{1}$ and $\Omega_{2}$ of $\Omega$ with $0\notin f(\Omega\backslash (\Omega_{1}\uplus \Omega_{2}))$,
\begin{equation}
			\label{AddPr}
			\deg(f,\Omega,\varepsilon)=\deg(f,\Omega_{1},\varepsilon)+\deg(f,\Omega_{2},\varepsilon).
\end{equation}

\item[{\rm (H)}] \textbf{Homotopy Invariance:} For every open subset $\mc{O}\subset U$ and each $\mc{O}$-admissible homotopy $(H,\O,\varepsilon)\in\mathscr{H}(\mc{O})$, we have that
\begin{equation}
			\label{HomPr2}
			\deg(H(0,\cdot),\Omega,\varepsilon_{0})=\deg(H(1,\cdot),\Omega,\varepsilon_{1}).
\end{equation}
\end{enumerate}
Moreover, for every open subset $\mc{O}\subset U$ and $(f,\Omega,\varepsilon)\in\mathscr{R}(\mc{O})$, with $\O$ connected and such that $\mathcal{R}_{Df}\neq\emptyset$, one has that, for every $p\in\mathcal{R}_{Df}$,
	\begin{equation}
		\label{5.5.12T}
		\deg(f,\Omega,\varepsilon)=\varepsilon(p)\cdot \sum_{u\in f^{-1}(0)\cap \Omega} (-1)^{\chi[\mathfrak{L}_{\omega,u},[a,b]]}
	\end{equation}
	where $\mathfrak{L}_{\omega,u}\in \mathscr{C}^\omega([a,b],\Phi_{0}(U,V))$ is any analytic curve   $\mc{A}$-homotopic to $Df\circ\gamma$, for some $\gamma\in\mathcal{C}([a,b],\Omega)$ such that
	$\gamma(a)=p$, $\gamma(b)=u$, and
	\begin{equation*}	
		\chi[\mathfrak{L}_{\omega,u},[a,b]]:=\sum_{\lambda\in
			\Sigma(\mathfrak{L}_{\omega,u})\cap[a,b]}\chi[\mathfrak{L}_{\omega,u},\lambda].
	\end{equation*}
\end{theorem}

The existence goes back to Fitzpatrick, Pejsachowicz and Rabier \cite{FPRb} for
$\mathcal{C}^{2}$ mappings, and to Pejsachowicz and Rabier \cite{PR} in the $\mathcal{C}^{1}$ setting. The uniqueness and the generalized Schauder formula \eqref{5.5.12T} were established by the authors in \cite{JJ2} and \cite{JJ}, respectively. Naturally, from the axioms (N), (A) and (H) one can readily get the most basic properties of the degree, as its excision and fundamental properties. For the purposes of this paper,  it is appropriate to sketch, very briefly, the construction of the degree carried over in \cite{FPRb, PR}. Let us start by defining the degree for regular triples. For every $(f,\O,\varepsilon)\in\mathscr{R}(\mc{O})$, by definition, $f\in\mathscr{F}^{1}_{0}(\mc{O},V)$ is a $\mathcal{C}^{1}$-Fredhom map of index zero and it is $\varepsilon$-orientable, i.e., $Df:\mc{O}\to\Phi_{0}(U,V)$ is an orientable map with orientation
$\varepsilon:\mathcal{R}_{Df} \to \mathbb{Z}_{2}$. Since $0\in\mc{RV}_{f}(\bar{\O})$, $f^{-1}(0)\cap \bar{\O}=f^{-1}(0)\cap \O$ is finite, possibly empty. Thus, if $f^{-1}(0)\cap \O\neq \emptyset$, we can define
\begin{equation}
	\label{5.5.13}
	\deg(f,\Omega,\varepsilon):=\sum_{u\in f^{-1}(0)\cap\O} \varepsilon(u),
\end{equation}
while we set $\mathrm{deg}(f,\Omega,\varepsilon):=0$ if $f^{-1}(0)\cap\O=\emptyset$. Note that, whenever $u\in f^{-1}(0)\cap \O$, $u\in\mc{R}_{Df}$ since $0$ is a regular value. Hence,  $\varepsilon(u)$ is well defined. If $0\notin \mc{RV}_{f}(\O)$, then we define
$$
   \deg(f,\O,\varepsilon):=\deg(f-v,\O,\varepsilon),
$$
where $v\in V$ is any regular value of $f:\O\to V$ lying in a sufficiently small neighbourhood of $0$ in $V$. Since $Df=D(f-v)$, the orientation map $\varepsilon$ is the same for both maps $f$ and $f-v$. The existence of the regular value is guaranteed by the Quinn--Sard--Smale theorem, \cite{Sa,Sm,QS}.

\subsection{Perturbation Theorems}

The introduction of an orientation associated to each particular  map $f$ might cause  some troubles in applications when dealing with the homotopy invariance of the degree, because this property relays on the particular global orientation chosen in the axiom (H) of Theorem \ref{thiii.1}. To precise what we mean, let $(H,\O,\varepsilon)\in\mathscr{H}(\mc{O})$ be a $\mc{O}$-admissible homotopy and suppose that  $\gamma\in\mathcal{C}([a,b],[0,1]\times\Omega)$ is a path such that $\gamma(a)=(0,p_0)$ and $\gamma(b)=(1,p_{1})$, where $p_{t}$ is a regular point of $D_{u}H(t,\cdot)$ for each $t\in\{0,1\}$. By \eqref{2O}, since  $\mathcal{R}_{DH_{t}}\neq \emptyset$ for each $t\in\{0,1\}$, there are two different orientations of $H_{t}\equiv H(t,\cdot)$. Thus, for any given regular point $(t,p_{t})$ of $H_{t}$, there exists a unique orientation such that $\varepsilon(t,p_{t})=1$, while  the other satisfies $\varepsilon(t,p_{t})=-1$. Let $\varepsilon_{p_t}$ denote the unique orientation of $H_t$ with $\varepsilon_{p_t}(p_{t})=1$, $t\in\{0,1\}$. Then, by the homotopy invariance of the degree, we have that
\begin{equation}
	\label{6.6.4} \deg(H_{0},\Omega,\varepsilon_{0})=\deg(H_{1},\Omega,\varepsilon_{1}).
\end{equation}
The formula \eqref{6.6.4} establishes the degree invariance under admissible  homotopies by choosing in $H_{t}$ the (global) orientation $\varepsilon_{t}:=\varepsilon(t,\cdot)$, $t\in [0,1]$. Thus, if, for example, we have that $\varepsilon_{0}=\varepsilon_{p_{0}}$ and $\varepsilon_{1}=-\varepsilon_{p_{1}}$, then
\begin{equation*}
	\deg(H_{0},\Omega,\varepsilon_{0})=\deg(H_{1},\Omega,\varepsilon_{1}),
\end{equation*}
though, paradoxically,
\begin{equation*}
	\deg(H_{0},\Omega,\varepsilon_{p_{0}})=-\deg(H_{1},\Omega,\varepsilon_{p_{1}}).
\end{equation*}
Therefore, in dealing with the homotopy invariance of the degree for Fredholm operators,  one should be extremely careful with the eventual changes of sign of the degree caused by the changes of
orientation, even when using its invariance by homotopy. Such a rather subtle problematic, outside the  Leray--Schauder degree, arises in the context of the degree for Fredholm operators by the absence of a global orientation in $GL(U,V)$. Nevertheless, one can easily get rid of this ambiguity by using the parity or, equivalently, the generalized algebraic multiplicity. Indeed, the next result holds.

\begin{lemma}
	\label{leiii.5}
	Let $(H,\O,\varepsilon)\in\mathscr{H}(\mc{O})$ be a $\mc{O}$-admissible homotopy with $\O$ connected and $p_{0}\in\mc{R}_{DH_{0}}$, $p_{1}\in\mc{R}_{DH_{1}}$. Then, for every path $\gamma\in\mc{C}([a,b],[0,1]\times\O)$ such that $\g(a)=(0,p_{0})$ and $\g(b)=(1,p_{1})$,
\begin{equation}
\label{HomPA}
		\deg(H_{0},\Omega,\varepsilon_{p_0})=\sigma(D_{u}H\circ \gamma,[a,b])
		\deg(H_{1},\Omega,\varepsilon_{p_1}).
\end{equation}
Therefore, thanks to Theorem \ref{th3.2},
\begin{equation}
		\label{HomMA}
		\deg(H_{0},\Omega,\varepsilon_{p_0})=(-1)^{\chi[\mathfrak{L}_{\omega},[a,b]]}
		\deg(H_{1},\Omega,\varepsilon_{p_1}),
\end{equation}
	where $\mathfrak{L}_{\omega}\in\mathscr{C}^\omega([a,b],\Phi_{0}(U,V))$ is any analytic map $\mc{A}$-homotopic to $D_{u}H\circ\gamma$.
\end{lemma}

\begin{proof}
	Since $p_{0}\in\mc{R}_{DH_{0}}$, by \eqref{2O}, $H_{0}$ admits two orientations, $\varepsilon_{p_{0}}$ and $-\varepsilon_{p_{0}}$. Consequently, either $\varepsilon_{0}=\varepsilon_{p_{0}}$, or $\varepsilon_{0}=-\varepsilon_{p_{0}}$. In either case, 
$$
  \varepsilon_{0}(q)=\varepsilon_{0}(p_{0})\cdot \varepsilon_{p_{0}}(q), \quad q\in\mc{R}_{DH_{0}},
$$
because $\varepsilon_{0}(p_{0})=1$ if $\varepsilon_{0}=\varepsilon_{p_{0}}$ and
$\varepsilon_{0}(p_{0})=-1$ if $\varepsilon_{0}=-\varepsilon_{p_{0}}$. Thus, by the definition
of the degree,
$$
   \deg(H_{0},\O,\varepsilon_{0})=\varepsilon_{0}(p_{0})\deg(H_{0},\O,\varepsilon_{p_{0}}).
$$
Thus, multiplying by $\varepsilon_{0}(p_{0})$ yields to
\begin{equation}
		\label{IdJ1}
		\deg(H_{0},\O,\varepsilon_{p_{0}})=\varepsilon_{0}(p_{0}) \deg(H_{0},\O,\varepsilon_{0}).
\end{equation}
Similarly, inter-exchanging $p_0$ by $p_1$ shows that
\begin{equation}
		\label{IdJ2}
		\deg(H_{1},\O,\varepsilon_{1})= \varepsilon_{1}(p_{1})\deg(H_{1},\O,\varepsilon_{p_{1}}).
\end{equation}
Combining the identities \eqref{IdJ1} and \eqref{IdJ2} with the invariance by homotopy of the degree, we find that
\begin{align*}
		\deg(H_{0},\O,\varepsilon_{p_{0}})&=\varepsilon_{0}(p_{0}) \deg(H_{0},\O,\varepsilon_{0}) \\ & = \varepsilon_{0}(p_{0}) \deg(H_{1},\O,\varepsilon_{1}) =\varepsilon_{0}(p_{0})\varepsilon_{1}(p_{1})\deg(H_{1},\O,\varepsilon_{p_{1}}) .
\end{align*}
On the other hand, thanks to \eqref{4.4.200}, it becomes apparent that, for every $\gamma\in\mc{C}([a,b],[0,1]\times\O)$ such that $\g(a)=(0,p_{0})$ and $\g(b)=(1,p_{1})$,
\begin{equation}
		\label{EqPat}
		\varepsilon_{0}(p_0)\varepsilon_{1}(p_{1})=\sigma(D_{u}H\circ \gamma,[a,b]).
\end{equation}
Therefore, we obtain that
\begin{equation*}
		\deg(H_{0},\Omega,\varepsilon_{p_0})=\sigma(D_{u}H\circ \gamma,[a,b])
		\deg(H_{1},\Omega,\varepsilon_{p_1}),
\end{equation*}
which is \eqref{HomPA}. Finally, owing to Theorem \ref{th3.2},
\begin{equation}
		\label{6.6.5}
		\sigma(D_{u}H\circ\gamma,[a,b])=(-1)^{\chi[\mathfrak{L}_{\omega},[a,b]]},
\end{equation}
	where $\mathfrak{L}_{\omega}\in\mathscr{C}^\omega([a,b],\Phi_{0}(U,V))$ is any analytic map $\mc{A}$-homotopic to $D_{u}H\circ\gamma$. Consequently, inserting \eqref{6.6.5} into \eqref{HomPA} yields to 	 \begin{equation}
		\label{6.6.6}
		\deg(H_{0},\Omega,\varepsilon_{p_0})=(-1)^{\chi[\mathfrak{L}_{\omega},[a,b]]}
		\deg(H_{1},\Omega,\varepsilon_{p_1}).
\end{equation}
This proves \eqref{HomMA} and ends the proof.
\end{proof}

\subsection{Generalized homotopy invariance}

This section collects and proves a generalized homotopy invariance property that is pivotal throughout this paper. Although it goes back to Fitzpatrick, Pejsachowicz and Rabier \cite{FPRa},
our proof polishes substantially the original one and it is adapted to the notations of this paper.
Subsequently, for any given subset $\O$ of $\mathbb{R}\times U$ and every $t\in \mathbb{R}$, we set
$$
\O_t:=\{u\in U\,:\;(t,u)\in\O\}.
$$
For any bounded open and connected subset $\O$ of $[0,1]\times U$, any open subset $\mc{O}$ of $U$ such that $\overline{\O}\subset [0,1]\times \mc{O}$ and any continuous map $H:[0,1]\times \mc{O}\to V$, it is said that $(H,\O,\varepsilon)$ is a \textit{generalized Fredholm $\mc{O}$-admissible homotopy} if the following conditions are satisfied:
\begin{enumerate}
	\item $H\in\mathscr{F}^{1}_{1}([0,1]\times \mc{O},V)$ is \textit{orientable} with orientation $\varepsilon:\mc{R}_{D_u H}\to\Z_{2}$,
	\item $H$ is proper on $\overline{\O}$,
	\item $0\notin H(\partial\O)$,
\end{enumerate}
The class of generalized $\mc{O}$-admissible homotopies is denoted by $\mathscr{G}(\mc{O})$ in this paper.

\begin{theorem}
	\label{T6.1.1}
Let $(H,\O,\varepsilon)\in \mathscr{G}(\mc{O})$ be a generalized $\mc{O}$-admissible homotopy and suppose that  $p_{0}\in\mc{R}_{DH_{0}}$ and $p_{1}\in\mc{R}_{DH_{1}}$. Then, for every path  $\gamma\in\mc{C}([a,b],[0,1]\times \mc{O})$ such that $\g(a)=(0,p_{0})$ and $\g(b)=(1,p_{1})$,
\begin{equation}
		\label{HomPA2}
		\deg(H_{0},\Omega_{0},\varepsilon_{p_0})=\sigma(D_{u}H\circ \gamma,[a,b])
		\deg(H_{1},\Omega_{1},\varepsilon_{p_1}).
\end{equation}
Therefore, by Theorem \ref{th3.2},
	\begin{equation}
		\label{6.6.7} \deg(H_{0},\O_{0},\varepsilon_{p_{0}})=(-1)^{\chi[\mathfrak{L}_{\omega},[a,b]]}\deg(H_{1},\O_{1},\varepsilon_{p_{1}}),
	\end{equation}
	where $\mathfrak{L}_{\omega}\in\mathscr{C}^\omega([a,b],\Phi_{0}(U,V))$ is any analytic map $\mc{A}$-homotopic to $D_{u}H\circ\gamma$.
\end{theorem}

\begin{proof}
The following concept of absolute degree goes back to \cite{FPRb},
\begin{equation*}
		\deg(f,\Omega):=|\deg(f,\Omega,\varepsilon)| \quad \hbox{for every}\;\; (f,\Omega,\varepsilon)\in\mathscr{A}(\mc{O}).
\end{equation*}
This degree is $\mathbb{N}$-valued and satisfies the excision property and the generalized invariance by homotopy. Indeed, suppose that $\mathcal{V}$ is an open subset of $\Omega$ such that $0\notin f(\Omega\backslash\mathcal{V})$. Then, by the additivity property of the topological degree, we have that
\begin{equation*}
		\deg(f,\Omega,\varepsilon)=\deg(f,\mathcal{V},\varepsilon).
\end{equation*}
Thus, taking absolute values yields	$\deg(f,\Omega)=\deg(f,\mathcal{V})$ and hence, the absolute degree satisfies the excision property.	Similarly, since the absolute degree is invariant by orientations, the proof of the generalized invariance by homotopy of the Leray--Schauder degree shows that, for every $(H,\O,\varepsilon)\in \mathscr{G}(\mc{O})$, we have that
\begin{equation*}
		\deg(H_{0},\O_{0})=\deg(H_{t},\O_{t})\quad \hbox{for all}\;\; t \in [0,1].
\end{equation*}
Therefore, the absolute degree also satisfies the generalized invariance by homotopy.
\par
Suppose that $\deg(H_{0},\O_{0},\varepsilon_{p_{0}})=0$. Then, since the absolute degree is invariant by (admissible) homotopies, we have that
\begin{equation*}
		0=\deg(H_{0},\O_{0})=\deg(H_{t},\O_{t}) \quad
		\hbox{for all}\;\;  t\in[0,1].
\end{equation*}
In particular, $\deg(H_{1},\O_{1})=0$ and hence, $\deg(H_{1},\O_{1},\varepsilon_{p_{1}})=0$. Therefore, \eqref{HomPA2} and \eqref{6.6.7} hold.
	\par
Subsequently, we suppose that $\deg(H_{0},\O_{0},\varepsilon_{p_{0}})=d\neq 0$. Then,
\begin{equation*}
		\deg(H_{t},\O_{t})=\deg(H_{0},\O_{0})=|d|\neq 0, \quad t\in[0,1].
\end{equation*}
This entails that, for every $t\in[0,1]$, there exists  a regular point $p_{t}\in\mathcal{R}_{DH_{t}}$, because, by definition, $\deg(f,\Omega,\varepsilon)=0$ if $(f,\Omega,\varepsilon)\in\mathscr{A}(\mc{O})$ does not admit regular points. In particular,  $\mathcal{R}_{DH_{t}}\neq\emptyset$ for all $t\in[0,1]$.
	We claim that, for every $t_{0}\in[0,1]$ and $p_{t_{0}}\in\mc{R}_{DH_{t_{0}}}$,  there exists $\varepsilon>0$ such that:
\begin{enumerate}
		\item[{\rm (i)}] $p_{t_{0}}\in\mathcal{R}_{DH_{t}}$ for all $t\in[t_{0}-\varepsilon,t_{0}+\varepsilon]$, and
		\item[{\rm (ii)}] $H^{-1}_{t}(0)\cap \O_{t}=H^{-1}_{t}(0)\cap \O_{t_{0}}$ for all $t\in[t_{0}-\varepsilon,t_{0}+\varepsilon]$,
\end{enumerate}
making the necessary changes in these statements when $t_{0}=0$ or $t_{0}=1$. Indeed, since
$$
	\mathcal{R}_{D_{u}H}=D_{u}H^{-1}(GL(U,V))
$$
is an open subset of $[0,1]\times \mathcal{O}$ and  $(t_{0},p_{t_{0}})\in\mathcal{R}_{D_{u}H}$, there exists $\varepsilon>0$ such that $(t,p_{t_{0}})\in \mathcal{R}_{D_{u}H}$ for all  $t\in[t_{0}-\varepsilon,t_{0}+\varepsilon]$. Thus, the property (i) follows from the set identity
	$$\mc{R}_{D_{u}H}=\bigcup_{t\in[0,1]}\{t\}\times \mc{R}_{DH_{t}}.$$
	Suppose that Property (ii) fails for all $\e>0$. Then, there exists a sequence $\{(t_{n},u_{n})\}_{n\in\mathbb{N}}\subset H^{-1}(0)\cap \O$ such that $\lim_{n\to\infty}t_{n}=t_{0}$ and $u_{n}\in\O_{t_{n}}\backslash\O_{t_{0}}$ for all $n\geq 1$. Since $H^{-1}(0)\cap \overline{\O}$ is compact, without loss of generality, we can assume that
	$$
	\lim_{n\to\infty}(t_{n},u_{n})=(t_{0},u_{0})\in H^{-1}(0)\cap \overline{\O}.
	$$
	Since $0\notin H(\partial\O)$, necessarily $(t_{0},u_{0})\in H^{-1}(0)\cap \O$. Therefore, $u_{0}\in H_{t_{0}}^{-1}(0)$ and $u_{0}\in \O_{t_{0}}$. In particular $\O_{t_{0}}\neq \emptyset$. But this contradicts the fact that $u_{n}\in\O_{t_{n}}\backslash\O_{t_{0}}$ for all  $n\geq 1$. So, Property (ii) also holds.
	\par
	Combining (i) and (ii) with the compactness of $[0,1]$, for some integer $m\in\mathbb{N}$, setting
	$t_{i}:=\frac{i}{m}$, $0\leq i \leq m$, there exists some $q_{i}\in\mc{O}$ which is a regular point of $DH_{t}$ for all $t\in[t_{i},t_{i+1}]$ and $0\leq i\leq m-1$. Choose $q_{0}=p_0$ and $q_{m-1}=p_1$. Since $H^{-1}_{t}(0)\cap \O_{t}=H^{-1}_{t}(0)\cap \O_{t_{i}}$ for all $t\in[t_{i},t_{i+1}]$, it follows from the excision property that
	\begin{equation*}	 \deg(H_{t},\O_{t},\varepsilon_{q_{i}})=\deg(H_{t},\O_{t_{i}},\varepsilon_{q_{i}}), \quad t\in[t_{i},t_{i+1}].
	\end{equation*}
	In particular,
	\begin{equation}
		\label{6.6.8}	 \deg(H_{t_{i+1}},\O_{t_{i+1}},\varepsilon_{q_{i}})=\deg(H_{t_{i+1}},\O_{t_{i}},\varepsilon_{q_{i}}).
	\end{equation}
	Moreover, since $0\notin H([t_{i},t_{i+1}]\times\partial\O_{t_{i}})$, by Lemma \ref{leiii.5},
	\begin{equation*}
		\deg(H_{t_{i}},\O_{t_{i}},\varepsilon_{q_{i}})=
		\sigma(D_{u}H\circ\gamma,[t_{i},t_{i+1}])\deg(H_{t_{i+1}},\O_{t_{i}},\varepsilon_{q_{i}}),
	\end{equation*}
	where $\gamma\in\mathcal{C}([t_{i},t_{i+1}],[t_{i},t_{i+1}]\times \mc{O})$ is the curve defined by $t\mapsto (t,q_{i})$, $t\in[t_{i},t_{i+1}]$. Since $D_{u}H(t,q_{i})\in GL(U,V)$ for each $t\in[t_{i},t_{i+1}]$, necessarily $\sigma(D_{u}H\circ\gamma,[t_{i},t_{i+1}])=1$, and therefore
	\begin{equation*}
		\deg(H_{t_{i}},\O_{t_{i}},\varepsilon_{q_{i}})=\deg(H_{t_{i+1}},\O_{t_{i}},\varepsilon_{q_{i}}).
	\end{equation*}
	Thus, by \eqref{6.6.8}, we find that, for every $i\in\{0,...,m-1\}$,
	\begin{equation}
		\label{6.6.9}
		\deg(H_{t_{i+1}},\O_{t_{i+1}},\varepsilon_{q_{i}})= \deg(H_{t_{i}},\O_{t_{i}},\varepsilon_{q_{i}}).
	\end{equation}
Once again by Lemma \ref{leiii.5},
	\begin{equation*}
		\deg(H_{t_{i+1}},\O_{t_{i+1}},\varepsilon_{q_{i}})=
		\sigma(D_{u}H\circ\gamma_{i},[0,1])\deg(H_{t_{i+1}},\O_{t_{i+1}},\varepsilon_{q_{i+1}}),
	\end{equation*}
	where $\gamma_{i}\in\mathcal{C}([0,1],\{t_{i+1}\}\times\mc{O})$ links $(t_{i+1},q_{i})$ to  $(t_{i+1},q_{i+1})$. Hence, by \eqref{6.6.9},
	\begin{equation*}
		\deg(H_{t_{i}},\O_{t_{i}},\varepsilon_{q_{i}})=
		\sigma(D_{u}H\circ\gamma_{i},[0,1])\deg(H_{t_{i+1}},\O_{t_{i+1}},\varepsilon_{q_{i+1}}).
	\end{equation*}
	Therefore, we can infer that
	\begin{align*}
		\deg(H_{0},\O_{0},\varepsilon_{p_{0}})=
		\deg(H_{t_{0}},\O_{t_{0}},\varepsilon_{q_{0}}) =
		\prod_{i=0}^{m-2}\sigma(D_{u}H\circ\gamma_{i},[0,1]) \deg(H_{t_{m-1}},\O_{t_{m-1}},\varepsilon_{q_{m-1}}).
	\end{align*}
Consequently, by \eqref{6.6.9}, we obtain that
	\begin{align*}
		\deg(H_{0},\O_{0},\varepsilon_{p_{0}})&=\prod_{i=0}^{m-2}\sigma(D_{u}H\circ\gamma_{i},[0,1]) \deg(H_{t_{m}},\O_{t_{m}},\varepsilon_{q_{m-1}})\\
		& =\prod_{i=0}^{m-2}\sigma(D_{u}H\circ\gamma_{i},[0,1]) \deg(H_{1},\O_{1},\varepsilon_{p_1}).
	\end{align*}
	Finally, let $\Gamma_{i}\in\mc{C}([t_{i},t_{i+1}],[t_{i},t_{i+1}]\times \mc{O})$ be the curves defined by $t\mapsto (t,q_{i})$, $t\in[t_{i},t_{i+1}]$, for each $0\leq i\leq m-1$. Then, after the necessary affine changes, the composite curve $\gamma\in\mc{C}([a,b],[0,1]\times\mc{O})$,
	$$\gamma:=\Gamma_{0}\ast \gamma_{0}\ast \Gamma_{1}\ast\gamma_{1}\ast \cdots \ast \gamma_{m-2}\ast \Gamma_{m-1},$$
	links $(0,p_{0})$ to $(1,p_{1})$. By the properties of the parity,
	\begin{align*}
		\sigma(D_{u}H\circ\gamma,[a,b])& =\prod_{i=0}^{m-2}\sigma(D_{u}H\circ\gamma_{i},[0,1])\prod_{i=0}^{m-1}\sigma(D_{u}H\circ\Gamma_{i},[t_{i},t_{i+1}]) \\
		&= \prod_{i=0}^{m-2}\sigma(D_{u}H\circ\gamma_{i},[0,1]).
	\end{align*}
Therefore, we finally obtain
$$
\deg(H_{0},\O_{0},\varepsilon_{p_{0}})=\sigma(D_{u}H\circ\gamma,[a,b])\deg(H_{1},\O_{1},\varepsilon_{p_1}).
$$
This proves \eqref{HomPA2}. The identity \eqref{6.6.7} follows from Theorem \ref{th3.2}. This ends the proof.
\end{proof}

A simplification of the preceding proof yields a counterpart of Theorem \ref{T6.1.1}
without specifying the regular points.

\begin{theorem}
	\label{T6.1.1.1}
	Let $(H,\O,\varepsilon)\in \mathscr{G}(\mc{O})$ be a generalized $\mc{O}$-admissible homotopy. Then,
	\begin{equation}
		\label{HomPA2.1}
		\deg(H_{0},\Omega_{0},\varepsilon_{0})=\deg(H_{1},\Omega_{1},\varepsilon_{1}).
	\end{equation}
\end{theorem}

\subsection{A Leray--Schauder continuation theorem}

We conclude this section by delivering a generalized version of the \emph{Leray--Schauder continuation theorem} for Fredholm operators of index zero. Some precursors in the context of the Leray--Schauder degree were given by Mawhin \cite{Maw}.

\begin{theorem}
	Let $(H,\O,\varepsilon)\in\mathscr{H}(\mc{O})$ be a $\mc{O}$-admissible homotopy with
	$\deg(H_{0},\Omega,\varepsilon_0)\neq 0$. Then, there exists a connected component $\mathscr{C}\subset H^{-1}(0)\cap\O$ that connects $\{0\}\times\Omega$ with $\{1\}\times\Omega$.
\end{theorem}

\begin{proof}
	Since $\deg(H_{0},\Omega,\varepsilon_0)\neq 0$, by the existence property,  $H_{0}^{-1}(0)\cap\O\neq \emptyset$, and $\deg(H_{0},\Omega)\neq 0$.  Let $\mathscr{D}$ be the disjoint union of the connected components $\mathscr{C}$ of $H^{-1}(0)\cap\O$ satisfying $\mathscr{C}\cap H_{0}^{-1}(0)\neq \emptyset$, and consider an isolating neighborhood, $\mathcal{U}$, of $\mathscr{D}$, i.e., an open subset of $[0,1]\times\Omega$ such that $\mathscr{D}\subset\mathcal{U}$ and $H^{-1}(0)\cap\partial\mathcal{U}=\emptyset$. The existence of $\mathcal{U}$ follows by Property (9.3) on Chapter I of Whyburn \cite{W}, see also the forthcoming Lemma \ref{Wh.}.
	\par
	If $\mathscr{D}$ intersects $\{1\}\times\Omega$, we are done by simply choosing $\mathscr{C}$ to be one of the connected components of $\mathscr{D}$ such that $\mathscr{C}\cap[\{1\}\times\Omega]\neq\emptyset$. Suppose that $\mathscr{D}$ does not intersect $\{1\}\times\Omega$. Then, necessarily, $\mathcal{P}_{t}(\mathscr{D})=[0,t_{0}]$ with $0\leq t_{0}<1$, where
	$\mathcal{P}_{t}:[0,1]\times \O\to [0,1]$, $(t,u)\mapsto t$, 
	stands for the $t$-projection operator. Since  $H_{0}^{-1}(0)\cap\O\subset\mathcal{U}_{0}$, by the excision property of the absolute degree, it is apparent that $\deg(H_{0},\Omega)=\deg(H_{0},\mathcal{U}_{0})$. Thus, by homotopy invariance,
	$$
	0\neq \deg(H_{0},\Omega)=\deg(H_t,\mathcal{U}_{t})\quad \hbox{for all}\;\; t\in [0,1].
	$$
	Therefore, $\mathcal{U}_t\neq \emptyset$ for all $t\in (t_0,1]$, which is impossible
	if $\mathcal{U}$ is chosen to be sufficiently close to $\mathscr{D}$. This ends the proof.
\end{proof}

\section{Local Bifurcation Theory}\label{SLBT}	

\noindent In this section we deliver our main local bifurcation theorem for Fredholm operators.   Essentially, it is a generalization and a re-elaboration of the local bifurcation results of Fitzpatrick, Pejsachowicz and Rabier \cite{FPRb}, Pejsachowicz and Rabier \cite{PR}, and \cite[Th. 6.2.1]{LG01}, through the concept of algebraic multiplicity of Esquinas and L\'{o}pez-G\'{o}mez \cite{Es,ELG,LG01}. It is worth-mentioning that it is a very substantial generalization of the
pioneering local bifurcation theorems of Krasnoselskij \cite{Kr}, Rabinowitz \cite{Ra} and
Ize \cite{Iz}, collected in \cite{Ni} by Nirenberg.
\par
Throughout this section, given a pair $(U,V)$ of real Banach spaces, two real values $\l_{-}<\l_{+}$ and a neighbourhood $\mc{U}\subset \R\times U$ of $[\l_{-},\l_{+}]\times \{0\}$, we deal with $\mathcal{C}^{1}$ operators $\mathfrak{F}:\mc{U}\subset \mathbb{R}\times U\to V$ satisfying the hypothesis:
\begin{enumerate}
	\item[{\rm (F1)}]  $\mathfrak{F}(\lambda,0)=0$ for all $(\lambda,0)\in\mc{U}$.
	\item[{\rm (F2)}] $D_{u}\mathfrak{F}(\lambda,0)\in\Phi_{0}(U,V)$ for all $(\lambda,0)\in\mc{U}$.
\end{enumerate}
\noindent The set of solutions
$\mc{T}:=\{(\l,0):(\l,0)\in\mc{U}\}$ is called the set of trivial solutions of $\mf{F}(\l,u)=0$, and
$$
  \left[ \mf{F}^{-1}(0)\backslash \mc{T}\right]\cup \{(\l,0):\;\l\in \Sigma(D_{u}\mathfrak{F}(\cdot,0))\}
$$
is referred to as the set of non-trivial solutions of $\mf{F}(\l,u)=0$. We will denote by
$$\mathfrak{L}(\l):=D_{u}\mf{F}(\l,0), \quad (\lambda,0)\in\mc{U},$$
the linearization of $\mf{F}$ on the set of trivial solutions. In particular $\mf{L}(\l)\in\Phi_{0}(U,V)$ is a continuous path of Fredholm operators of index zero. Given $(\lambda_{0},0)\in\mc{U}$, it is said that $(\l_{0} , 0)$ is a bifurcation point of $\mf{F}(\l,u)=0$ from $\mc{T}$ if there exists a sequence $\{(\l_{n},u_{n})\}_{n\in\mathbb{N}}\subset\mf{F}^{-1}(0)$, with $u_{n}\neq 0$ for all $n \geq 1$, such that
$\lim_{n\to\infty} (\l_{n}, u_{n}) = (\l_{0}, 0)$.
\par
Let $\mf{F}\in\mc{C}^{1}(\mc{U},V)$ satisfying conditions {\rm{(F1)--(F3)}} and let $(\l_{0},0)\in \mc{U}$ be a bifurcation point of $\mf{F}(\l,u)=0$ from $\mc{T}$. Then, by the implicit function theorem, $\l_{0}\in\Sigma(\mf{L})$. The main result of this section reads as follows.

\begin{theorem}[\textbf{Local bifurcation}]
	\label{T6.2.1}
	Let $(U, V)$ be a pair of real Banach spaces, $\l_{-}, \l_{+}$ be two real numbers such that $\l_{-}<\l_{+}$, and $\mc{U}\subset \R\times U$ be an open neighborhood of $[\l_{-},\l_{+}]\times \{0\}$. Consider $\mathfrak{F}\in\mathcal{C}^{1}(\mc{U},V)$ such that:
	\begin{enumerate}
		\item[{\rm (a)}] $\mathfrak{F}(\lambda,0)=0$ and $\mf{L}(\l):=D_{u}\mathfrak{F}(\lambda,0)\in\Phi_{0}(U,V)$, $(\lambda,0)\in \mc{U}$.
		\item[{\rm (b)}] $\mathfrak{L}(\lambda_{\pm})\in GL(U,V)$ and
		\begin{equation}
			\label{6.6.10}
			\chi[\mathfrak{L}_{\omega},[\lambda_{-},\lambda_{+}]]\in 2\mathbb{N}+1,
		\end{equation}
		where $\mathfrak{L}_{\omega}\in\mathscr{C}^{\omega}([\lambda_{-},\lambda_{+}],\Phi_{0}(U,V))$ is any  analytic curve $\mc{A}$-homotopic to $\mf{L}(\l)$, $\lambda\in[\lambda_{-},\lambda_{+}]$.
	\end{enumerate}
Then, the following assertions are satisfied:
	\begin{enumerate}
		\item[{\rm (i)}] There exists $\lambda_{0}\in(\lambda_{-},\lambda_{+})$ such that $(\lambda_0,0)$ is a bifurcation point of $\mathfrak{F}(\lambda,u)=0$ from $\mc{T}\equiv\{(\lambda,0):(\l,0)\in\mc{U}\}$.
		\item[{\rm (ii)}] There exists $\eta_{0}>0$  such that,  for every $\eta\in(0,\eta_{0})$, there is a connected component of the set of non-trivial solutions of $\mf{F}(\l,u)=0$,
		$$
		\mathfrak{C}\subset\mathfrak{F}^{-1}(0)\backslash \{(\lambda,0)\in\mc{U}:\lambda\notin\Sigma(\mathfrak{L})\},
		$$
		joining $T:=\{(\lambda,0):\lambda\in\Sigma(\mathfrak{L})\cap (\lambda_{-},\lambda_{+})\}$ to
		the surface $\|x\|=\eta$.
	\end{enumerate}
\end{theorem}

In the proof of this result we need the following result of Whyburn \cite{W}.

\begin{lemma}
	\label{Wh.}
	Let $(M,d)$ be a compact metric space and $A$ and $B$ two disjoint compact subsets of M. Then, either there exists a connected component of $M$ meeting both $A$ and $B$, or $M = M_{A} \uplus M_{B},$
	where $M_{A}$ and $M_{B}$ are disjoint compact subsets of $M$ containing $A$
	and $B$, respectively.
\end{lemma}

\begin{proof}[Proof of Theorem \ref{T6.2.1}]
	Let us start by fixing $\varepsilon>0$ so that
	$$[\l_{-}-\varepsilon,\l_{+}+\varepsilon]\times \{0\}\subset \mc{U}.$$
	Since $D_{u}\mathfrak{F}(\lambda,0)\in\Phi_{0}(U,V)$ for all $\lambda\in[\l_{-}-\varepsilon,\l_{+}+\varepsilon]$, necessarily $$D\mathfrak{F}(\lambda,0)\in\Phi_{1}(\mathbb{R}\times U,V) \quad \text{for all } \lambda\in[\l_{-}-\varepsilon,\l_{+}+\varepsilon].$$
	Thus, by the Fredholm stability theorems (see, e.g., Kato \cite{Ka}), for every $\lambda\in[\lambda_{-}-\varepsilon,\lambda_{+}+\varepsilon]$, there exists
	an open interval $\mathcal{I}(\lambda)\subset[\lambda_{-}-\varepsilon,\lambda_{+}+\varepsilon]$, and $r(\l)>0$ sufficiently small such that
	$$
	D\mathfrak{F}(\lambda,u)\in \Phi_{1}(\mathbb{R}\times U,V) \quad \hbox{for all } (\lambda,u)\in\mathcal{I}(\lambda)\times B_{r(\lambda)}.
	$$
	By the compactness of $[\lambda_{-}-\varepsilon,\lambda_{+}+\varepsilon]\times\{0\}$, there exist an integer $N\geq 1$ and $N$ points  $\lambda_i\in [\lambda_{-}-\varepsilon,\lambda_{+}+\varepsilon]$, $1\leq i\leq N$, such that
	$$
	[\lambda_{-}-\varepsilon,\lambda_{+}+\varepsilon]\times\{0\}\subset \bigcup_{i=1}^{N}\mathcal{I}(\lambda_{i})\times B_{r(\lambda_{i})}.
	$$
	Therefore $D\mathfrak{F}(\lambda,u)\in\Phi_{1}(\mathbb{R}\times U,V)$ for all $\lambda\in [\lambda_{-}-\varepsilon,\lambda_{+}+\varepsilon]$ and $u\in B_{r}$, where
	$$
	r:=\min\{r(\lambda_{1}),r(\lambda_{2}),...,r(\lambda_{N})\}>0.
	$$
	Moreover, since the Fredholm maps are locally proper (see, e.g., Smale \cite{Sm}), we can repeat the previous argument to show that actually $\mathfrak{F}$ is a proper Fredholm map of index one on $[\lambda_{-}-\varepsilon,\lambda_{+}+\varepsilon]\times \overline{B}_r$ for sufficiently small $r>0$. Moreover, shrinking $r>0$ if necessary, we can suppose that $D_{u}\mf{F}(\l_{\pm},u)\in GL(U,V)$ for each $u\in B_{r}$. Consequently, $\mf{F}(\l_{\pm},\cdot)$ is orientable and $\mathcal{R}_{D_{u}\mathfrak{F}(\lambda_{\pm},\cdot)}\neq \emptyset$. Hence $\mathfrak{F}$ is orientable on $[\lambda_{-}-\varepsilon,\lambda_{+}+\varepsilon]\times B_r$. Denoting
$\mc{O}:=B_{r}$ and $\O:=B_{r/2}$, we have proved that $(H,\O,\varepsilon)\in \mathscr{H}(\mc{O})$ is a $\mc{O}$-admissible homotopy.
\par
To prove the part (i) we argue by contradiction. Suppose that $\mathfrak{F}(\lambda,u)=0$ does not admit a bifurcation value in $(\lambda_{-},\lambda_{+})$. Then, shortening $r>0$, if necessary, one can assume that
	$$
	([\lambda_{-},\lambda_{+}]\times \overline{\O})\cap \mathfrak{F}^{-1}(0)=[\lambda_{-},\lambda_{+}]\times\{0\}.
	$$
	In particular, $0\notin \mathfrak{F}([\lambda_{-},\lambda_{+}]\times \partial\O)$. Since $D_{u}\mathfrak{F}(\lambda_{\pm},0)\in GL(U,V)$, $0$ is a regular value of both $D_{u}\mathfrak{F}(\lambda_\pm,0)$. Thus, Lemma \ref{leiii.5} with $p_0=p_1=0$ yields
	$$
	\deg(\mathfrak{F}(\lambda_{-},\cdot),\O,\varepsilon_{0})=
	(-1)^{\chi[\mathfrak{L}_{\omega},[\lambda_{-},\lambda_{+}]]}
	\deg(\mathfrak{F}(\lambda_{+},\cdot),\O,\varepsilon_{0}),
	$$
where $\varepsilon_0$ is the orientation such that $\varepsilon_0(0)=1$, and $\mathfrak{L}_{\omega}\in\mathscr{C}^{\omega}([\lambda_{-},\lambda_{+}],\Phi_{0}(U,V))$ is any analytic curve $\mc{A}$-homotopic to $D_{u}\mathfrak{F}(\lambda,0)$. As, due to the hypothesis (b),  $$\chi[\mathfrak{L}_{\omega},[\lambda_{-},\lambda_{+}]]\in 2\mathbb{N}+1,$$ we find that
	\begin{equation}
		\label{InvDeg}
		\deg(\mathfrak{F}(\lambda_{-},\cdot),\O,\varepsilon_{0})=
		-\deg(\mathfrak{F}(\lambda_{+},\cdot),\O,\varepsilon_{0}).
	\end{equation}
	On the other hand, since $D_{u}\mathfrak{F}(\lambda_{\pm},0)\in GL(U,V)$, by the inverse function theorem, shortening $r>0$ if necessary, we can suppose that $\mf{F}(\l_{\pm},\cdot): \O\to V$ is injective. Consequently, by the definition of the degree for regular values, it is apparent that
	$$
	\deg(\mathfrak{F}(\lambda_{\pm},\cdot),\O,\varepsilon_{0})=\sum_{u\in \mf{F}_{\l_{\pm}}^{-1}(0)\cap\O}\varepsilon_{0}(u)=\varepsilon_{0}(0)=1.
	$$
This contradicts \eqref{InvDeg} and concludes the proof of Part (i).
\par
Now, we will prove the part (ii). Fix $r>0$ so that $\mf{F}$ is a proper Fredholm map of index one on $[\l_{-}-\varepsilon,\l_{+}+\varepsilon]\times \overline{B}_{r}$ as we have done in the proof of the part (i). Subsequently, for every $0<\eta<r$, we consider the closed cylinder $Q_\eta:=[\lambda_{-},\lambda_{+}]\times \overline{B}_\eta$, and the set of non-trivial solutions of $\mathfrak{F}(\lambda,u)=0$,
	\begin{equation*}
		\mathcal{S}\equiv \{(\lambda,u)\in \mathfrak{F}^{-1}(0):u\neq 0\}\uplus \{(\lambda,0)\in\mc{U}:\lambda\in \Sigma(\mathfrak{L})\}.
	\end{equation*}
	It is easily seen that $\mathcal{S}$ and $\Sigma(\mathfrak{L})$ are closed. Thus, as $Q_{\eta}$ is closed and bounded and $\mathcal{S}\subset \mathfrak{F}^{-1}(0)$, the set $M:=\mathcal{S}\cap Q_{\eta}$ is compact, because $\mathfrak{F}$ is proper on $[\l_{-}-\varepsilon,\l_{+}+\varepsilon]\times \overline{B}_{r}$. Now, consider the subsets of $M$
	\begin{equation*}
		A:=\{(\lambda,u)\in M: \|u\|=\eta\} \;\; \text{ and }\;\;
		B:=\{(\lambda,0)\in M: \lambda\in\Sigma(\mathfrak{L})\}.
	\end{equation*}
	By Part (i), we already know that there exists $\lambda_0\in (\lambda_{-},\lambda_{+})$ such that
	$(\lambda_0,0)$ is a bifurcation point of $\mathfrak{F}(\lambda,u)=0$ from $(\lambda,0)$. Thus,
	there exists  $\eta>0$ such that $A\neq\emptyset$. Clearly,  $B\neq \emptyset$ because  $(\lambda_{0},0)\in B$. Therefore, $A$ and $B$ are non-empty disjoint compact subsets of $M$.
	Part (ii) establishes the existence of a continuum $\mathfrak{C}$ linking $A$ to $B$.
	To prove this we argue by contradiction. So, assume that $A$ and $B$ are not in the same
	connected component of $M$. Then, according to Lemma \ref{Wh.}, $M=M_{A}\uplus M_{B}$, where $M_{A}$ and $M_{B}$ are disjoint compact subsets of $M$ containing $A$ and $B$. Since $\mathrm{dist\,}(M_{A},M_{B})>0$, there exists $\delta>0$ such that the open $\delta$-neighborhood
$M^{\delta}_{B}:=M_{B}+B_{\delta}(0,0)$ satisfies $M^{\delta}_{B}\cap M_{A}=\emptyset$. Moreover, since $\lambda_{\pm}\notin \Sigma(\mathfrak{L})$, by the implicit function theorem,  we can choose $\eta>0$ and $\varepsilon>0$ sufficiently small so that $(\lambda,u)\in M$ implies $\lambda_{-}+\varepsilon<\lambda<\lambda_{+}-\varepsilon$. Hence, shortening $\delta>0$, if necessary, we have that
	\begin{equation}
		\label{6.6.11}
		(\lambda,u)\in M^{\delta}_{B} \;\; \Rightarrow \;\; \lambda_{-}+\frac{\varepsilon}{2}<\lambda<\lambda_{+}-\frac{\varepsilon}{2},
	\end{equation}
	and therefore $M^{\delta}_{B}\subset Q_{\eta}$. By construction, $\partial M^{\delta}_{B}\cap M=\emptyset$. Moreover, $Q_{\beta}\cap M_{A}=\emptyset$
	for sufficiently small $\beta\in (0,\eta)$. Thus, for these $\beta$'s, the set
$\O:=\Int Q_{\beta}\cup M^{\delta}_{B}$, where the interior is taken over $[\l_{-},\l_{+}]\times U$, is a bounded and connected open in $[\l_{-},\l_{+}]\times U$ and satisfies
$M_{B}\subset \O$ and $\partial\O\cap M=\emptyset$.
	\begin{figure}[t]
		\begin{center}

			\tikzset{every picture/.style={line width=0.75pt}} 
			
			\begin{tikzpicture}[x=0.75pt,y=0.75pt,yscale=-1,xscale=1]
				
				\draw    (118.59,53.93) -- (119,163) ;
				\draw    (295.59,53.93) -- (296,163) ;
				\draw    (118.59,53.93) -- (295.59,53.93) ;
				\draw [line width=2.25]    (194.88,133.16) .. controls (305.88,37.16) and (274.88,130.16) .. (225,163) ;
				\draw [line width=2.25]    (136,54) .. controls (131,108) and (163,119) .. (192,70) ;
				\draw  [dash pattern={on 4.5pt off 4.5pt}]  (65.49,108.29) -- (295.79,108.47) ;
				\draw  [fill={rgb, 255:red, 0; green, 0; blue, 0 }  ,fill opacity=1 ] (133,54) .. controls (133,55.66) and (134.34,57) .. (136,57) .. controls (137.66,57) and (139,55.66) .. (139,54) .. controls (139,52.34) and (137.66,51) .. (136,51) .. controls (134.34,51) and (133,52.34) .. (133,54) -- cycle ;
				\draw [line width=2.25]    (59.68,162.65) -- (339.37,162.75) ;
				\draw    (66,33) -- (66,170) ;
				\draw  [fill={rgb, 255:red, 0; green, 0; blue, 0 }  ,fill opacity=1 ] (222,163) .. controls (222,164.66) and (223.34,166) .. (225,166) .. controls (226.66,166) and (228,164.66) .. (228,163) .. controls (228,161.34) and (226.66,160) .. (225,160) .. controls (223.34,160) and (222,161.34) .. (222,163) -- cycle ;
				\draw [fill={rgb, 255:red, 74; green, 144; blue, 226 }  ,fill opacity=0.3 ][line width=0.75]  [dash pattern={on 0.84pt off 2.51pt}]  (200.38,108.48) .. controls (286.38,33.48) and (299.52,68.74) .. (281.52,107.74) ;
				\draw  [dash pattern={on 0.84pt off 2.51pt}]  (200.38,108.48) .. controls (181.38,130.48) and (185,154) .. (225,124) ;
				\draw  [dash pattern={on 0.84pt off 2.51pt}]  (225,124) .. controls (265,94) and (242,142) .. (212,161) ;
				\draw  [dash pattern={on 4.5pt off 4.5pt}]  (65.16,53.96) -- (118.59,53.93) ;
				\draw  [dash pattern={on 0.84pt off 2.51pt}]  (281.52,107.74) .. controls (266.33,137.33) and (267.67,127.67) .. (246.33,161.33) ;
				\draw    (271,119) -- (322.03,128.63) ;
				\draw [shift={(324,129)}, rotate = 190.68] [color={rgb, 255:red, 0; green, 0; blue, 0 }  ][line width=0.75]    (10.93,-3.29) .. controls (6.95,-1.4) and (3.31,-0.3) .. (0,0) .. controls (3.31,0.3) and (6.95,1.4) .. (10.93,3.29)   ;
				\draw  [fill={rgb, 255:red, 74; green, 144; blue, 226 }  ,fill opacity=0.28 ][dash pattern={on 0.84pt off 2.51pt}] (119,108.47) -- (296.21,108.47) -- (296.21,163) -- (119,163) -- cycle ;
				
				\draw (129,32.4) node [anchor=north west][inner sep=0.75pt]    {$A$};
				\draw (219,167.4) node [anchor=north west][inner sep=0.75pt]    {$B$};
				\draw (46,100.4) node [anchor=north west][inner sep=0.75pt]    {$\beta $};
				\draw (44,47.4) node [anchor=north west][inner sep=0.75pt]    {$\eta $};
				\draw (143,68.4) node [anchor=north west][inner sep=0.75pt]    {$M_{A}$};
				\draw (257,71.4) node [anchor=north west][inner sep=0.75pt]    {$M_{B}$};
				\draw (328,119.4) node [anchor=north west][inner sep=0.75pt]    {$M_{B}^{\delta }$};
				\draw (334,167.4) node [anchor=north west][inner sep=0.75pt]    {$\lambda $};
				\draw (71,15.4) node [anchor=north west][inner sep=0.75pt]    {$\| \cdot \| $};
				\draw (111,164.4) node [anchor=north west][inner sep=0.75pt]    {$\lambda _{-}$};
				\draw (289,164.4) node [anchor=north west][inner sep=0.75pt]    {$\lambda _{+}$};

			\end{tikzpicture}
		\end{center}
		
		\caption{Scheme of the construction.}
		\label{F5.3}
	\end{figure}
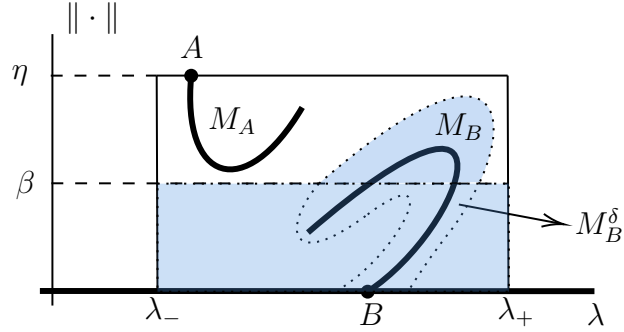
	Setting $\mc{O}:=B_{\eta}$, since $\overline{\O}\subset [\l_{-},\l_{+}]\times \mc{O}$ and $\partial\O\cap M=\emptyset$, it follows that $(\mf{F},\O,\varepsilon)\in\mathscr{G}(\mc{O})$, $\mf{F}:[\l_{-},\l_{+}]\times \mc{O}\to V$, is a generalized $\mc{O}$-admissible homotopy.
	Since $p_{0}=0$ is a regular point of $D\mathfrak{F}_{\lambda_{\pm}}$, by Theorem \ref{T6.1.1}, we find that
\begin{equation}
		\label{6.6.91}
		\deg(\mathfrak{F}(\lambda_{-},\cdot),\O_{\lambda_{-}},\varepsilon_{0})
		=(-1)^{\chi[\mathfrak{L}_{\omega},[\lambda_{-},\lambda_{+}]]}
		\deg(\mathfrak{F}(\lambda_{+},\cdot),\O_{\lambda_{+}},\varepsilon_{0}),
\end{equation}
where $\varepsilon_0$ is the orientation such that $\varepsilon_{0}(0)=1$, and  $\mathfrak{L}_{\omega}\in\mathscr{C}^{\omega}([\lambda_{-},\lambda_{+}],\Phi_{0}(U,V))$ is any analytic curve $\mc{A}$-homotopic to $\mathfrak{L}(\l)$, $\l\in[\l_{-},\l_{+}]$. Furthermore, by  \eqref{6.6.11},
	$\O_{\lambda_{-}}=\O_{\lambda_{+}}=B_{\beta}$, and, by hypothesis,  $\chi[\mathfrak{L}_{\omega},[\lambda_{-},\lambda_{+}]]\in 2\mathbb{N}+1$. Therefore, from \eqref{6.6.91}
	we can infer that
	$$
	\deg(\mathfrak{F}(\lambda_{-},\cdot),B_{\beta},\varepsilon_{0})
	=-\deg(\mathfrak{F}(\lambda_{+},\cdot),B_{\beta},\varepsilon_{0}),
	$$
	which is impossible, as we have already shown at the end of the proof of Part (i). Therefore, $A$ and $B$ must be part of the same connected component of $M$. This ends the proof.
\end{proof}

\section{Global Bifurcation Theory}\label{SGBT}

\noindent In this section, we are going to adapt the global theorem of L\'{o}pez-G\'{o}mez and Mora-Corral \cite{LGMC} to the context of the degree for Fredholm operators of Fitzpatrick,  Pejsachowicz and  Rabier \cite{FPRb,PR}. Throughout this part, we consider a $\mc{C}^{1}$ function $\mathfrak{F}:\mathbb{R}\times U\to V$ such that
\begin{enumerate}
	\item[(F1)] $\mf{F}$ is orientable with orientation $\varepsilon:\mc{R}_{D_{u}H}\to\Z_{2}$.
	\item[(F2)] $\mathfrak{F}(\lambda,0)=0$ for all $\lambda\in\mathbb{R}$.
	\item[(F3)] $D_{u}\mathfrak{F}(\lambda,u)\in\Phi_{0}(U,V)$ for every $\lambda\in\mathbb{R}$ and $u\in U$.
	\item[(F4)] $\mathfrak{F}$ is proper on bounded and closed subsets of $\mathbb{R}\times U$.
\end{enumerate}
The next definition fixes  the concept of admissible family of intervals for $\Sigma(\mathfrak{L})$.
Given two disjoint real intervals, $A$ and $B$, it is said that $A<B$ if $a<b$ for every $a\in A$ and $b\in B$.

\begin{definition}
	Let $\mc{J}$ be a non-empty locally finite family of disjoint non-empty open intervals of $\R$.  It is said that $\mc{J}$ is an admissible family of intervals for $\Sigma(\mathfrak{L})$ if
	\begin{equation*}
		J \cap \Sigma(\mathfrak{L}) = \emptyset \qquad \mbox{for all } \; J\in \mathcal{J}.
	\end{equation*}
	If, in addition, there are $r, s\in\Z\cup \{\pm\infty\}$, $r \leq s$, such that $\mc{J}= \{ J_{n} \}_{n=r}^s$ with 	$J_{n-1} < J_n$ for all $n\in\Z\cap [r+1,s]$, then  $\mc{J}$ is said to be
	an admissible ordered family of intervals for $\Sigma(\mathfrak{L})$.
\end{definition}
\noindent Associated to any admissible ordered family of intervals $\mathcal{J}=\{J_n\}_{n=r}^s$ for $\Sigma(\mathfrak{L})$, we have the associated family of compact intervals
$\mathcal{I}=\{I_n\}_{n=r+1}^s$ defined by
$$
I_n := [\sup J_{n-1}, \inf J_n]\,,\qquad  n\in\Z\cap [r+1,s].
$$
For each $n\in \Z \cap [r+1,s]$, one has that $I_n\neq \emptyset$, because $J_{n-1}<J_n$, though $I_n$ might consist of a single point. Moreover, $I_n < I_{n+1}$ for each $n\in\Z\cap [r+1,s-1]$. By \cite[Le. 4.3]{LGMC05}, the family $\mc{I}$ is also locally finite.
\noindent Now, we introduce the concept of \emph{$\mc{J}$-parity map} associated to an  admissible ordered family
of intervals, $\mc{J}= \{ J_n \}_{n=r}^s$  for $\Sigma(\mathfrak{L})$. Choose a point $\lambda_{n}\in J_{n}$ for every $n\in\Z\cap [r,s]$. Since $\l_{r}\notin \Sigma(\mf{L})$, necessarily $\mf{L}(\l_{r})\in GL(U,V)$ and $p=0$ is a regular point of $\mf{F}(\l_{r},\cdot)$. In fact, $p=0$ is a regular point of $\mf{F}(\l_{n},\cdot)$ for each $n\in\Z\cap [r,s]$. Let us suppose, without lost of generality, that $\varepsilon_{\l_{r}}(0)=\varepsilon(\l_{r},0)=1$.
This choice determines uniquely the total orientation $\varepsilon$. Consider the sequence $\{a_{n}\}_{n=r}^{s}\subset\{-1,1\}$ defined recursively by
\begin{equation*}
	a_{r}=1, \quad a_{n}=a_{n-1} \cdot \sigma(\mf{L},[\lambda_{n-1},\lambda_{n}]) \quad
	\hbox{for} \;\; n\in\mathbb{Z}\cap[r+1,s].
\end{equation*}
\noindent The following two lemmas provide us with two closed formulae for $a_{n}$.
\begin{lemma}
\label{Lemma9.4.1,2}
For every $n\in\mathbb{Z}\cap[r+1,s]$, we have that $a_{n}=\sigma(\mf{L},[\l_{r},\l_{n}])$.
\end{lemma}
\begin{proof}
	Let $n\in\mathbb{Z}\cap[r+1,s]$. By applying inductively the properties of the parity yields 	 \begin{align*}
		a_{n}=a_{n-1}\cdot \sigma(\mf{L},[\l_{n-1},\l_{n}])=a_{r}\cdot \prod_{i=r+1}^{n}\sigma(\mf{L},[\l_{i-1},\l_{i}])=\sigma(\mf{L},[\l_{r},\l_{n}]).
	\end{align*}
	This concludes the proof.
\end{proof}
\begin{lemma}
	\label{Lemma9.4.2,2}
	For every $n\in\N$, we have that $a_{n}=\varepsilon_{\l_{n}}(0)$.
\end{lemma}
\begin{proof}
	By Lemma \ref{Lemma9.4.1,2}, we have that $a_{n}=\sigma(\mf{L},[\l_{r},\l_{n}])$, for all  $n\in\mathbb{Z}\cap[r+1,s]$. Thus, by the definition of the orientation, 	
$$
 \sigma(\mf{L},[\l_{r},\l_{n}])=\varepsilon(\l_{r},0)\cdot\varepsilon(\l_{n},0)=\varepsilon_{\l_{r}}(0)\cdot \varepsilon_{\l_{n}}(0).
$$
Since we have set $\varepsilon_{\l_{r}}(0)=1$, it follows that $a_{n}=\varepsilon_{\l_{n}}(0)$.
\end{proof}

\noindent The next definition introduces the notion of $\mc{J}$-parity map.

\begin{definition}
	The $\mc{J}$-parity map $\mathcal{P}$ associated to the admissible ordered 	family $\mc{J}:=\{ J_n \}_{n=r}^s$  is defined through
	\begin{equation*}
		\mathcal{P}: \;\mathcal{I} \longrightarrow \{-1,0,1\} \,, \qquad
		\mathcal{P}(I_{n})=\frac{a_{n}-a_{n-1}}{2},
	\end{equation*}
	where $\mathcal{I}=\{I_n\}_{n=r+1}^s$.
\end{definition}
\noindent Note that, setting
$$
\Gamma_0:=\{n\in \mathbb{Z}\cap [r+1,s] : a_{n-1}=a_n\}, \quad
\Gamma_1:= \{n\in \mathbb{Z}\cap [r+1,s] : a_{n-1}\neq a_n\},
$$
the $\mc{J}$-parity $\mc{P}$ satisfies the following properties:
\begin{enumerate}
	\item[(a)] $\mathcal{P}(I_{n})=0$ if $n\in \Gamma_0$,
	\item[(b)] $\mathcal{P}(I_{n})=\pm 1$ if $n\in \Gamma_1$,
	\item[(c)] $\mathcal{P}(I_{n})\mathcal{P}(I_{m})=-1$ if $n, m\in \Gamma_1$ with $n<m$ and $(n,m)\cap \Gamma_1=\emptyset$.
\end{enumerate}
The following lemma will be invoked later. It relates the $\mc{J}$-parity map with the  degree.
\begin{lemma}
	\label{L9.4.4,2}
Let $n\in \mathbb{Z}\cap [r+1,s]$. Then, for sufficiently small $\rho>0$ and $\delta>0$,
\begin{equation}
		\label{eqll}
		2\mc{P}(I_{n})=\deg(\mf{F}(\l_{n},\cdot), B_{\rho}, \varepsilon_{\l_{n}})-\deg(\mf{F}(\l_{n-1},\cdot), B_{\rho}, \varepsilon_{\l_{n-1}}).
\end{equation}
\end{lemma}
\begin{proof}
	Fix some $n\in \mathbb{Z}\cap [r+1,s]$. By definition,
	\begin{equation*}
		2\mc{P}(I_{n})=a_{n}-a_{n-1}=a_{n-1}\left[\sigma(\mf{L},[\lambda_{n-1},\lambda_{n}])-1\right].
	\end{equation*}
Thus, by Lemma \ref{Lemma9.4.2,2},
	\begin{equation}
		\label{eqqLl,2}
		2\mc{P}(I_{n})=\varepsilon_{\l_{n-1}}(0)\left[\sigma(\mf{L},[\lambda_{n-1},\lambda_{n}])-1\right].
	\end{equation}
	Recall that $p=0$ is a regular point of $\mf{F}(\l,\cdot)$ for each $\l\in\{\l_{n-1},\l_{n}\}$. By the definition of the orientation,
	\begin{equation*}
		\varepsilon_{\l_{n-1}}(0)\cdot \varepsilon_{\l_{n}}(0)=\sigma(\mf{L},[\lambda_{n-1},\lambda_{n}]),
	\end{equation*}
	or equivalently,
	\begin{equation}
		\label{eqL9.4.2,2}
		\varepsilon_{\l_{n}}(0)=\varepsilon_{\l_{n-1}}(0)\cdot \sigma(\mf{L},[\lambda_{n-1},\lambda_{n}]).
	\end{equation}
	Since $\mf{L}(\l_{n}),\mf{L}(\l_{n-1})\in GL(U,V)$, for sufficiently small $\rho>0$, by the definition of the degree and \eqref{eqL9.4.2,2}, we infer that
	\begin{align*}
		& \deg(\mf{F}(\l_{n},\cdot), B_{\rho}, \varepsilon_{\l_{n}})= \varepsilon_{\l_{n}}(0)=\varepsilon_{\l_{n-1}}(0)\cdot \sigma(\mf{L},[\lambda_{n-1},\lambda_{n}]), \\
		& \deg(\mf{F}(\l_{n-1},\cdot), B_{\rho}, \varepsilon_{\l_{n-1}})=\varepsilon_{\l_{n-1}}(0).
	\end{align*}
Therefore, we deduce that
	\begin{equation*}
		\begin{split}
			\deg(\mf{F}(\l_{n},\cdot), B_{\rho}, \varepsilon_{\l_{n}})-\deg(\mf{F}(\l_{n-1},\cdot), B_{\rho}, \varepsilon_{\l_{n-1}})=\varepsilon_{\l_{n-1}}(0)\left[\sigma(\mf{L},[\lambda_{n-1},\lambda_{n}])-1\right].
		\end{split}
	\end{equation*}
Combining this identity together with \eqref{eqqLl,2} yields \eqref{eqll} and ends the proof.
\end{proof}
\noindent Associated to any admissible ordered family $\mc{J}=\{J_n\}_{n=r}^s$ for $\Sigma(\mathfrak{L})$, we will consider the corresponding \textit{set of non-trivial solutions} through
\begin{equation*}
	\mc{S}_\mc{J}:= \Big( \mf{F}^{-1}(0) \cap [\R \times (U\setminus \{0\})]\Big)\cup
	\left[ \left(\R \setminus \bigcup_{n=r}^s J_n \right) \times \{0\}\right].
\end{equation*}
Clearly, the set $\mc{S}_\mc{J}\subset \mf{F}^{-1}(0)$ is closed. We will say that $\mathfrak{C}\subset \mc{S}_\mc{J}$ is a \textit{component} of $\mc{S}_\mc{J}$ if it is a non-empty closed and connected subset of $\mc{S}_\mc{J}$ maximal for the inclusion, i.e., if it is a connected component of $ \mc{S}_\mc{J}$. Since $\mf{F}$ is proper on closed and bounded subsets, it is easily seen that every bounded component $\mf{C}$ is compact.
\par
Figure \ref{F6.1}  shows an admissible unbounded set of non-trivial solutions $\mc{S}_\mc{J}$ together with
the intervals, $J_1$, $J_2$, $J_3$ and $J_4$, of the admissible ordered family $\mathcal{J}$; $J_1$ and
$J_4$ are unbounded open intervals, while $J_2$ and $J_3$ are bounded. In this example,
$\mc{S}_\mc{J}$ consists of four components, $\mathfrak{C}_i$, $i\in\{1,2,3,4\}$, though, in general, it might consist of an arbitrarily large number of components. The set $\mathbb{R}\setminus \bigcup_{i=1}^4 J_i$ consists of three compact intervals, $I_1$, $I_2$ and $I_3$, containing $\Sigma(\mf{L})$, which, in this particular example, it consists of two compact intervals and three single isolated points. The set $\Sigma(\mf{L})$ is colored in green. The components $\mathfrak{C}_1$ and $\mathfrak{C}_2$ are separated away
from the real axis, which represents the trivial solution $(\lambda,0)$. This is why these components are usually referred to as \emph{isolas}. $\mathfrak{C}_1$ is unbounded in $\R\times U$, while $\mathfrak{C}_2$ is bounded. The remaining two components bifurcate from $(\lambda,0)$ at a variety of spectral values.  $\mathfrak{C}_3$ bifurcates from $I_1$ and it is unbounded, and $\mathfrak{C}_4$ bifurcates from $I_2$ and  $I_3$, and it is bounded. Thus, $\mathfrak{C}_{3}$ contains $I_{1}$ and $\mathfrak{C}_{4}$ contains $I_{2}$ and $I_{3}$.

\begin{figure}[t]
	\begin{center}
		\includegraphics[scale=0.3]{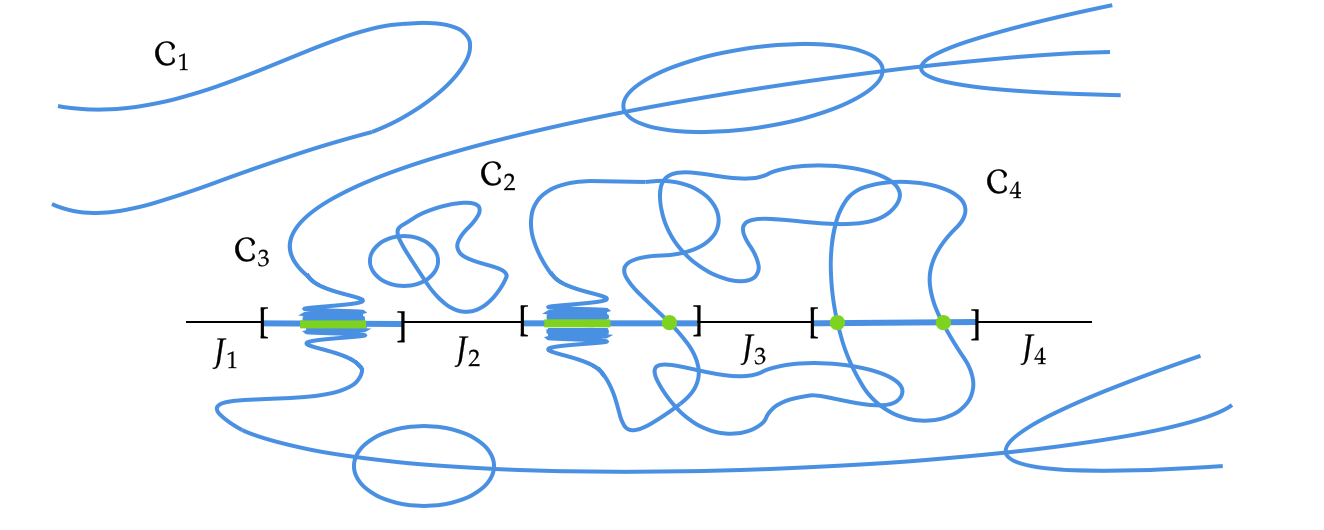}
		\caption{The set of non-trivial solutions $\mc{S}_{\mathcal{J}}$}
		\label{F6.1}
	\end{center}
\end{figure}

Subsequently, for any given component, $\mf{C}$, of $\mc{S}_\mc{J}$, we will denote
\begin{equation*}
	\mc{B}\equiv \mc{B}(\mf{C}):=\left\{ n\in \mathbb{Z}\cap [r+1,s]\;:\; (I_n\times \{0\})\cap\mf{C}\neq \emptyset\right\}.
\end{equation*}
For example, in the special case described by Figure \ref{F6.1}, $\mc{B}(\mf{C}_3)= \{1\}$ and
$\mc{B}(\mf{C}_4)= \{2,3\}$. The next result  shows that $\mc{B}(\mf{C})$ is finite if $\mf{C}$ is compact (see \cite[Le. 2.4]{LGMCC} for its proof).
\begin{lemma}
	\label{le25}  For every bounded component
	$\mf{C}\subset \mc{S}_{\mc{J}}$, $\mc{B}=\mc{B}(\mf{C})$ 	is finite. If, in addition,
	\begin{equation}
		\label{Condi}
		\mf{C} \cap (\R \times \{0\}) \subset \left(\cup_{n=r}^s J_n \cup \cup_{n=r+1}^s I_n \right) \times \{0\},
	\end{equation}
	then there exists $\a>0$ such that
	\begin{equation}\label{24}
		\left[ K + B_{\a}(0,0) \right] \cap \Big[ \big( \R \setminus \cup_{n=r}^s J_n \big) \times \{0\} \Big]    \subset \cup_{n\in \mc{B}} I_n \times \{0\} .
	\end{equation}
\end{lemma}

\noindent Throughout the rest of the section, we fix an admissible ordered family
$\mc{J}=\{J_n\}_{n=r}^s$ for $\Sigma(\mathfrak{L})$ and a bounded component $\mf{C}$ of $\mc{S}_\mc{J}$. Thanks to Lemma\,\ref{le25}, $\mc{B}=\mc{B}(\mf{C})$ is finite. Moreover,
$\cup_{n\in\mc{B}} I_i \times \{0\} \subset \mf{C}$, since $\mf{C}$ is a component. Also, every $\l\in \R \setminus \cup_{n=r}^s J_n$ with $(\l,0) \in \mf{C}$ satisfies $\l\in
\cup_{n\in\mc{B}} I_n$, since $\mf{C}$ is bounded. Thus,
\begin{equation}
	\label{32}
	\mf{C} \cap (\R \times \{0\}) =  \cup_{n\in\mc{B}} I_n \times \{0\} \,.
\end{equation}
The following concept plays a pivotal  role in our subsequent analysis.

\begin{definition}
	\label{D6.3.3}
	A bounded open set $\O \subset \R \times U$ is said to be an open isolating neighbourhood of $\mf{C}$ with size $\eta>0$ if the following conditions are satisfied:
	\begin{enumerate}
		\item[{\rm (a)}] $\mf{C}\subset \O \subset \mf{C}+B_\eta(0,0)$,
		
		\item[{\rm (b)}] $\p \O \cap \mc{S}_\mc{J} =\emptyset$,
		
		\item[{\rm (c)}] $\Big[ \big( \R \setminus \cup_{n=r}^s J_n \big) \times \{0\} \Big] \cap \O \subset		\cup_{n\in \mc{B}} I_n \times \{0\}$.
	\end{enumerate}
\end{definition}

\noindent The condition (a) entails that $\p\Omega$ can be taken as close as we wish to $\mf{C}$,  the condition  (b) means that $\p\Omega$ cannot admit any non-trivial solution, and (c) holds as soon as $\eta>0$ is sufficiently small in (a). Note that item (c) implies that
$$
\mc{B}(\mf{C})=\{ n\in\Z\cap[r+1,s] : (I_n \times \{0\}) \cap \O
\neq \emptyset\}.
$$
To prove the main result of this section, we need the following technical lemma concerning the existence of an open isolating neighbourhood for $\mf{C}$ (see \cite[Pr. 3.3]{LGMCC} for its proof).

\begin{lemma}
	\label{L6.3.4}
	Let  $\mc{J}:=\{ J_n\}_{n=r}^s$  be an admissible ordered family for $\Sigma(\mf{L})$, and  suppose that $\mf{C}$ is a bounded component of $\mc{S}_\mc{J}$. Then, for every $\eta>0$, $\mf{C}$
	possesses an open isolating neighborhood $\O$ of size $\eta$. Moreover, for every open isolating neighborhood $\O$ of $\mf{C}$  and sufficiently small $\e>0$, there exists $\r_0 >0$ such that, for every $0 < \r \leq \r_0$ and
	\begin{equation}
		\label{6.6.17}
		\l \in \Lambda(\mf{C},\varepsilon):= \mathbb{R} \setminus \cup_{n \in \mc{B}} \left[ I_n + (-\e,\e) \right],
	\end{equation}
	some of the following alternatives occurs:
	\begin{enumerate}
		\item[\rm (i)] Either $\overline{B}_\r \cap \O_\l=\emptyset$, or
		
		\item[\rm (ii)] $\left\{u\in \overline{B}_{\r}: \mf{F}(\l,u)=0 \right\}= \{0\}$.
	\end{enumerate}
\end{lemma}

\noindent The main result of this section can be stated as follows.

\begin{theorem}
	\label{T6.3.5,2} Let $\mf{F}\in\mc{C}^{1}(\R\times U,V)$ be a map satisfying {\rm (F1)-(F4)},  $\mc{J}:=\{ J_n \}_{n=r}^s$  be an admissible ordered family for $\Sigma(\mf{L})$, and  suppose that $\mf{C}$ is a bounded component of $\mc{S}_\mc{J}$. Then,
	\begin{equation}
		\label{6.6.18}
		\sum_{n\in \mc{B}(\mf{C})} \mathcal{P}(I_n)=0,
	\end{equation}
	where $\mathcal{P}$ is the $\mc{J}$-parity map.
\end{theorem}
\begin{proof}
	Since $\mf{C}$ is compact, $\mc{B}(\mf{C})$ is finite and we can write $\mc{B}=\{n_{k}\}_{k=1}^{m}$ where $m$ is the cardinal of the set $\mc{B}(\mf{C})$ and $n_{k}\in \Z\cap [r+1,s]$ for each $1\leq k\leq m$ with
	$$n_{1}<n_{2}<\cdots<n_{k}<\cdots<n_{m}.$$
	Let $\O$ be an open isolating neighborhood of $\mf{C}$ with size $\eta>0$. Pick a sufficiently small $\delta>0$ such that:
	\begin{enumerate}
		\item $\left(\cup_{n\in\mathcal{B}}I_{n}\times \{0\}\right)+B_{\delta}(0)\subset\Omega$.
		\item $\{I_n + (-\d/2 , \d/2)\}_{n\in\mc{B}}$ consists of disjoint intervals.
		\item Setting $\lambda_{n}^- := \inf I_n-\d/2$ and $\lambda_n^+ := \sup I_{n} + \d/2$
		for each integer $n\in\mc{B}$, one has that $\lambda_n^- \in J_{n -1}$ and $\lambda_n^+ \in J_{n}$ for all $n\in\mc{B}$.
	\end{enumerate}
	Moreover, thanks to Lemma \ref{L6.3.4} with $\e=\delta/2$, there exists $\r_0 >0$ such that, for every $0 < \r \leq \r_0$ and $\lambda \in  \Lambda(\mf{C},\delta/2)$, either
	\begin{equation}
		\label{Thusby}
	\overline{B}_\r \cap \O_\l=\emptyset \ \ \text{  or  } \ \ \{u\in \overline{B}_{\r}: 	\mf{F}(\l,u)=0 \} = \{0\}.
\end{equation}
	Since $\lambda^{\pm}_{n}\notin \Sigma(\mathfrak{L})$, necessarily $D_{u}\mathfrak{F}(\lambda^{\pm}_{n},0)\in GL(U,V)$. Thus, $p_{0}=0$ is a regular point of $\mathfrak{F}(\lambda^{\pm}_{n},\cdot)$ for each $n\in\mc{B}$.
	In order to apply the homotopy invariance property of the degree, the following lemma is needed.
	\begin{lemma}
		\label{LL31,2}
		Let $1\leq k\leq m$. Then, for every  $\l\in[\l_{n_{k}}^{-},\l_{n_{k}}^{+}]$,
		\begin{equation}
			\label{6.6.19}
			0\notin \mathfrak{F}(\{\lambda\}\times\partial\Omega_{\lambda}).
		\end{equation}
	\end{lemma}
\begin{proof}
	Let $\l\in [\l_{n_{k}}^{-},\l_{n_{k}}^{+}]$ and $u\in U$ such that $(\l,u)\in \p \O$ and $\mathfrak{F}(\lambda,u)=0$. Since $(\l,u)\in \p \O$, then, by Definition \ref{D6.3.3}, item (b), $(\l,u)\notin \mc{S}$. Thus, since $\mathfrak{F}(\lambda,u)=0$, necessarily $u=0$. Moreover, due to the property (1), $(\l,0)\in \O$ for all  $\lambda\in[\lambda^{-}_{n_{k}},\lambda^{+}_{n_{k}}]$. So, $(\l,0)\notin\p\O$, a contradiction. This ends the proof of the lemma.
\end{proof}
\noindent Consequently, setting for each $1\leq k\leq m$,
$$\O_{n_{k}}:=\O\cap ([\l_{n_{k}}^{-},\l_{n_{k}}^{+}]\times U),$$
the tern $(\mf{F},\O_{n_{k}},\varepsilon)$ is a generalized $U$-admissible homotopy, $(\mf{F},\O_{n_{k}},\varepsilon)\in\mathscr{G}(U)$, and from Theorem \ref{T6.1.1.1},
\begin{equation}
	\label{6.6.20}	
	\deg(\mathfrak{F}(\lambda^{-}_{n_{k}},\cdot),\Omega_{\lambda^{-}_{n_{k}}},\varepsilon_{\l^{-}_{n_{k}}})=
	\deg(\mathfrak{F}(\lambda^{+}_{n_{k}},\cdot),\Omega_{\lambda^{+}_{n_{k}}},\varepsilon_{\l^{+}_{n_{k}}}).
\end{equation}
In order to apply again the homotopy invariance of the degree, the following counterpart of Lemma \ref{LL31,2} is needed.
\begin{lemma}
	\label{LL311,2}
	Let $1\leq k \leq m-1$. Then,
	\begin{equation}
		\label{6.6.23K}
		\mf{F}(\l,u)\neq 0 \quad \hbox{if}\;\; \l\in  [\l_{n_{k}} ^+,\l_{n_{k+1}} ^-] \;\;  \hbox{and}\;\;
		u\in\p (\O_\l \setminus \overline{B}_{\r}).
	\end{equation}
\end{lemma}
\begin{proof}
	Let $(\l,u)$ be such that $ \l\in  [\l_{n_{k}} ^+,\l_{n_{k+1}} ^-] $ and $u\in\p (\O_\l \setminus \overline{B}_{\r})$ for some $1\leq n\leq m$. Then, $u\neq 0$ and $\lambda \in  \Lambda(\mf{C},\delta/2)$. Thus, by \eqref{Thusby}, either
	$$
	\overline{B}_\r \cap \O_\l=\emptyset \quad  \text{ or } \quad  \{u\in \overline{B}_{\r}: \mf{F}(\l,u)=0 \}= \{0\}.
	$$
	Suppose $\overline{B}_\r \cap 	\O_\l=\emptyset$. Then, $\O_{\l} \setminus \overline{B}_{\r}=\O_{\l}$ and, hence, $(\l,u)\in \p \O$. Thus $(\l,u)\notin \mc{S}$ and, therefore, $\mf{F}(\l,u)\neq 0$, since $u\neq 0$. Now, suppose that
	\begin{equation}
		\label{6.6.24,2}
		\{u\in \overline{B}_{\r}: \mf{F}(\l,u)=0 \}= \{0\}.
	\end{equation}
	Then, $u\in \p (\O_{\l}\setminus \overline{B}_{\r}) \subset
	\p \O_{\l}\cup \p B_{\r}$. Moreover, by \eqref{6.6.24,2}, $\mf{F}(\l,u)\neq 0$ if
	$u\in\p B_{\r}$, whereas  $(\l,u)\in\p \O$ if $u\in \p \O_{\l}$. In any circumstances,  $(\l,u)\notin \mc{S}$. Therefore, 	$\mf{F}(\l,u)\neq 0$, since $u\neq 0$, which concludes the proof of \eqref{6.6.23K}.
\end{proof}
\noindent According to \eqref{6.6.23K}, setting for $1\leq k \leq m-1$,
$$ \O^{n_{k}}:=\O\cap([\l_{n_{k}} ^+,\l_{n_{k+1}} ^-]\times U), \quad \Gamma^{n_{k}}:=\O^{n_{k}}\backslash ([\l_{n_{k}} ^+,\l_{n_{k+1}} ^-]\times\overline{B}_{\r}),$$
the tern $(\mf{F},\Gamma^{n_{k}},\varepsilon)$ is a generalized $U$-admissible homotopy, $(\mf{F},\Gamma^{n_{k}},\varepsilon)\in \mathscr{G}(U)·$. Thus, we can apply Theorem \ref{T6.1.1.1} to infer that
\begin{equation*}
	\deg(\mathfrak{F}(\l_{n_{k}}^{+},\cdot),\O_{\l_{n_{k}}^{+}} \setminus \overline{B}_{\r},\varepsilon_{\l^{+}_{n_{k}}}) =
	\deg(\mathfrak{F}(\l^{-}_{n+1},\cdot),\O_{\l^{-}_{n_{k+1}}} \setminus \overline{B}_{\r},\varepsilon_{\l^{-}_{n_{k+1}}}).
\end{equation*}
Let us denote
\begin{equation*}
	\left\{\begin{array}{l}
		\deg \big(\mf{F}(\l,\cdot), \O_{\l} \setminus \overline{B}_{\r}, \varepsilon_{\l} \big) =
		d_{k}\quad \hbox{if}\;\; \l \in \{\l_{n_{k}} ^+,\l_{n_{k+1}}^{-}\}, \quad 1 \leq  k \leq m-1, \\[1ex]
		\deg \big(\mf{F}(\l_{n_{m}}^{+},\cdot), \O_{\l_{n_{m}}^{+}} \setminus \overline{B}_{\rho}, \varepsilon_{\l_{n_{m}}^{+}}  \big) =  d_m,  \\[1ex]
		\deg \big(\mf{F}(\l_{n_{1}}^{-},\cdot), \O_{\l_{n_{1}}^{-}} \setminus \overline{B}_{\r}, \varepsilon_{\l_{n_{1}}^{-}}  \big) = d_0.   \end{array}\right.
\end{equation*}
By the additivity property of the degree, we have that, for every $1\leq k\leq m$,
\begin{align*}
	\deg \big(\mf{F}(\lambda_{n_{k}}^- ,\cdot), \O_{\lambda_{n_{k}}^-}, \varepsilon_{\l_{n_{k}}^{-}} \big) & =  d_{k-1} + \deg\big(\mathfrak{F}(\lambda_{n_{k}}^-,\cdot),B_{\r},\varepsilon_{\l_{n_{k}}^{-}}\big), \\
	\deg \big(\mf{F}(\lambda_{n_{k}}^+ ,\cdot), \O_{\lambda_{n_{k}}^+}, \varepsilon_{\l^{+}_{n_{k}}}\big) & = d_{k} + \deg\big(\mathfrak{F}(\lambda_{n_{k}}^+,\cdot),B_{\r},\varepsilon_{\l_{n_{k}}^{+}}\big).
\end{align*}
Thus, by \eqref{6.6.20} and Lemma \ref{L9.4.4,2}, for sufficiently small $\rho_{0}>0$ and each $\r \in (0,\r_0]$,
\begin{align*}
	d_{k-1} - d_{k} = \deg\big(\mathfrak{F}(\lambda_{n_{k}}^+,\cdot),B_{\r},\varepsilon_{\l^{+}_{n_{k}}}\big) -  \deg\big(\mathfrak{F}(\lambda_{n_{k}}^-,\cdot),B_{\r},\varepsilon_{\l^{-}_{n_{k}}}\big) =2\mathcal{P}(I_{n_{k}}).
\end{align*}
Therefore, adding up these identities yields
\begin{equation}
	\label{6.6.27}
	d_0-d_{m}=\sum_{k=1}^{m} (d_{k- 1}-d_k) = 2 \sum_{k=1}^{m} \mathcal{P}(I_{n_{k}})=2\sum_{n\in\mc{B}(\mf{C})}\mc{P}(I_{n}).
\end{equation}
Finally, since $\mathfrak{C}$ is bounded, there exists $\lambda_{\ast}<\lambda^{-}_{n_{1}}$ and  $\lambda^{\ast}>\lambda^{+}_{n_{m}}$ such that $\Omega_{\lambda_{\ast}}=\Omega_{\lambda^{\ast}}=\emptyset$.
Thus, applying Theorem \ref{T6.1.1.1} with the generalized $U$-admissible homotopies $(\mf{F},\Gamma_{\ast},\varepsilon)$, $(\mf{F},\Gamma^{\ast},\varepsilon)\in\mathscr{G}(U)$, where
$$ \O_{\ast}:=\O\cap([\l_{\ast},\l_{n_{1}} ^{-}]\times U), \quad \Gamma_{\ast}:=\O_{\ast}\backslash ([\l_{\ast},\l_{n_{1}} ^{-}]\times\overline{B}_{\r}),$$
$$ \O^{\ast}:=\O\cap([\l_{n_{m}} ^{+},\l^{\ast}]\times U), \quad \Gamma^{\ast}:=\O^{\ast}\backslash ([\l_{n_{m}} ^{+},\l^{\ast}]\times\overline{B}_{\r}),$$
we find that
\begin{equation*}
	\begin{array}{l}
		0=\deg(\mathfrak{F}(\lambda_{\ast},\cdot),\Omega_{\lambda_{\ast}}\backslash \overline{B}_{\rho},\varepsilon_{\l_{\ast}})=\deg(\mathfrak{F}(\lambda^{-}_{n_{1}},\cdot),\Omega_{\lambda^{-}_{n_{1}}}\backslash \overline{B}_{\rho},\varepsilon_{\l^{-}_{n_{1}}})=d_{0}, \\[1ex]
		0=\deg(\mathfrak{F}(\lambda^{\ast},\cdot),\Omega_{\lambda^{\ast}}\backslash \overline{B}_{\rho},\varepsilon_{\l^{\ast}})=\deg(\mathfrak{F}(\lambda^{+}_{n_{m}},\cdot),\Omega_{\lambda^{+}_{n_{m}}}\backslash \overline{B}_{\rho},\varepsilon_{\l^{+}_{n_{m}}})=d_{m}.
	\end{array}
\end{equation*}
Since $d_{0}=d_{m}=0$, the identity \eqref{6.6.27} ends the proof.
\end{proof}

\noindent As a direct consequence of Theorems \ref{T6.2.1} and \ref{T6.3.5,2}, the following global alternative holds.

\begin{theorem}[\textbf{Global alternative}]
\label{C6.3.6}
Let $\mathfrak{F}\in\mathcal{C}^{1}(\mathbb{R}\times U,V)$ be a map satisfying {\rm (F1)--(F4)} and $\mc{J}:=\{ J_n \}_{n=r}^s$ be an admissible ordered family for $\Sigma(\mf{L})$. Suppose there exists $n\in\Z\cap[r+1,s]$ such that, for sufficiently small $\delta>0$,
\begin{equation}
	\label{6.6.10,2}
	\chi[\mathfrak{L}_{\omega},[\lambda_{-},\lambda_{+}]]\in 2\mathbb{N}+1, \quad \l_{-}=\inf I_{n}-\delta, \ \l_{+}=\sup I_{n} + \delta,
\end{equation}
where $\mathfrak{L}_{\omega}\in\mathscr{C}^{\omega}([\lambda_{-},\lambda_{+}],\Phi_{0}(U,V))$ is any  analytic curve $\mc{A}$-homotopic to $\mf{L}(\l)$, $\lambda\in[\lambda_{-},\lambda_{+}]$. Then, there exists a component $\mf{C}$ of the set of non-trivial solutions $\mc{S}_{\mc{J}}$ such that $\mf{C}\cap (I_{n}\times \{0\})\neq \emptyset$. Moreover, one of the following non-excluding alternatives occur:
\begin{enumerate}
	\item[{\rm (i)}] $\mf{C}$ is unbounded.
	\item[{\rm (ii)}] There exists $m\in\Z\cap[r+1,s]$, $m\neq n$, such that $\mf{C}\cap (I_{m}\times \{0\})\neq \emptyset$.
\end{enumerate}
\end{theorem}

\begin{proof}
Since the hypothesis of Theorem \ref{T6.2.1} hold, we infer the existence of a component $\mf{C}$ of $\mc{S}_{\mc{J}}$ such that $(\l_{0},0)\in \mf{C}$ for some $\l_{0}\in(\l_{-},\l_{+})$. Since for sufficiently small $\delta>0$,
$$\mf{L}(\l)\in GL(U,V), \quad \l\in [\l_{-},\inf I_{n}]\cup [\sup I_{n},\l_{+}],$$
necessarily, $\l_{0}\in I_{n}$. Hence, $\mf{C}\cap (I_{n}\times \{0\})\neq \emptyset$ and we have proved the first part of the theorem. Let us prove the second part. If $\mf{C}$ is unbounded, then (i) holds. Suppose $\mf{C}$ is bounded. Then, by Theorem \ref{T6.3.5,2},
\begin{equation}
	\label{6.6.18,2}
	\sum_{\ell\in \mc{B}(\mf{C})} \mathcal{P}(I_\ell)=0,
\end{equation}
where $\mathcal{P}$ is the $\mc{J}$-parity map. By \eqref{eqqLl,2}, we obtain
\begin{equation}
	\label{eqqLl,2.1}
	2\mc{P}(I_{n})=\varepsilon_{\l_{-}}(0)\left[\sigma(\mf{L},[\lambda_{-},\lambda_{+}])-1\right].
\end{equation}
On the other hand, by Theorem \ref{th3.2} and \eqref{6.6.10,2}, we find that,
for every analytic curve $\mathfrak{L}_{\omega}\in\mathscr{C}^{\omega}([\lambda_{-},\lambda_{+}],\Phi_{0}(U,V))$ $\mc{A}$-homotopic to $\mf{L}(\l)$, $\lambda\in[\lambda_{-},\lambda_{+}]$,
$$
   \sigma(\mf{L},[\lambda_{-},\lambda_{+}])=(-1)^{\chi[\mathfrak{L}_{\omega},[\lambda_{-},\lambda_{+}]]}=-1.
$$
Consequently, by \eqref{eqqLl,2.1}, we deduce that $\mc{P}(I_{n})=-\varepsilon_{\l_{-}}(0)\neq 0$. Since $n\in\mc{B}(\mf{C})$, according to  \eqref{6.6.18,2}, there exists $m\in\mc{B}(\mf{C})$, $m\neq n$. (ii) holds if $\mf{C}$ is bounded. This ends the proof.
\end{proof}

Under the general assumptions of the local theorem of Crandall and Rabinowitz \cite{CR},
$\l_0$ is an isolated eigenvalue of $\mf{L}(\l)$. Thus, setting $\l_\pm =\l_0\pm \d$ for sufficiently small $\d$ and using \eqref{ii.4}, we are lead to
$$
\chi[\mathfrak{L}_{\omega},[\lambda_{-},\lambda_{+}]]=
\chi[\mathfrak{L},[\lambda_{-},\lambda_{+}]]= \chi[\mathfrak{L},\l_0]=
\dim N[\mf{L}_0]=1.
$$
Therefore, the local theorem of \cite{CR} is actually global.
Although Shi and Wang \cite{XW} observed that
in the setting of \cite{CR} the global alternative of Rabinowitz also holds, Theorem \ref{T6.3.5,2}
is a substantially sharper result, as it establishes, in addition, the validity of \eqref{6.6.18}, which is
a very sophisticated condition of global topological nature. The validity of this condition in the context of the Leray--Schauder degree goes back to Nirenberg \cite{Ni} and Magnus \cite{Ma}. Rather paradoxically, the simplicity of the global alternative of Rabinowitz and the
topological technicalities of the underlying theory hid for almost 50 years the much stronger
condition \eqref{6.6.18} until \cite{LGMC05} was published (see \cite{LG02} for further details).

\section{Sharp local bifurcation analysis from simple degenerate eigenvalues}\label{SLBA}

\noindent This section is devoted to the study of bifurcation from simple degenerate eigenvalues for analytic nonlinearities. We begin by recalling some  basic concepts for analytic maps between Banach spaces. Given $\K\in\{\R,\C\}$, an integer $n\geq 1$,  and $n+1$ $\K$-Banach spaces $U_{1}, U_{2}, \cdots, U_{n}$, $V$, a map
$L: U_{1}\times \cdots \times U_{n} \to V$ is said to be  $\K$-\textit{multilinear} if it is $\K$-linear in each variable $u_i\in U_i$ for all $i=1,...,n$. Naturally, in such case, $L$ is said to be bounded if
$$
   \|L\|:=\sup \{\|L(u_{1},\cdots, u_{n})\|  : \|u_{1}\|, \cdots, \|u_{n}\|\leq 1\}<\infty.
$$
The space of the bounded $\K$-multilinear operators is denoted by $\mc{M}(U_{1},\cdots, U_{n};V)$, and we simply set  $\mc{M}(U_1,\cdots, U_n; V)\equiv \mc{M}^{n}(U,V)$ if $U_{1}=U_{2}=\cdots=U_{n}=U$. An operator
$L\in \mc{M}(U_{1},\cdots,U_{n}; V)$ is called \textit{symmetric} if, for every permutation $\sigma\in \Sigma_{n}$,
$$
   L(u_{1},u_{2},\cdots, u_{n})=L(u_{\sigma(1)},u_{\sigma(2)},\cdots, u_{\sigma(n)}),
$$
where $\Sigma_{n}$ stands for the symmetric group of permutations of $\{1,\cdots,n\}$. In this section, the space of bounded $\K$-multilinear symmetric operators is denoted by $\mc{S}(U_{1},\cdots,U_{n};V)$, and we set $\mc{S}(U,\cdots,U;V)=\mc{S}^n(U;V)$. Accordingly, for every $L\in \mc{S}^{n}(U,V)$, we set
\begin{equation}
	\label{abreviation}
	L(u,u,\cdots,u)\equiv Lu^{n}, \quad u\in U.
\end{equation}
Given a pair of $\K$-Banach spaces $(U,V)$ and a point $u_{0}\in U$, a map $\mf{F}: U \to V$ is said to be $\K$-\textit{analytic} at $u_{0}$ if there exist a neighbourhood $\mc{U}_{u_{0}}$ of $u_{0}$ in $U$ and $r>0$ such that
$$
\mf{F}(u)=\sum_{n=0}^{\infty}L_{n}(u-u_{0})^{n}, \quad u\in \mc{U}_{u_{0}},\qquad
\sup_{n\geq 0} \left( r^{n} \|L_{n}\|\right) <\infty,
$$
where $\mf{F}(u_{0})=L_{0}\in V$ and $L_{n}\in\mc{S}^{n}(U, V)$ for all  $n\in\N$.
A map $\mf{F}: U\to V$ is said to be $\K$-analytic if it is $\K$-analytic at every point $u\in U$. In this section, the space of analytic functions $U\to V$ is denoted by $\mc{C}^{\omega}(U,V)$, and we simply say that $\mf{F}$ is analytic, without specifying where, if there is no ambiguity.

\subsection{Analytic Lyapunov--Schmidt reduction}\label{SectionALS}
In this section, we will perform a general Lyapunov--Schmidt reduction for analytic nonlinearities. Precisely, we consider a field $\K\in\{\R,\C\}$, an analytic map $\mf{F}\in\mc{C}^{\omega}(\K\times U, V)$, and a point $(\l_{0},u_{0})\in \mf{F}^{-1}(0)$ such that $D_{u}\mf{F}(\l_{0},u_{0})\in \Phi_{0}(U,V)$. For any given pair $(P,Q)$  of $D_{u}\mathfrak{F}(\lambda_{0},u_{0})$-projections,
$P:U\to N[D_{u}\mathfrak{F}(\lambda_{0},u_{0})]$, $Q:V\to R[D_{u}\mathfrak{F}(\lambda_{0},u_{0})]$,
we can decompose
\begin{align*}
	U= N[D_{u}\mathfrak{F}(\lambda_{0},u_{0})] \oplus Y \quad (Y\equiv N[P]), \qquad V =Z \oplus R[D_{u}\mathfrak{F}(\lambda_{0},u_{0})]  \quad  (Z\equiv N[Q]).
\end{align*}
and identify $\mathbb{K}\times N[D_{u}\mathfrak{F}(\lambda_{0},u_{0})]$ with $\mathbb{K}^{n+1}$, where
$$n:=\dim N[D_{u}\mathfrak{F}(\lambda_{0},u_{0})]=\codim R[D_{u}\mathfrak{F}(\lambda_{0},u_{0})],$$
via the linear isomorphism
\begin{equation}
	\label{OperadorT}
	T:\mathbb{K}\times N[D_{u}\mathfrak{F}(\lambda_{0},u_{0})] \to \K\times \K^{n}, \quad (\l,x)\mapsto (\l, Lx).
\end{equation}
Similarly, one can identify $Z$ with $\K^{n}$ via another linear isomorphism
$S: Z \longrightarrow \K^{n}$, because $\dim Z=n$. Since every $u\in U$ admits a unique
decomposition as
\begin{equation}
	\label{ProyeccionesGenerales}
	u = u_{0}+x + y, \quad x = P [u-u_{0}], \ \ y = (I_U-P)[u-u_{0}],
\end{equation}
the equation $\mf{F}(\l,u)=0$ is equivalent to the system
\begin{equation}
	 Q \mathfrak{F}(\l,u_{0}+x+y)=	0,  \quad (I_V-Q)\mathfrak{F}(\l,u_{0}+x+y)=  0.
\label{(3.1.3,0)}
\end{equation}
Consider the operator
\begin{equation*}
	\mathfrak{H}\, :\, \K \times N[D_{u}\mathfrak{F}(\lambda_{0},u_{0})]\times Y \to V, \quad \mathfrak{H}(\l,x,y):=Q \mathfrak{F}(\l,u_{0}+x+y).
\end{equation*}
This operator is analytic and it satisfies $\mathfrak{H}(\l_0,0,0)=0$.
Moreover, its linearization
$$
D_y\mathfrak{H}(\l_0,0,0) = Q D_{u}\mathfrak{F}(\lambda_{0},u_{0})|_Y\, :\,
Y \to R[D_{u}\mathfrak{F}(\lambda_{0},u_{0})]
$$
is an isomorphism. Thus, by the implicit function theorem, there exist a neighborhood $\mc{U}$ of $(\l_0,0)$ in $\K\times N[D_{u}\mathfrak{F}(\lambda_{0},u_{0})]$ and an analytic map
$\mc{Y}:\mc{U}\to Y$ such that
\begin{equation}
	\mathfrak{H}(\l,x,{\mc Y}(\l,x))=0 \quad\hbox{for all}\;\; (\l,x)\in\mc{U}.
	\label{(3.1.4,0)}
\end{equation}
In particular $\mc{Y}(\l_{0},0)=0$. Moreover, there exists a neighbourhood
$\mc{O}$ of $(\l,u)=(\l_0,u_{0})$ in
$
\K\times U
$
such that if
$(\l,u)=(\l,u_{0}+x+y)\in \mc{O}$ and $\mathfrak{H}(\l,x,y)=0$, then $y = {\mc Y}(\l,x)$. Finally, substituting ${\mc Y}(\l,x)$ into the second equation of the system \eqref{(3.1.3,0)} yields
\begin{equation}
	(I_V - Q)\mathfrak{F}(\l,u_{0}+x+{\mc Y}(\l,x))=0, \quad (\l,x)\in \mc{U}.
	\label{(3.1.7,0)}
\end{equation}
Therefore, $(\l,x)\in\mc{U}$ is a solution of \eqref{(3.1.7,0)} if, and only if,
$(\l,u)=(\l,u_{0}+x+{\mc Y}(\l,x))\in\mc{O}$ satisfies $\mf{F}(\l,u)=0$. Consequently, considering the open set $\O:=T(\mc{U})$, where $T$ is the isomorphism \eqref{OperadorT}, we have reduced the equation $\mf{F}(\l,u)=0$ in the neighbourhood $\mc{O}$ of $(\l_{0},u_{0})$, to the problem of finding out the zeroes of the finite dimensional analytic map $\mathfrak{G}: \O\subset\mathbb{K}\times \K^{n} \to \mathbb{K}^{n}$ defined by
\begin{equation}
	\label{LSRANA}
	\mf{G}(\lambda,z):=S(I_{V}-Q)\mathfrak{F}(\lambda,u_{0}+L^{-1}z+\mathcal{Y}(\lambda,L^{-1}z)).
\end{equation}
In particular $\mf{G}(\l_{0},0)=0$. Consequently, the following result holds.

\begin{theorem}
	\label{Lyap-Smith}
	Let $\mf{F}\in\mc{C}^{\omega}(\K\times U, V)$ and $(\l_{0},u_{0})\in\mf{F}^{-1}(0)$ with  $D_{u}\mf{F}(\l_{0},u_{0})\in \Phi_{0}(U,V)$. Then, there exists a neighbourhood $\mc{O}$ of $(\l_{0},u_{0})$ in $\K\times U$ such that the maps
\begin{equation}
\label{juca1}
\begin{split}
		\Psi:\mathfrak{F}^{-1}(0)\cap \mc{O}\longrightarrow \mathfrak{G}^{-1}(0), &  \quad (\lambda,u)\mapsto (\lambda,LP(u-u_{0})),\\  	\Psi^{-1}:\mathfrak{G}^{-1}(0)\longrightarrow \mathfrak{F}^{-1}(0)\cap\mc{O}, & \quad (\lambda,z)\mapsto (\lambda,u_{0}+L^{-1}z+\mathcal{Y}(\lambda,L^{-1}z)),
\end{split}
\end{equation}
are mutually inverses, where $\mf{G}:\O\to\K^{n}$ is given by \eqref{LSRANA}. Moreover, for every $(\l,z)\in \mf{G}^{-1}(0)$, the following statements are equivalent:
	\begin{enumerate}
		\item[{\rm (1)}] $D_{z}\mathfrak{G}(\l,z)\neq 0$.
		\item[{\rm (2)}] $D_{u}\mathfrak{F}(\lambda, L^{-1}z+\mathcal{Y}(\lambda,L^{-1}z))\in GL(U,V)$.
	\end{enumerate}
\end{theorem}

\subsection{Sharp local bifurcation analysis from simple degenerate eigenvalues}\label{SSLBA}

\noindent In this section we ascertain the local structure of
the solution sets for analytic nonlinearities at simple degenerate eigenvalues. Throughout it, we
assume that $\mathbb{K}\in\{\mathbb{R},\mathbb{C}\}$, that $U$ and $V$ are $\mathbb{K}$-Banach spaces, and that $\mathfrak{F}\in\mathcal{C}^{\omega}(\mathbb{K}\times U,V)$ satisfies the following assumptions:
\begin{enumerate}
	\item[(F1)] $\mathfrak{F}(\lambda,0)=0$ for every $\lambda\in\mathbb{K}$.
	\item[(F2)] $\mf{L}(\l):=D_{u}\mathfrak{F}(\lambda,0)\in\Phi_{0}(U,V)$ for all $\lambda\in\mathbb{K}$.
	\item[(F3)] $\l_{0}\in\Sigma(\mf{L})$ is an isolated eigenvalue such that
	$$N[\mf{L}_{0}]=\mathrm{span}[\varphi_{0}] \quad \text{ for some } \varphi_0\in U\backslash\{0\}.$$
\end{enumerate}
Subsequently, we denote by $\langle \cdot,\cdot\rangle: U\times U^{\ast}\to\mathbb{K}$ the duality pairing between $U$ and its topological dual space $U^{\ast}$. By the Hahn--Banach theorem, there exists $\varphi_{0}^{\ast}\in U^{\ast}$ such that $\langle \varphi_{0},\varphi_{0}^{\ast}\rangle=1$. Let us consider a pair $\mc{P}=(P,Q)$ of $\mf{L}_{0}$-projections, $P:U\to N[\mathfrak{L}_{0}]$ and $Q:V\to R[\mathfrak{L}_0]$, where $P$ is given by
\begin{equation}
\label{juca2}
   P(u):=\langle u,\varphi_{0}^{\ast}\rangle \varphi_{0} \quad\hbox{for all}\;\; u\in U.
\end{equation}
Then, we have the topological direct sum decompositions
\begin{align*}
	U = N[\mathfrak{L}_{0}] \oplus Y\quad (Y\equiv N[P]), \quad
	V = Z \oplus R[\mf{L}_{0}] \quad (Z\equiv N[Q]).
\end{align*}
In the sequel, we identify $\mathbb{K}\times N[\mathfrak{L}_{0}]$ with $\mathbb{K}^{2}$ via the
linear isomorphism
\begin{equation}
	\label{Identification}
	T:\mathbb{K}\times N[\mathfrak{L}_{0}]\longrightarrow \K^{2}, \quad T(\lambda,z\varphi_{0})=(\lambda,z),
\end{equation}
and $Z$ with $\K$ via another linear isomorphism, $S:Z\to \K$, whose expression is not relevant. Then, according to the results of Section \ref{SectionALS} applied to $\mf{F}(\l,u)=0$ on $(\l_{0},0)$
with the $\mf{L}_{0}$-projections $(P,Q)$, there exist two neighbourhoods, $\mc{U}\subset \K\times N[\mf{L}_{0}]$, $(\l_{0},0)\in \mc{U}$, $\mc{O}\subset \K\times U$, $(\l_0,0)\in \mc{O}$, and an analytic operator $\mc{Y}:\mc{U}\to Y$ such that
the maps \eqref{juca1} with $LP(u-u_0)=  \langle u, \varphi_{0}^{\ast}\rangle$ and $L^{-1}z=z\v_0$, are mutually inverse. Note that now $\mathfrak{G}$ is given by \eqref{LSRANA} with
$L^{-1}z=z\v_0$ and $\O=T(\mc{U})$. Without loss of generality, we can assume that $(\lambda_{0},0)=(0,0)$. Since $\mathfrak{G}$ is analytic and $\mathfrak{G}(\lambda,0)=0$, it can be expressed as
\begin{equation*}
	\mathfrak{G}(\lambda,z)=\sum_{i\geq 0,\; j\geq 1} a_{ij}\lambda^{i} z^{j}, \qquad (\lambda,z)\sim (0,0),
\end{equation*}
for some coefficients $a_{ij}\in\mathbb{K}$, $(i,j)\in\mathbb{Z}_{+}^{2}$, $j\neq 0$. Thus, there exists an analytic function $g:\Omega\to \mathbb{K}$ such that
\begin{equation}
	\label{IddPOLYN}
	\mathfrak{G}(\lambda,z)=z\sum_{i\geq 0,\; j\geq 1} a_{ij}\lambda^{i} z^{j-1}=zg(\lambda,z), \qquad (\lambda,z)\sim (0,0).
\end{equation}
The next lemma shows how the algebraic multiplicity $\chi$ of $\mf{L}(\l)$ is related to $\mf{G}$.
\begin{lemma}
	\label{Lemma12.1.2}
	$\chi[\mathfrak{L},0]=\ord_{\lambda=0}D_{z}\mathfrak{G}(\lambda,0)=\ord_{\lambda=0} g(\lambda,0)$.
\end{lemma}
\begin{proof}
Since $\mf{F}$ is analytic, $\mf{L}\in\mc{C}^{\omega}(\K,\Phi_{0}(U,V))$. Thus, by hypothesis (F3), $\l_{0}=0$ is an isolated eigenvalue. Hence, by Theorems 4.4.1 and 4.4.4 of \cite{LG01},  $0\in\Alg(\mf{L})$. Therefore, $\chi[\mf{L},0]$ is well defined. On the other hand, by \cite[Th. 1.2]{JJ3}, we have that $\chi[\mathfrak{L},0]=\ord_{\l=0} D_{z}\mf{G}(\l,0)$. The proof is complete.
\end{proof}
\noindent Combining Lemma \ref{Lemma12.1.2} with the Weierstrass preparation theorem (see \cite[Th. 5.3.1]{BT}), there exist a open neighborhood $\mathscr{U}=\mathscr{U}_{\lambda}\times\mathscr{U}_{z}$ of $(0,0)$ in $\Omega$, an analytic function $c:\mathscr{U}\to\mathbb{K}$ such that  $c(0,0)\neq 0$, and
$\chi:=\chi[\mf{L},0]\geq 1$ analytic functions $c_{j}:\mathscr{U}_{z}\to\mathbb{K}$ such that $c_{j}(0)=0$
for all $0\leq j\leq \chi-1$, and
\begin{equation*}
	g(\lambda,z)=c(\lambda,z)\left[\lambda^{\chi}+c_{\chi-1}(z)\lambda^{\chi-1}+\cdots
	+c_{1}(z)\lambda+c_{0}(z) \right].
\end{equation*}
The factoring monic polynomial
\begin{equation}
	\label{6.6.1000}
	p(\l,z):= \lambda^{\chi}+c_{\chi-1}(z)\lambda^{\chi-1}+\cdots
	+c_{1}(z)\lambda+c_{0}(z)
\end{equation}
is often called the \emph{Weierstrass polynomial} of $g(\l,z)$.
These results can be summarized into the following theorem.

\begin{theorem}
	\label{T6.4.1}
	Let $\mf{F}\in\mc{C}^{\omega}(\K\times U, V)$ be a map satisfying {\rm (F1)--(F3)}. Then, there exist a neighbourhood $\mathscr{U}$ of $(0,0)$ in $\mathbb{K}^{2}$, an analytic function $c:\mathscr{U}\to\mathbb{K}$, with $c(0,0)\neq 0$, and $\chi=\chi[\mathfrak{L},0]\geq 1$ analytic functions $c_{j}:\mathscr{U}_{z}\to\mathbb{K}$, $0\leq j\leq \chi-1$, with $c_{j}(0)=0$, such that
\begin{equation*}
\mathfrak{G}(\lambda,z) \equiv zc(\lambda,z)(\lambda^{\chi}+c_{\chi-1}(z)\lambda^{\chi-1}
+\cdots+c_{1}(z)\lambda+c_{0}(z)),\quad (\l,z)\in\mathscr{U},
\end{equation*}
satisfies $\mf{G}(\l,z)=0$ 	if and only if $(\lambda,u)=(\lambda,z\varphi_{0}+\mathcal{Y}(\lambda,z\varphi_{0}))\in \mathfrak{F}^{-1}(0)$.
Therefore, the associated Weierstrass polynomial \eqref{6.6.1000} provides us with the local structure of $\mathfrak{F}^{-1}(0)$ at $(0,0)\in\mathbb{K}\times U$.
\end{theorem}

\noindent This result reduces the analysis of the local structure of the solutions of the infinite dimensional problem $\mathfrak{F}(\lambda,u)=0$ to the analysis of the zeros of the Weierstrass polynomial $p(\l,z)$. According to Theorem \ref{T6.4.1}, in the simplest case when $\chi=1$, the local structure of $\mf{F}^{-1}(0)$ is determined by
\begin{equation*}
	\mathfrak{G}(\lambda,z)=zc(\lambda,z)(\lambda+c_{0}(z))=0, \quad (\lambda,z)\in\mathscr{U}.
\end{equation*}
Since $c$ is analytic and $c(0,0)\neq 0$, shortening the neighborhood $\mathscr{U}$, if necessary, we can assume that $c(\lambda,z)\neq 0$ for all $(\lambda,z)\in\mathscr{U}$. Thus, $\mathfrak{G}(\lambda,z)=0$ if and only if either $z=0$, or $\lambda=-c_{0}(z)$, $z\in\mathscr{U}_{z}$, which provides us with the analytic counterpart of the main theorem of Crandall and Rabinowitz \cite{CR}. Throughout the rest of this section, we will distinguish between two different cases, according to the nature of $\mathbb{K}$.

\subsection{The complex case $\mathbb{K}=\mathbb{C}$}
Then, given a domain $\Omega$ of $\mathbb{C}$, we will denote by $\mc{O}(\O)$ and $\mathcal{M}(\O)$, the spaces of holomorphic and meromorphic functions defined on $\O$, respectively, and we denote by $\mathcal{M}(\O)[\lambda]$ (resp. $\mathcal{O}(\O)[\lambda]$) the space of polynomials in $\lambda$ with coefficients in the space of meromorphic (resp. holomorphic) functions on $\O$. Let
$\mathscr{U}=\mathscr{U}_{\lambda}\times\mathscr{U}_{z}$ be the open set of $\C^{2}$ whose existence was established after the proof of Lemma \ref{Lemma12.1.2}. Since $\mathcal{M}(\mathscr{U}_{z})$ is a field, $\mathcal{M}(\mathscr{U}_{z})[\lambda]$ is a Unique Factorization Domain (UFD). Thus,  there exists an integer $n\geq 1$ and $n$ monic irreducible polynomials in  $\mathcal{M}(\mathscr{U}_{z})[\lambda]$, say  $p_{i}(\lambda,z)$, with $p_{i}(\lambda,0)=\lambda^{\deg p_{i}}$, $i\in\{1,...,n\}$, such that
\begin{equation*}
	p(\lambda,z)=\prod_{i=1}^{n}p_{i}(\lambda,z),
\end{equation*}
where $p(\l,z)$ is the Weierstrass polynomial \eqref{6.6.1000}. Subsequently, we  set $d_i\equiv \deg p_i$ for all $i\in\{1,...,n\}$. Since $p_{i}(\l,z)\in\mc{M}(\mathscr{U}_{z})[\l]$, we can express
\begin{equation*}
	p_{i}(\l,z)=\l^{d_{i}}+c_{i,d_{i}-1}(z)\l^{d_{i}-1}+\cdots+c_{i,1}(z)\l+c_{i,0}(z), \quad z\in\mathscr{U}_{z},
\end{equation*}
for some meromorphic functions $c_{i,k}:\mathscr{U}_{z}\to \C$ such that $c_{i,k}(0)=0$, $1\leq i\leq n$,
$1\leq k \leq d_{i}-1$. Actually, shrinking $\mathscr{U}$, if necessary, we can suppose that each $c_{i,k}\in\mc{O}(\mathscr{U}_{z})$. Thus, $p_{i}(\l,z)\in\mc{O}(\mathscr{U}_{z})[\l]$.
\par
Suppose that $d_i\geq 2$ for some $i\in\{1,...,n\}$,  and let us denote by $\Delta_{i}(z)$, $z\in\mathscr{U}_{z}$,  the discriminant of the polynomial $p_i(\lambda,z)$. Since $p_i(\lambda,0)=\lambda^{d_i}$ with $d_i\geq 2$, $\l=0$ is, at least,  a double root. Hence, $\Delta_{i}(0)=0$. Moreover, $\Delta_{i}\not\equiv 0$, because $p_{i}(\lambda,z)$ is irreducible in $\mathcal{M}(\mathscr{U}_{z})[\lambda]$. Thus, since the discriminant is analytic, shortening $\mathscr{U}$ if necessary, we can assume that $\Delta_{i}(z)\neq 0$ for all $z\in\mathscr{U}_{z}\backslash\{0\}$.
Consequently, for every $z\in\mathscr{U}_{z}\backslash\{0\}$, the irreducible factor $p_{i}(\lambda,z)$ admits $d_i$ simple  roots. Pick $z\in\mathscr{U}_{z}\backslash\{0\}$. Then, by \cite[Cor. 8.8]{OF}, there exists a neighborhood $\mathscr{V}$ of $z$ in $\mathscr{U}_{z}\backslash\{0\}$, and $d_i$ holomorphic functions, $\varphi_{j}\in\mathcal{O}(\mathscr{V})$, $j\in\{1,...,d_i\}$, such that
\begin{equation*}
	p_{i}(\lambda,z)=\prod_{j=1}^{d_i}(\lambda-\varphi_{j}(z)), \quad z\in\mathscr{V}.
\end{equation*}
Consequently, the set $\mathcal{S}_{i}:=p_i^{-1}(0)\cap[\mathbb{C}\times (\mathscr{U}_z\setminus\{0\})]$ consists, locally, of the analytic curves $(\lambda,z)=(\varphi_j(z),z)$,
$z\in\mathscr{V}$, $j\in\{1,...,d_i\}$, and therefore, $\mathcal{S}_{i}$ can be endowed with the structure of an open Riemann surface, i.e., an open one-dimensional smooth complex manifold. Obviously, this is also true if $d_1=1$. So, it holds for all $i\in\{1,...,n\}$. The next result establishes that
$(0,0)$ is  an accumulation point of $\mathcal{S}_{i}\subset \mathbb{C}^{2}$ for all $i\in\{1,...,n\}$.
According to Theorem \ref{T6.4.1}, in the complex setting, $(0,0)\in\mathbb{C}\times U$ always is a bifurcation point of $\mf{F}(\l,u)=0$, regardless the value of $\chi$.

\begin{lemma}
$(0,0)\in\C^{2}$ is an accumulation point of $\mc{S}_{i}$ for all $i\in\{1,...,n\}$.
\end{lemma}
\begin{proof}
We claim that, for every $(\l,z)\in \C\times\mathscr{U}_{z}$ such that $p_{i}(\l,z)=0$,
\begin{equation}
\label{polyy0}
		|\l|\leq \max\Big\{1, \sum_{j=0}^{d_{i}-1}\|c_{i,j}\|_{L^\infty(\mathscr{U}_{z})} \Big\}\equiv \mathcal{R}.
\end{equation}
Note that $\mathscr{U}_{z}$ can be shortened so that $c_{i,j}\in\mathcal{C}(\mathscr{\bar U}_{z})$. Suppose $p_{i}(\l,z)=0$. Then,
\begin{equation}
		\label{polyy}
		\l^{d_{i}}=-\sum_{j=0}^{d_{i}-1}c_{i,j}(z)\l^{j}.
\end{equation}
Thus, \eqref{polyy0} holds if $|\l|\leq 1$. Suppose $|\l|>1$. Then, \eqref{polyy} implies that
\begin{equation}
		\label{polyy2}
		|\l|^{d_{i}}\leq |\l|^{d_{i}-1} \sum_{j=0}^{d_{i}-1}\|c_{ij}\|_{L^\infty(\mathscr{U}_{z})},
\end{equation}
and dividing \eqref{polyy2} by $|\l|^{d_{i}-1}$ yields \eqref{polyy0}. On the other hand, since the domain $\mathscr{W}:=\{z\in\C : \mf{R}z>0\}\cap\mathscr{U}_z$ is simply connected, by analytic continuation, it follows from  \cite[Cor. 8.8]{OF} that there exists a holomorphic function $\varphi:\mathscr{W}\to \C$ such that $p_{i}(\varphi(z),z)=0$ for all $z\in \mathscr{W}$.
In particular, we have that
\begin{equation}
		\label{Polyy}
		\lim_{z\to 0}p_{i}(\varphi(z),z)=0.
\end{equation}
Moreover, by \eqref{polyy0}, $|\varphi(z)|\leq \mc{R}$ for all $z\in\mathscr{W}$. Let $(z_{n})_{n\in\N}$ be a sequence in $\mathscr{W}$ such that $z_{n}\to 0$ as $n\to\infty$ and set $w_{n}:=\varphi(z_{n})$, $n\geq 1$.
Since $|w_{n}|\leq \mc{R}$ for all $n\in\N$, by compactness, there exists a subsequence, relabeled by $n$, such that $w_{n}\to w_{0}$ as $n\to\infty$ with $|w_{0}|\leq \mc{R}$. By \eqref{Polyy}, $p_{i}(w_{0},0)=0$. Since $p_{i}(\l,0)=\l^{d_{i}}$, necessarily $w_{0}=0$. Therefore, $(0,0)$ is an accumulation point of $\mc{S}_{i}$.
\end{proof}

\par
To gain some insight into the the local structure of the zeros near $(0,0)$, we consider the $z$-projection operator  $\pi_{i}:\mathcal{S}_{i}\longrightarrow\mathscr{U}_{z}\backslash\{0\}$, $(\lambda,z)\mapsto z$. By construction, $\pi_{i}$ is a $d_i$-sheeted holomorphic covering map whose associated  fiber at $z\in\mathscr{U}_{z}\backslash\{0\}$ is the set of simple roots of $p_{i}(\lambda,z)=0$. To study the singularity at $(0,0)$ or, equivalently, to add the multiple root $(0,0)$ of $p_{i}(\lambda,z)$ to the Riemann surface $\mathcal{S}_{i}$, one can extend $\mathcal{S}_{i}$ to a Riemann surface, $\mc{X}_{i}$, with a branched covering map $\Pi_{i}:\mc{X}_{i}\to\mathscr{U}_{z}$ such that $\Pi_{i}|_{\mc{S}_{i}}=\pi$, by means of the following classical theorem of B. Riemann (see, e.g., \cite[Th. 8.4]{OF} for a proof).

\begin{theorem}
Let $\O$ be a domain in $\C$ and $z_{0}\in\O$. Suppose $X$ is a Riemann surface and
$\pi:X \longrightarrow \O\backslash\{z_{0}\}$, is a proper holomorphic covering map. Then, $\pi$ extends to a branched covering of $\O$. In other words, there exists a Riemann surface $\mc{X}$, with $X\subset \mc{X}$, and a proper holomorphic map $\Pi:\mc{X}\to\O$ such that $\mc{X}\backslash \Pi^{-1}(z_{0})=X$ and $\Pi|_{X}=\pi$.
\end{theorem}

For every $i\in\{1,...,n\}$, the extended $\mc{X}_{i}$ is called the Riemann surface of the irreducible factor $p_{i}(\lambda,z)\in\mathcal{M}(\mathscr{U}_{z})[\lambda]$, and $\Pi_{i}:\mc{X}_{i}\to\mathscr{U}_{z}$ is refereed to  as the $d_i$-sheeted branched covering map associated to it. In particular $\mathcal{S}_{i}=\mc{X}_{i}\backslash\Pi_{i}^{-1}(0)$. Therefore, the local structure of $\mathfrak{G}^{-1}(0)$ in a neighborhood of  $(0,0)\in\mathbb{C}^{2}$ is completely determined by the pairs $(\mc{X}_{i},\Pi_{i})$, $i\in\{1,...,n\}$. The internal  structure of the pairs $(\mc{X}_{i},\Pi_{i})$ can further be analysed through symmetries via the concept of \textit{deck transformation}. A deck transformation associated to the pair $(\mc{X}_{i},\Pi_{i})$ is a biholomorphic map $\varphi: \mc{X}_{i}\to \mc{X}_{i}$ making the following diagram commutative:
\begin{equation}
	\xymatrix@C+1em{
		\mc{X}_i \ar^{\varphi}[rr] \ar_{\Pi_{i}}[dr] & & \mc{X}_{i} \ar^{\Pi_{i}}[dl] \\
		&\mathscr{U}_{z}&
	}
\end{equation}
Let us denote by $\Deck(\mc{X}_{i},\Pi_{i})$ the group of the deck transformations associated to  $(\mc{X}_{i},\Pi_{i})$. This group gives information about the nature of the singularity at $(0,0)$. To study the structure of $\Deck(\mc{X}_{i},\Pi_{i})$, we need some preliminaries on field extensions. Consider the field extensions $\mathcal{M}(\mathscr{U}_{z})\subset \mathcal{M}(\mc{X}_{i})$, $i\in\{1,...,n\}$, where we are denoting by $\mathcal{M}(\mc{X}_{i})$ the field of meromorphic functions defined on $\mc{X}_{i}$. As,  due to \cite[Th. 8.12]{OF}, the field isomorphism
\begin{equation}
	\label{ISOmK}
	\mathcal{M}(\mc{X}_{i})\simeq \frac{\mathcal{M}(\mathscr{U}_{z})[\lambda]}{(p_{i}(\lambda,z))},
\end{equation}
holds, the extensions $\mathcal{M}(\mathscr{U}_{z})\subset \mathcal{M}(\mc{X}_{i})$ are determined by the algebraic field extensions
\begin{equation}
	\label{6.6.37}
	\mathcal{M}(\mathscr{U}_{z})\subset \frac{\mathcal{M}(\mathscr{U}_{z})[\lambda]}{(p_{i}(\lambda,z))}.
\end{equation}
Some elementary algebra shows that
\begin{equation*}	 \frac{\mathcal{M}(\mathscr{U}_{z})[\lambda]}{(p_{i}(\lambda,z))}=\left\{\sum_{j=0}^{d_i-1} a_{j}(z)\lambda^{j}: \;a_{j}\in\mathcal{M}(\mathscr{U}_{z}),\; j\in\{1,2,\cdots,d_{i}-1\}\right\}\Big{/}\sim,
\end{equation*}
where $\deg p_{i}=d_{i}$ and $f_1\sim f_2$ if and only if $f_1-f_2\in (p_{i}(\lambda,z))$. Thus, it is apparent that
\begin{equation*}	
   \left[\frac{\mathcal{M}(\mathscr{U}_{z})[\lambda]}{(p_{i}(\lambda,z))}:
	\mathcal{M}(\mathscr{U}_{z})\right]=d_i=\deg p_{i}.
\end{equation*}
As usual in field theory, $[L:K]$ stands for the degree of any field extension $K\subset L$, i.e., the dimension of $L$  viewed as a $K$-vector space. Finally, let us associate to the field extensions $\mathcal{M}(\mathscr{U}_{z})\subset \mathcal{M}(\mc{X}_{i})$  their associated Galois groups
$\mathscr{G}_{i}(\mathcal{M}(\mc{X}_{i}) / \mathcal{M}(\mathscr{U}_{z}))$, $i\in\{1,\cdots,n\}$. The Galois group of a field extension $K\subset L$, $\mathscr{G}(L/K)$, consists of the set of field $K$-automorphisms (or permutations)  $\sigma : L \to L$ leaving invariant $K$. Thanks to  \eqref{ISOmK}, the following group isomorphism holds
\begin{equation}
	\label{ISOmKi}
	\mathscr{G}_{i}(\mathcal{M}(\mc{X}_{i}) / \mathcal{M}(\mathscr{U}_{z})) \simeq
	\mathscr{G}\left(\frac{\mathcal{M}(\mathscr{U}_{z})[\lambda]}{(p_{i}(\lambda,z))}
	\Big{/}\mathcal{M}(\mathscr{U}_{z})\right)=:\mathscr{G}_{i}.
\end{equation}
According to \cite[Th. 8.12]{OF}, the map $\mf{I}: \Deck(\mc{X}_{i},\Pi_{i})\to \mathscr{G}_{i}(\mathcal{M}(\mc{X}_{i}) / \mathcal{M}(\mathscr{U}_{z}))$, $\varphi \mapsto \sigma[\varphi]$,
where,
$$\sigma[\varphi]: \mathscr{G}_{i}(\mathcal{M}(\mc{X}_{i}) / \mathcal{M}(\mathscr{U}_{z})) \to \mathscr{G}_{i}(\mathcal{M}(\mc{X}_{i}) / \mathcal{M}(\mathscr{U}_{z})) , \quad f\mapsto f\circ \varphi^{-1},$$
is a group isomorphism. Therefore, by \eqref{ISOmKi},  $\Deck(\mc{X}_{i},\Pi_{i})$ is isomorphic to $\mathscr{G}_{i}$ for all $i\in \{1,\cdots,n\}$. By the fundamental theorem of the Galois theory, for every $i\in\{1,...,n\}$, the roots $\l=\l(z)$ of the irreducible polynomial $p_{i}(\lambda,z)$ can be expressed as a composition of radicals and meromorphic functions on $\mathscr{U}_{z}$ if and only if the Galois group $\mathscr{G}_{i}$ is solvable.  For instance, the Galois group of the irreducible polynomial
$\l^2+z\l+1\in\mc{M}(\mathscr{U}_{z})[\lambda]$, is solvable because its roots can be expressed as $\l(z)=\frac{1}{2}(-z\pm \sqrt{z^2-1}), \quad z\in\mathscr{U}_{z}$.
Precisely, its Galois group is given by $\mathscr{G}(L/K)=\{1,\sigma\}$, where
$L=\frac{\mc{M}(\mathscr{U}_{z})[\lambda]}{(\l^2+z\l+1)}$, $K=\mc{M}(\mathscr{U}_{z})$, $1:L\to L$ stands for the identity map, and $\sigma:L\to L$ is defined by $\sigma(\l)=-\l$. The previous results can be summarized into the next one.

\begin{theorem}
	\label{T6.4.2}
	Suppose that $\mathfrak{F}\in\mathcal{C}^{\omega}(\mathbb{C}\times U,V)$ satisfies {\rm (F1)--(F3)} and its 	 associated Weierstrass polynomial $p(\lambda,z)$ has irreducible components $p_{i}(\lambda,z)$, $1\leq i \leq n$. Then, $(0,0)\in\mathbb{C}\times U$ is a bifurcation point of $\mf{F}(\l,u)=0$ from $(\l,0)$ at $\l=0$, regardless the value of $\chi$. Moreover, the set $\mathfrak{F}^{-1}(0)$ near $(0,0)$ is in one-to-one correspondence with
$\mathscr{T}_{i}=\bigcup_{i=1}^{n}\mc{S}_{i} \cup \{(\lambda,0):\lambda\in\mathscr{U}_{\lambda}\}$, and the structure of this set is given by the pair $(\mc{X}_{i},\Pi_{i})$, where $\mc{X}_{i}$ is the Riemann surface of the polynomial $p_{i}(\lambda,z)$, $i\in\{1,...,n\}$, and $\Pi_{i}:\mc{X}_{i}\to \mathscr{U}_{z}$ stands for its associated branched covering map. Also, the structure of the singularity at $(0,0)$ is described by the Galois group $\Deck(\mc{X}_{i},\Pi_{i})\simeq \mathscr{G}_{i}$,
$i\in \{1,\cdots,n\}$. Furthermore, for every $i\in\{1,...,n\}$, $\mathscr{T}_{i}$ can be expressed as
$\mathscr{T}_{i}=\{(\l(z),z): z\in\mathscr{U}_{z}\}\cap \mathscr{U}$, where $\l:\mathscr{U}_{z}\to \C$ is an algebraic composition of radicals and meromorphic functions if and only if $\mathscr{G}_{i}$
	is solvable.
\end{theorem}

\subsection{The real case $\mathbb{K}=\mathbb{R}$} \label{Sect. Real} As in this case $\mathscr{U}_{z}\subset\mathbb{R}$ is an open interval containing $0$, $\mathscr{U}_{z}\backslash\{0\}=\mathscr{U}_{z}^{+}\uplus \mathscr{U}_{z}^{-}$, where,
for some $\a<0<\o$, $\mathscr{U}^{-}_{z}$ and $\mathscr{U}^{+}_{z}$ are the intervals $\mathscr{U}^{-}_{z}=(\a,0)$ and $\mathscr{U}^{+}_{z}=(0,\o)$.  Throughout this section, we denote
\begin{align*}
	\mathscr{B} & :=\{x+iy\in\mathbb{C} :\; x,\, y\in\mathscr{U}_{z}\}\equiv \mathscr{U}_z\times \mathscr{U}_z,\\  \mathscr{B}_{+} & :=\{x+iy\in\mathbb{C} :\; x\in\mathscr{U}_{z}^{+}, \, y\in\mathscr{U}_{z}\}
	\equiv \mathscr{U}_z^+\times \mathscr{U}_z, \\ \mathscr{B}_{-} & :=\{x+iy\in\mathbb{C} :\; x\in\mathscr{U}_{z}^{-},\, y\in\mathscr{U}_{z}\}\equiv \mathscr{U}_z^-\times \mathscr{U}_z.
\end{align*}
In this case, the Weierstrass polynomial  $p(\lambda,z)$ lies in $\mathcal{M}(\mathscr{U}_{z})[\lambda]$, and it can be regarded as a polynomial of $\mathcal{M}(\mathscr{B})[\lambda]$ by complexifying the variable $z$ to $z=x+iy$. Thus, as in the complex case, we can decompose the complex polynomial
$p(\lambda,z)\in\mathcal{M}(\mathscr{B})[\lambda]$ in its irreducible components
$p(\lambda,z)=\prod_{i=1}^{n}p_{i}(\lambda,z)$, $z\in\mathscr{B}$. Since $\mathscr{B}_{\pm}$ is a simply connected open set not containing $0$, by analytic continuation, we find from \cite[Cor. 8.8]{OF}, that, for every $i\in \{1,...,n\}$, there exist an integer $m_i\geq 1$ and $2m_i$ analytic functions,
$\varphi^{i}_{j}:\mathscr{B}_{+}\to\mathbb{C}$, $\phi^{i}_{j}:\mathscr{B}_{-}\to\mathbb{C}$,
$j\in\{1,..., m_{i}\}$, such that
\begin{equation*}
	p_{i}(\lambda,z) =\prod_{j=1}^{m_{i}}(\lambda-\varphi^{i}_{j}(z)) \;\; \hbox{for all}\;\;
	z\in\mathscr{B}_{+}, \quad p_{i}(\lambda,z) =\prod_{j=1}^{m_{i}}(\lambda-\phi^{i}_{j}(z))
	\;\;  \hbox{for all}\;\; z\in\mathscr{B}_{-}.
\end{equation*}
Thus, since $\chi=\deg p = \sum_{i=1}^n m_i$, $p(\l,z)$, $p(\l,z)$ can be factorized as
\begin{equation*}
	p(\lambda,z) =\prod_{k=1}^{\chi}(\lambda-\varphi_{k}(z)) \;\; \hbox{for all}\;\; z\in\mathscr{B}_{+},
\quad  p(\lambda,z)  =\prod_{k=1}^{\chi}(\lambda-\phi_{k}(z)) \;\; \hbox{for all}\;\;  z\in\mathscr{B}_{-},
\end{equation*}
for some analytic functions $\varphi_{k}:\mathscr{B}_{+}\to\mathbb{C}$ and  $\phi_{k}:\mathscr{B}_{-}\to\mathbb{C}$, $1\leq k \leq \chi$. Subsequently, we will consider the complex functions $\varphi_{k}(z)$ and $\phi_{k}(z)$, $1\leq k \leq \chi$,  as functions of a real variable, $z\in \mathbb{R}$. Since the zeroes of the analytic non-zero functions are isolated, shortening the interval $\mathscr{U}_{z}$,  if necessary, for every $k\in\{1,...,\chi\}$ and $h_k\in\{\v_k,\phi_k\}$, some of the next excluding options occurs. Either (a) $h_k(\mathscr{U}^{+}_{z})\subset \mathbb{R}$; or (b) $h_k(\mathscr{U}_{z}^{+})\cap \mathbb{R}=\emptyset$. Indeed, since $\mathrm{Im\,} h_k(z)=0$ if and only if $h_k(z)\in \mathbb{R}$, and $\mathrm{Im\,} h_k$ is analytic, either $\mathrm{Im\,} h_k=0$, or $\mathrm{Im\,} h_k$ cannot vanish for sufficiently small  $\mathscr{U}_{z}^+=(0,\o)$.
\par
Next, we pick a point $(\lambda_{0},z_{0})\in \mathfrak{G}^{-1}(0)$ with $z_0\in \mathscr{U}_z$. If $z_{0}=0$, then $(\lambda_{0},z_{0})=(\l_0,0)$ belongs to the trivial branch. If $z_{0}\neq 0$, then
either $z_0\in \mathscr{U}_z^+$, or $z_0\in\mathscr{U}_z^-$. Suppose, without loss of generality, that $z_{0}\in\mathscr{U}^{+}_{z}$. Then, since $p(\l_0,z_0)=0$, there exists $k\in \{1,...,\chi\}$ such that $\lambda_{0}=\v_k(z_0)$. Therefore, every zero $(\lambda_{0},z_{0})$ of  $\mathfrak{G}$ belongs to an analytic curve. Consequently, at least locally, $\mathfrak{G}^{-1}(0)$ consists of local branches of analytic curves. Moreover, since there are, at most,  $\chi$ analytic curves in each direction, from $(0,0)$ might emanate,  at most,  $2\chi+2$ analytic curves: $\chi$ from each direction $\mathscr{U}^{\pm}_{z}$, plus the two half branches of the trivial curve.
\par
When, in addition, $\chi =\deg p$ is an odd integer, then, for every $z\in\mathscr{U}_{z}\backslash\{0\}$, $p(\l,z)$ has a real root. Thus, there exist $k_{1}, k_{2}\in\{1,2,...,\chi\}$ such that $\varphi_{k_{1}}(\mathscr{U}_{z}^{+})\subset \mathbb{R}$ and $\phi_{k_{2}}(\mathscr{U}_{z}^{-})\subset \mathbb{R}$. Consequently, $\varphi_{k_{1}}:\mathscr{U}_{z}^{+}\to \mathbb{R}$ and $\phi_{k_{2}}:\mathscr{U}_{z}^{-}\to \mathbb{R}$ are real analytic functions. Therefore,
$\mf{G}^{-1}(0)$ contains, at least, the two branches of the trivial curve and two real analytic curves emanating from $(\l,0)$ at $\l=0$. In particular, $(0,0)$ is a bifurcation point of $\mf{F}=0$ from $(\l,0)$, in agrement with Theorem  \ref{T6.2.1}.
\par
To conclude this subsection, we present a new method, based on the Sturm theorem, for determining the exact number of branches of analytic curves that can bifurcate from $(0,0)$. We already know that
there are, at most, $2\chi+2$ curves. Namely, at most $\chi$ from $\mathscr{U}_{z}^{+}$, at most $\chi$ from $\mathscr{U}_{z}^{-}$,  and the remaining two are the half branches of the trivial curve $(\l,0)$, which always exist. We will focus our attention into the case when $z \in \mathscr{U}_z^+$, as the case when
$z\in \mathscr{U}_{z}^{-}$ follows the same general patterns.
\par
Given a polynomial $q\in\mathbb{R}[x]$ of degree $n\geq 1$, we define the Sturm chain of the polynomial $q$ as the sequence of polynomials
\begin{equation*}
	q_{0}(x):=q(x), \quad q_{1}(x):=q'(x), \quad q_{i}(x):=-\hbox{rem\,}(q_{i-2}(x),q_{i-1}(x)),
\end{equation*}
for all $i\geq 2$, where  $\hbox{rem\,}(q_{i-2},q_{i-1})$ stands for the remainder of the Euclidean division of $q_{i-2}$ by $q_{i-1}$. The length of the Sturm chain is, at most, $n$, the degree of $q$.
For every  $\xi\in\mathbb{R}\setminus q^{-1}(0)$, let $V(\xi)$ denote the number of sign changes of the Sturm chain  $(q_{0}(\xi),q_{1}(\xi),...,q_{n}(\xi))$ without taking into account zeroes. Then, the Sturm  theorem states that, whenever $a<b$ are not roots of $q$, the number  $V(a)-V(b)$ counts  the distinct real roots of $q$ in the interval $(a,b)$. By the Lyapunov--Schmidt reduction procedure described at the beginning of this section, we already know that, locally at  $(0,0)$,  the solutions of $\mathfrak{F}(\l,u)=0$ are in one-to-one correspondence with the zeroes $(\l,u)\in \mathscr{U}$ of the reduced nonlinear operator
$\mathfrak{G}(\lambda,z)=zc(\lambda,z)p(\l,z)$, $(\l,z)\in\mathscr{U}$, where $p(\l,z)$ is the associated
Weierstrass polynomial. According to our previous analysis,  we can factorize $p(\l,z)$ in $\mathscr{U}_{z}^{+}$ as $p(\lambda,z)=\prod_{j=1}^{\chi}(\lambda-\varphi_{i}(z))$, $z\in \mathscr{U}_{z}^{+}$, for some analytic functions $\varphi_{i}:\mathscr{U}^{+}_{z}\to\mathbb{C}$, $1\leq i \leq \chi$, and we already know that, for every $1\leq i\leq \chi$, either (i) $\varphi_{i}(\mathscr{U}^{+}_{z})\subset \mathbb{R}$, or (ii) $\varphi_{i}(\mathscr{U}_{z}^{+})\cap \mathbb{R}=\emptyset$. Subsequently, we consider the neighborhood
$\mathscr{U}\equiv\mathscr{U}_{\l}\times\mathscr{U}_{z}$, shortened so that $\left(\{\lambda_{\pm}\}\times \mathscr{U}_{z}\right)\cap 	p^{-1}(0)=\left\{(\lambda_{\pm},0)\right\}$ for some $\l_{-},\l_{+}\in\mathscr{U}_{\l}$, $\l_{-}<\l_{+}$. This can be easily accomplished from  the implicit function theorem, because $\mf{L}(\l)\in GL(U,V)$ for $\l\sim 0$, $\l\neq 0$.  For this choice of $\mathscr{U}$ we have that, for every $z\in\mathscr{U}^{+}_{z}$, the number of real roots of the polynomial $p(\lambda,z)$ in the interval $(\lambda_{-},\lambda_{+})$ is constant; it equals the number of $i$'s with
$i\in\{1,...,\chi\}$ for which $\v_i(z)\in\mathbb{R}$. Next, for every
$(\lambda,z)\in\mathscr{U}_{\lambda}\times\mathscr{U}^{\pm}_{z}$, we consider the following Sturm chain of polynomials in $\l$
\begin{align*}
	p_{0}(\lambda,z) &:=p(\lambda,z), \qquad  p_{1}(\lambda,z):=\partial_{\lambda}p(\lambda,z), \\ p_{i}(\lambda,z) & :=-\text{rem}_{\lambda}(p_{i-2}(\lambda,z),p_{i-1}(\lambda,z)) \quad
	\hbox{for all}\;\;  i\geq 2,
\end{align*}
where the $z$ variable is regarded as a label for the polynomial. For each
$\xi\in[\lambda_{-},\lambda_{+}]$, we denote by $V(\xi,z)$ the number of sign changes of the
chain  $(p_{0}(\xi,z),p_{1}(\xi,z),\ldots,p_{\chi}(\xi,z))$. As the number of real roots of $p(\lambda,z)$ in $(\lambda_{-},\lambda_{+})$ is constant for all $z\in\mathscr{U}^{+}_{z}$, by the Sturm theorem, $V(\lambda_{-},z)-V(\lambda_{+},z)$ is constant for all $z\in\mathscr{U}^{+}_{x}$, and
it equals the number of distinct (real) analytical curves emanating from $(0,0)$ in the direction of $\mathscr{U}^{+}_{z}$. Similarly, $V(\lambda_{-},z)-V(\lambda_{+},z)$ is constant for all $x\in\mathscr{U}^{-}_{z}$, and it equals the number of distinct (real) analytical curves emanating from $(0,0)$ in the direction of $\mathscr{U}^{-}_{z}$. Thus, the next result holds.

\begin{theorem}
	\label{T6.4.3}
	Suppose that $\mathfrak{F}\in\mathcal{C}^{\omega}(\mathbb{R}\times U,V)$ satisfies {\rm (F1)--(F3)}. Then, in a sufficiently small neighborhood of $(0,0)\in\mathbb{R}\times U$, the set $\mathfrak{F}^{-1}(0)$ consists of finitely many branches of analytical curves. Precisely, it has
	\begin{equation*}
		\mathcal{N}_{(0,0)}=[V(\lambda_{-},z_{+})-V(\lambda_{+},z_{+})]+
		[V(\lambda_{-},z_{-})-V(\lambda_{+},z_{-})]+2,
	\end{equation*}
	distinct half branches of analytic curve, regardless the values of $z_{\pm}\in\mathscr{U}^{\pm}_{z}$.  In particular, if $\chi$ is an odd integer, then $\mathcal{N}_{(0,0)}\geq 3$, and  $(0,0)$ is a bifurcation point of $\mf{F}(\l,u)=0$ from $(\l,0)$.
\end{theorem}

As Theorem \ref{T6.4.3} counts the exact number of analytic branches of $\mf{F}^{-1}(0)$,
it provides us with a substantial improvement of the pioneering results of
Kielh\"{o}fer \cite{Ki}, where it was established that $\mf{F}^{-1}(0)$ possesses at most
$2\chi+2$ analytic branches emanating from $(0,0)$.

\subsection{Kielh\"ofer's result}

We conclude this section by getting  a result of Kielh\"ofer \cite{Ki} that will be useful for analyzing the example of Section \ref{SDODP}. Throughout this subsection we work with a pair $(U,V)$ of real Banach spaces with continuous inclusion $U\subset V$ and an analytic map $\mf{F}:\R\times U\to V$ satisfying  (F1)--(F3) and, in addition,
\vspace{0.2cm}
\begin{enumerate}
	\item[(F4)] $\l_{0}\in\Sigma(\mf{L})$ is an simple eigenvalue of $\mf{L}_{0}$, i.e., $V=N[\mf{L}_{0}]\oplus R[\mf{L}_{0}]$.
\end{enumerate}
\vspace{0.2cm}
Subsequently,  we suppose, without loss of generality, that $\l_{0}=0$. By  (F4), it follows from the Hahn--Banach theorem that there exists $\varphi_{0}^{\ast}\in V^{\ast}$ such that $\langle \varphi_{0}, \varphi_{0}^{\ast}\rangle=1$ and
$$
  R[\mf{L}_{0}]=\{v\in V: \langle v, \varphi_{0}^{\ast}\rangle=0\},
$$
where $\langle \cdot,\cdot\rangle: V\times V^{\ast} \to \R$ is the duality pairing on $V$. By \cite[Le. 4.4.1]{LGMC}, the zero eigenvalue of $\mf{L}_0$ perturbs into a unique eigenvalue $\mu(\l)$ of
$\mf{L}(\l)$. Precisely, there exist $\delta>0$ and two
(unique) analytic functions $\mu \in \mc{C}^{\o} (-\delta, \delta)$ and $\varphi \in \mc{C}^{\o}
((-\delta,\delta), U)$ such that $\mu(0)=0$, $\varphi(0)=\varphi_0$, $\varphi(\l)-\varphi_0\in R[\mf{L}_{0}]$, and $\mf{L}(\l) \varphi(\l) = \mu(\l) \varphi(\l)$ for all $\l\in (-\delta, \delta)$. Now, consider the pair of projections $(P,Q)$ defined by $P: V\to N[\mf{L}_{0}]$, $P(v):=\langle v, \varphi_{0}^{\ast}\rangle \varphi_{0}$, $Q: V \to R[\mf{L}_{0}]$, $Q:= I_{V}-P$. By the theory of Sections \ref{SectionALS} and \ref{SSLBA}, performing a Lyapunov--Schmidt reduction with pair $(P,Q)$, it is  apparent that
the zeroes of $\mf{F}$ near $(0,0)$, say in the neighbourhood $\mc{O}$, are in one-to-one correspondence with the zeroes of the finite dimensional analytic map $\mathfrak{G}: \O\subset\mathbb{R}\times \R \longrightarrow \mathbb{R}$, defined by
\begin{equation}
	\label{FDAM}
	\mf{G}(\lambda,z):=S(I_{V}-Q)\mathfrak{F}(\lambda,z\varphi_{0}+\mathcal{Y}(\lambda,z\varphi_{0})), \quad (\l,z)\in \O,
\end{equation}
where $S$ is the linear isomorphism given by
$S: N[\mf{L}_{0}]\to\R$, $S(z\varphi_{0})=z$.
By the definition of the projections $(P,Q)$, $\mf{G}(\lambda,z)=\langle\mathfrak{F}(\lambda,z\varphi_{0}+\mathcal{Y}
(\lambda,z\varphi_{0})),\varphi_{0}^{\ast}\rangle$. Moreover, according to the analysis of Section \ref{SSLBA}, there exists an analytic function $g:\O\to\R$ such that $\mf{G}(\l,z)=z g(\l,z)$.
By Lemma \ref{Lemma12.1.2}, $\chi\equiv \chi[\mf{L},0]=\ord_{\l=0}g(\l,0)$. Thus, $g:\O\to\mathbb{R}$ has the expansion
\begin{equation}
	\label{iv.2}
	g(\lambda,z)=\sum_{\nu=0}^{s}C_{\nu}\lambda^{j_{\nu}}z^{\ell_{\nu}}+\sum_{j,k}C_{j,k}\lambda^{j}z^{k},
\end{equation}
where $(\ell_{0},j_{0})=(0,\chi)$, $\chi>j_{1}>\cdots>j_{s}$, $0<\ell_{1}<\cdots<\ell_{s}$, and the summation of the second sum is taken only on the points $(k,j)$ lying above the polygonal line joining $(0,\chi)$, $(\ell_1,j_1)$, $\cdots$, $(\ell_{s},j_{s})$, or on the line $j=j_{s}$. The polygonal line joining the points $(0,\chi)$, $(\ell_{1},j_{1})$, $\cdots$, $(\ell_{s},j_{s})$ is usually called the \textit{Newton's polygon} of $g$. This shows the validity of Lemma 5.4 of \cite{Ki}, which can be stated as follows.

\begin{theorem}
	\label{th4.2}
	Let $\mathfrak{F}:\mathbb{R}\times U\to V$ be an analytic map satisfying {\rm (F1)--(F4)} with $\chi[\mathfrak{L},0]=\chi\geq 1$ and having the expansion  	
	\begin{equation}
		\label{iv.3}	 \mathfrak{F}(\lambda,u)=\mathfrak{L}(\lambda)u+\sum_{\nu=1}^{s}\lambda^{j_{\nu}}L^{j_{\nu}}_{\ell_{\nu}+1}u^{\ell_{\nu}+1}
		+\sum_{j,k}\lambda^{j}L^{j}_{k+1}u^{k+1},
	\end{equation}
	where $\chi>j_{1}>\cdots>j_{s}$,  $0<\ell_{1}<\cdots<\ell_{s}$, and the summation of the second sum is taken on the points $(k,j)$ lying above the polygonal line joining $(0,\chi)$, $(\ell_1,j_1)$, $\cdots$, $(\ell_{s},j_{s})$, or on the line $j=j_{s}$. If, in addition, $H_{\nu}:=\langle L^{j_{\nu}}_{\ell_{\nu}+1}\varphi_{0}^{\ell_{\nu}+1},\varphi_{0}^{\ast}\rangle \neq 0$ for all
$\nu=1,\cdots,s$, then the Newton polygon associated to the reduced map $g:\O\to\mathbb{R}$ defined by  \eqref{iv.2}  is given by the polygonal line joining the points $(0,\chi)$, $(\ell_1,j_1)$, $\cdots$, $(\ell_{s},j_{s})$. Furthermore, the corresponding coefficients are given by $
C_{0}=\frac{\mu^{(\chi)}(0)}{\chi!}$, $C_{\nu}=H_{\nu}$, $\nu=1,\cdots,s$.
\end{theorem}

\noindent Consequently, under the hypothesis of Theorem \ref{th4.2}, the zero set $\mf{F}^{-1}(0)$ locally at  $(0,0)\in\R\times U$ is in one-to-one correspondence with the zeroes of the finite dimensional analytic map
$\mf{G}=\mf{G}(\l,z)$, given by
$$
\mf{G}(\l,z)=z\Big[\frac{\mu^{(\chi)}(0)}{\chi!}\l^{\chi}+\sum_{\nu=1}^{s}\langle L^{j_{\nu}}_{\ell_{\nu}+1}\varphi_{0}^{\ell_{\nu}+1},\varphi_{0}^{\ast}\rangle \l^{j_{\nu}} z^{\ell_{\nu}}+\text{higher order terms (hot)}\Big].
$$

\section{Analytic global alternative}\label{SAGA}

\noindent The aim of this section is to sharpen the global bifurcation theorems for analytic nonlinearities of Dancer \cite{Da,Da73,Da732} and Buffoni and Tolland \cite{BT}. These results where originally stated  for the special case of $1$-transversal eigenvalues, where the theorem of Crandall and Rabinowitz \cite{CR} applies. Our main goal is to generalize these findings up to cover the degenerate case when $\chi[\mf{L},\l_{0}]\geq 2$. This is imperative in many applications where $\chi[\mf{L},\l_{0}]\neq 1$, as the one given in Section \ref{SDODP}. Throughout this section, given a pair $(U,V)$ of real Banach spaces, we consider a map $\mathfrak{F}\in\mathcal{C}^{\omega}(\mathbb{R}\times U,V)$ satisfying the following properties:
\begin{enumerate}
	\item[(F1)] $\mathfrak{F}(\lambda,0)=0$ for all $\lambda\in\mathbb{R}$.
	\item[(F2)] $D_{u}\mathfrak{F}(\lambda,u)\in\Phi_{0}(U,V)$ for all
	$(\lambda,u)\in\mathbb{R}\times U$.
	\item[(F3)] $\l_{0}\in\Sigma(\mf{L})$ is an isolated eigenvalue such that $N[\mf{L}_{0}]=\mathrm{span}[\varphi_{0}]$ for some $\varphi_0\in U\backslash\{0\}$.
	\item[(F4)] $\mf{F}$ is proper on closed and bounded subsets of $\R\times U$.
\end{enumerate}
Given an analytic nonlinearity $\mf{F}\in\mc{C}^{\o}(\mathbb{R}\times U,V)$ satisfying conditions {\rm (F1)--(F4)}, it is said that $(\lambda,u)\in\mathfrak{F}^{-1}(0)$ is a \emph{regular point} if $D_{u}\mathfrak{F}(\lambda,u)\in GL(U,V)$. In the contrary case when $D_{u}\mathfrak{F}(\lambda,u)\notin GL(U,V)$ is called \emph{singular}. The set of regular points of $\mf{F}$ will be denoted by $\mc{R}(\mf{F})$. According to the exchange stability principle (see, e.g., \cite{CRex} and \cite[Th. 2.4.2]{LG01}), when $\chi[\mf{L},\l_{0}]=1$, in a neighborhood of $(0,0)$ all nontrivial solutions  of an analytic operator are regular points, unless the bifurcation is vertical.  Subsequently, we will give some general sufficient conditions so that, locally at $(0,0)$, the bifurcated solutions are regular points of $\mf{F}^{-1}(0)$ even in the degenerate case when $\chi[\mf{L},\l_{0}]\geq 2$.
\par
By Theorem \ref{T6.4.1}, we already know that, locally in a neighbourhood $\mathscr{U}$ of $(0,0)$, the solutions of $\mathfrak{F}(\lambda,u)=0$ are in one-to-one correspondence with the zeroes of  the reduced map $\mathfrak{G}(\lambda,z)=zc(\lambda,z)p(\l,z)$, $(\l,z)\in \mathscr{U}$, where $p(\l,z)$ is the associated Weierstrass polynomial. Moreover, by Theorem \ref{Lyap-Smith}, there is a local bijection between
the regular points of $\mf{F}^{-1}(0)$ and those of $\mathfrak{G}^{-1}(0)$. Obviously, the regular points of $\mathfrak{G}^{-1}(0)$ are characterized through the condition $D_{z}\mathfrak{G}(\lambda,z)\neq 0$.
By the analysis already done in Section \ref{Sect. Real}, in $\mathscr{U}_{z}^{\pm}$ one can factorize $\mf{G}(\l,z)$ in the form $\mathfrak{G}(\lambda,z)=zc(\lambda,z)\prod_{j=1}^{\chi}(\lambda-\varphi_{i}(z))$,
$z\in \mathscr{U}_{z}^{\pm}$ for some analytic functions $\varphi_{i}:\mathscr{U}^{\pm}_{z}\to\mathbb{C}$, $1\leq i \leq \chi$. Thus, differentiating $\mf{G}(\l,z)$ with respect to $z$, we find that
\begin{align*}
	D_{z}\mathfrak{G}(\lambda,z)= & [c(\lambda,z)+zD_{z}c(\lambda,z)]p(\l,z) -zc(\lambda,z)\sum_{i=1}^{\chi} \frac{d\varphi_{i}}{dz}(z)\prod_{j\neq i}(\lambda-\varphi_{j}(z)).
\end{align*}
Hence, a given local analytic branch $\lambda=\varphi_{k}(z)$, $k\in\{1,...,\chi\}$,  of zeroes of $\mathfrak{G}(\lambda,z)$ consists of singular points if and only if
$ D_{z}\mathfrak{G}(\varphi_{k}(z),z)=0$, or, equivalently, if
\begin{equation*}
	0 = \sum_{i=1}^{\chi}\frac{d\varphi_{i}}{dz}(z)\prod_{j\neq i} (\varphi_{k}(z)-\varphi_{j}(z))=\frac{d\varphi_{k}}{dz}(z)\prod_{j\neq k}(\varphi_{k}(z)-\varphi_{j}(z)).
\end{equation*}
Therefore, the next result holds.

\begin{theorem}
	\label{T6.4.4}
	Let $\lambda=\varphi_{k}(z)$, $z\in\mathscr{U}_{z}^{\pm}$, be a local analytic branch of zeroes of $\mathfrak{G}(\lambda,z)$ bifurcating from $(0,0)$. Then, either this branch consists of
	regular solutions, or it consists of singular solutions. Moreover, it consists of
	singular solutions if and only if some of the following conditions occurs:
	\begin{enumerate}
		\item[{\rm (a)}] $\varphi_{k}:\mathscr{U}^{\pm}_{z}\to\mathbb{C}$ is constant, i.e., there is $\varphi_{0}\in\mathbb{C}$ such that  $\varphi_{k}(z)=\varphi_{0}$ for all $z\in\mathscr{U}^{\pm}_{z}$.
		\item[{\rm (b)}] There exists another $j\in\{1,...,\chi\}$, $j\neq k$,  such that $\varphi_{k}=\varphi_{j}$ in $\mathscr{U}_z^\pm$.
	\end{enumerate}
\end{theorem}

As a byproduct of Theorem \ref{T6.4.4}, when $p(\lambda,z)$ is irreducible on $\mathcal{M}(\mathscr{U})[\lambda]$, neither it can admit multiple roots in the splitting field, nor any constant function can be a root of $p(\lambda,z)$. So,  all bifurcating branches from $(0,0)$ consist of regular points. The next result sharpens, very substantially,  Theorem 9.1.1 of \cite{BT} up to cover the  more general case when $\chi\geq 1$, where the main result of Crandall and Rabinowitz \cite{CR} fails.

\begin{theorem}
	\label{th.BT}
	Let $\mathfrak{F}\in\mathcal{C}^{\omega}(\mathbb{R}\times U,V)$ be an analytic map satisfying {\rm (F1)--(F4)}. Suppose that there exists an analytic injective curve $\gamma:(0,1)\to \R\times U$ such that  $\gamma(0,1)\subset \mf{F}^{-1}(0)\backslash \mc{T}$, $\lim_{t\da 0}\gamma(t)=(0,0)$, and $\gamma(0,1)\subset \mc{R}(\mf{F})$. Then, there exists a locally injective continuous path
$\G: [0,\infty)\to \R\times U$, with $\Gamma([0,\infty))\subset \mf{F}^{-1}(0)$,  for which there exists $0<\delta<1$ satisfying $\Gamma(t)=\gamma(t)$ for all $t\in (0,\delta)$. Moreover, $\Gamma$ satisfies one of the following non-excluding alternatives:
	\begin{enumerate}
		\item[{\rm (a)}]   $\lim_{t\ua \infty} \|\Gamma(t)\|_{\mathbb{R}\times U}= +\infty$.
		\item[{\rm (b)}]   $\Gamma$ is a closed loop, i.e., there exists $T>0$ such that $\Gamma(T)=(0,0)$.
	\end{enumerate}
	In particular, when the associated Weierstrass polynomial, $p(\lambda,z)$, is irreducible, every analytical curve of $\mathfrak{F}^{-1}(0)$ emanating from $(0,0)$ can be extended to a locally injective continuous path $\G: [0,+\infty)\to \mathbb{R}\times U$ on $\mathfrak{F}^{-1}(0)$ satisfying some of these alternatives.
\end{theorem}

\begin{proof}
As the proof of this result follows,  almost \emph{mutatis mutandis}, the proof of \cite[Th. 9.1.1]{BT},
we simply sketch it. It relies on the theorem of structure of analytic manifolds \cite[Th. 7.4.7]{BT}.
According to it, either $\mf{F}^{-1}(0)$ consists of singular points, or its set of singular points is discrete. As $\g$ is a local branch consisting of regular points, necessarily the set of singular
points of $\mf{F}^{-1}(0)$ is discrete. On the other hand, by complexifying $U$, it follows from the theorem of structure of analytic varieties \cite[Th.7.4.7]{BT} and the parametrization theorem \cite[Cor.7.5.3]{BT} that  any arc of analytic branch can be prolonged, after passing any singular point, to another
branch of analytic curve. After taking the maximal route (see Step 3 of page 118 of \cite{BT}), we conclude the existence of the extension curve. Finally, if condition (a) fails, since the bounded subsets of $\mathfrak{F}^{-1}(0)$ are compact, there exists an accumulation point that must coincide with a singular point. By the definition of maximal route,  $\Gamma$ must be a closed loop (see p. 119 of \cite{BT} for any further detail).
\end{proof}

The global alternatives of Theorems \ref{C6.3.6} and \ref{th.BT} are independent. Indeed, if the connected component of $\mf{F}^{-1}(0)$ bifurcating from $(0,0)$, say
$\mf{C}$,  is bounded, then, according to Theorem \ref{th.BT}, $\mf{F}^{-1}(0)$ contains a closed loop. But this does not entail, necessarily, the existence of some $(\l_1,0)\in\mf{C}$ with $\l_1\neq 0$, as it is guaranteed by Theorem \ref{C6.3.6} when, in addition, $\chi$ is odd. Conversely, when $\mf{C}$ is bounded and $\chi$ is odd, then, owing to Theorem \ref{C6.3.6}, $(\l_1,0)\in \mf{C}$ for some $\l_1\neq 0$, though this does not entail that any local analytic curve bifurcating from $(0,0)$ can be continued to a global closed loop.
\par
Such an independence is far from surprising, as these global alternatives are of a completely different nature: algebraic the one of Theorem \ref{th.BT} and topological the classical one of Theorem \ref{C6.3.6}. Actually, the proof of Theorem \ref{th.BT} does not invoke the degree, which was essential for the proof of Theorem \ref{C6.3.6}, but simply the theorem of structure of analytic varieties, which remained outside the proof of Theorem \ref{C6.3.6}. Thus, it should not come to surprise that they are alternatives of a rather different nature, though certainly reminiscent.

\subsection{Global graphs}\label{SGG}
In this section, inspired by a novel idea of Dancer \cite[Th.3]{Da732}, we study the global structure of the zero set of an analytic nonlinearity satisfying the special conditions  set out below. Given a pair $(U,V)$ of real Banach spaces, we consider an analytic map $\mathfrak{F}\in\mathcal{C}^{\omega}(\mathbb{R}\times U,V)$ satisfying the following properties:
\begin{enumerate}
	\item[(F1)] $\mathfrak{F}(\lambda,0)=0$ for all $\lambda\in\mathbb{R}$.
	\item[(F2)] $D_{u}\mathfrak{F}(\lambda,u)\in\Phi_{0}(U,V)$ for all
	$(\lambda,u)\in\mathbb{R}\times U$.
	\item[(F3)] $\dim N[D_{u}\mathfrak{F}(\lambda,u)]\in\{0,1\}$ for all $(\lambda,u)\in\mathbb{R}\times U$.
\end{enumerate}
These assumptions are fulfilled in most of the applications involving one-dimensional nonlinear boundary value problems. The next definition introduces the concept of
\emph{analytic graph} that we are going to use in this section to describe the global structure of the zero set of $\mf{F}$.

\begin{definition}[\textbf{Analytic graph}]
	\label{D6.4.6}
	A closed subset $\mathscr{A}\neq \emptyset$ of $\mathbb{R}\times U$ is said to be an analytic graph when, for every $(\lambda,u)\in\mathscr{A}$, one of the following excluding options occurs:
	\begin{enumerate}
		\item[{\rm (a)}]  There exists $\varepsilon>0$ such that $B_\e(\lambda,u)\cap \mathscr{A}$ is the graph of an injective analytic curve $\gamma:(-\delta,\delta)\to\mathbb{R}\times U$, $\gamma(0)=(\l,u)$, i.e., $B_\e(\lambda,u)\cap \mathscr{A}=\gamma((-\delta,\delta))$. In such case, $(\lambda,u)$ is said to be an edge point.

		\item[{\rm (b)}] Item {\rm(a)} does not occur and there exists $\varepsilon>0$ such that either
		$B_\e(\lambda,u)\cap \mathscr{A}=\{(\l,u)\}$, or $B_\e(\lambda,u)\cap \mathscr{A}$ consists of the point $(\l,u)$ and $N\geq 1$ graphs of analytic injective curves $\gamma_{i}:(0,1)\to \mathbb{R}\times U$, $(\l,u)\notin \gamma_{i}((0,1))$, $i\in\{1,\cdots,N\}$, such that $\gamma_{i}(t)\to (\l,u)$ as $t\uparrow 1$, i.e.,
$$
  B_\e(\lambda,u)\cap \mathscr{A}=\bigcup_{i=1}^{N}\gamma_{i}((0,1))\cup \{(\l,u)\}.
$$
		In such case, $(\lambda,u)$ is called a vertex, or nodal point.
		\item[{\rm(c)}] There exists $\varepsilon>0$, an open subset $\O\subset \R^{2}$ and an analytic homeomorphism $\Psi: \Omega \to B_\e(\lambda,u)\cap \mathscr{A}$. In such case, $(\l,u)$ is called a residual point.
	\end{enumerate}
\end{definition}
	
	For any given analytic graph, $\mathscr{A}$, we will denote by $\mathcal{G}_{\mathscr{A}}$ the set of edge points of $\mathscr{A}$, by $\mc{V}_{\mathscr{A}}$ the set of vertex points of $\mathscr{A}$ and by $\mc{R}_{\mathscr{A}}$ the set of residual points of $\mathscr{A}$. From the definition we infer that $\mc{G}_{\mathscr{A}}$ and $\mc{R}_{\mathscr{A}}$ are open subsets of $\mathscr{A}$ and that $\mc{V}_{\mathscr{A}}$ is closed in $\mathscr{A}$. We define the \textit{skeleton} of $\mathscr{A}$ by $\mc{K}_{\mathscr{A}}:=\mc{V}_{\mathscr{A}}\uplus \mc{G}_{\mathscr{A}}$. It is easy to see that $\mc{K}_{\mathscr{A}}$ is an open subset of $\mathscr{A}$. Clearly, $\mathscr{A}=\mc{K}_{\mathscr{A}}\uplus \mc{R}_{\mathscr{A}}$ and $\mc{K}_{\mathscr{A}}\cap \mc{R}_{\mathscr{A}}=\emptyset$. The main result of this section reads as follows.

\noindent

\begin{theorem}
	\label{T6.4.8}
	Let $\mathfrak{F}\in\mathcal{C}^{\omega}(\mathbb{R}\times U,V)$ satisfying {\rm (F1)--(F3)}. Then, $\mathfrak{F}^{-1}(0)$ is an analytic graph of $\mathbb{R}\times U$.
\end{theorem}

\begin{proof}
	Pick a point $(\lambda_{0},u_{0})\in\mathfrak{F}^{-1}(0)$. Then, either $D_{u}\mathfrak{F}(\lambda_{0},u_{0})\in GL(U,V)$, or, due  to (F3),
	we have that
	\begin{equation}
		\label{6.6.43}
		\dim N[D_{u}\mathfrak{F}(\lambda_{0},u_{0})]=1.
	\end{equation}
	Suppose that $D_{u}\mathfrak{F}(\lambda_{0},u_{0})\in GL(U,V)$. Then, by the implicit function theorem, in a neighborhood of $(\lambda_{0},u_{0})$, $\mathfrak{F}^{-1}(0)$ consists of the graph of an injective analytic curve through the point $(\lambda_{0},u_{0})$. Therefore,  $(\lambda_{0},u_{0})$ is an edge point.
	\par
	Suppose \eqref{6.6.43}.  In this case, $N[D_{u}\mathfrak{F}(\lambda_{0},u_{0})]=\text{span}\{\varphi\}$
	for some $\varphi\in U\backslash\{0\}$. By the Hahn--Banach theorem, there exists $\varphi^{\ast}\in U^{\ast}$ such that $\langle \varphi, \varphi^{\ast}\rangle =1$. Let
$P:U\to N[D_{u}\mathfrak{F}(\lambda_{0},u_{0})]$ be the continuous projection defined by $P(u):=\langle u,\varphi^{\ast}\rangle \varphi$ for all $u\in U$, and consider any other continuous projection
$Q:V\to R[D_{u}\mathfrak{F}(\lambda_{0},u_{0})]$. Then,
$$
		U= N[D_{u}\mathfrak{F}(\lambda_{0},u_{0})] \oplus Y \;\; (Y=N[P]), \quad V =Z \oplus R[D_{u}\mathfrak{F}(\lambda_{0},u_{0})]  \;\; (Z=N[Q]).
$$
	In the sequel, we identify $\mathbb{R}\times N[D_{u}\mathfrak{F}(\lambda_{0},u_{0})]$ with $\mathbb{R}^{2}$ via the linear
	isomorphism
	\begin{equation*}
		T:\mathbb{R}\times N[D_{u}\mathfrak{F}(\lambda_{0},u_{0})]\longrightarrow \R^{2}, \quad T(\lambda,z\varphi)=(\lambda,L[z\varphi]),
	\end{equation*}
where, $L: N[D_{u}\mathfrak{F}(\lambda_{0},u_{0})]\longrightarrow \R$ is defined by $L[z\varphi]:=z$,
and we identify $Z$ with $\R$ via another fixed isomorphism $S:Z\to\R$. As in in Section \ref{SectionALS},
performing a Lyapunov--Schmidt reduction to $\mf{F}(\l,u)=0$ on $(\l_{0},u_{0})$ under the pair of $D_{u}\mf{F}(\l_{0},u_{0})$-projections $(P,Q)$, it becomes apparent that there are
a  neighborhood $\mc{U}$ of $(\l_{0},0)$ in $\R\times N[D_{u}\mathfrak{F}(\lambda_{0},u_{0})]$,
a neighborhood $\mc{O}$ of $(\l_0,u_{0})$ in $\R\times N[D_{u}\mathfrak{F}(\lambda_{0},u_{0})]$,
and an analytic operator 	$\mc{Y}:\mc{U}\to Y$ such that the maps
\begin{align*}
		\Psi:\mathfrak{F}^{-1}(0)\cap \mc{O}\longrightarrow \mathfrak{G}^{-1}(0), & \quad (\lambda,u)\mapsto (\lambda,\langle u-u_{0}, \varphi^{\ast}\rangle),\\
		\Psi^{-1}:\mathfrak{G}^{-1}(0)\longrightarrow \mathfrak{F}^{-1}(0)\cap\mc{O}, & \quad (\lambda,z)\mapsto (\lambda,u_{0}+z\varphi+\mathcal{Y}(\lambda,z\varphi)),
	\end{align*}
are mutually inverses, where $\mf{G}:\O\subset \R\times \R\longrightarrow \R$, $\mf{G}=\mf{G}(\l,z)$, is given by
	\begin{equation*}
		\mf{G}(\lambda,z):=S(I_{V}-Q)\mathfrak{F}(\lambda,u_{0}+z\varphi+\mathcal{Y}(\lambda,z\varphi)), \quad (\l,z)\in\O,
	\end{equation*}
	where $\O:=\{(\l,z): (\l,z\varphi)\in\mc{U}\}\subset \K\times \K$. 	In particular,  $\mf{G}(\l_{0},0)=0$. If $\mathfrak{G}\equiv 0$ in $\O$, then the map
	$\Psi^{-1}:\O\to \mathfrak{F}^{-1}(0)\cap\mc{O}$, $(\lambda,z)\mapsto (\lambda,u_{0}+z\varphi+\mathcal{Y}(\lambda,z\varphi))$, is an analytic homeomorphism and consequently $(\l_{0},u_{0})$ is a residual point. So, subsequently we suppose that $\mf{G}\not\equiv 0$. Let $\alpha\in\N\cup\{0\}$ be the minimum integer such that
	\begin{equation}
		\label{eq. de orden}
		\xi:=\ord_{\l=\l_{0}}D^{\alpha}_{z}\mf{G}(\l,0)<\infty.
	\end{equation}
The existence of $\alpha$ is guaranteed by the fact that  $\mf{G}\equiv 0$ if
$D_{\l}^{\beta}D^{\alpha}_{x}\mf{G}(\l_{0},0)=0$ for all $\alpha,\beta\in \N\cup\{0\}$. By 
\eqref{eq. de orden}, there exists an analytic function $g:\O\to \R$ such that
$\mf{G}(\l,z)=z^{\alpha}g(\l,z)$ for all $(\l,z)\in \O$. A direct computation shows that
$\xi=\ord_{\l=\l_{0}}g(\l,0)$. Thus, by the Weierstrass preparation theorem, there exists a neighborhood $\mathscr{U}\subset \O$ of $(\l_{0},0)$, such that
$$
   \mathfrak{G}(\lambda,z)=  \ z^{\alpha}c(\lambda,z)
		\left[(\lambda-\l_{0})^{\xi}+c_{\xi-1}(z)(\lambda-\l_{0})^{\xi-1}
		+\cdots+c_{0}(z)\right], \quad (\lambda,z)\in\mathscr{U},
$$
for some analytic function $c:\mathscr{U}\to\mathbb{R}$, with $c(\l_{0},0)\neq 0$,  and $\xi$ analytic functions $c_{j}:\mathscr{U}_{z}\to\mathbb{R}$, with $c_{j}(0)=0$ for all $0\leq j\leq \xi-1$. Then, setting  $\mathscr{U}^{+}_{z}:=(a,0)$, $\mathscr{U}^{-}_{z}:=(0,b)$ and adapting
	the argument of Section \ref{Sect. Real}, it becomes apparent  that there exist $\xi$ analytic functions $\varphi_{i}:\mathscr{U}^{\pm}_{z}\to\mathbb{C}$, $1\leq i \leq \xi$, such that
	\begin{equation*}
		\mathfrak{G}(\lambda,z)= z^{\alpha}c(\lambda,z)\prod_{i=1}^{\xi}(\lambda-\l_{0}-\varphi_{i}(z)), \quad z\in\mathscr{U}^{\pm}_{z}.
	\end{equation*}
	Thus, in a neighborhood of the point $(\lambda_{0},u_{0})$, $\mathfrak{F}^{-1}(0)$ consists of, at most,  $2\xi+2$ graphs of injective analytic curves and the point $(\l_{0},u_{0})$. Therefore, $(\lambda_{0},u_{0})$ is a vertex point. Indeed, either it is isolated, or there emanate from it finitely many arcs of (real) analytic curve. This concludes the proof.
\end{proof}
Actually, the proof of Theorem \ref{T6.4.8} shows that every subset of $\mf{F}^{-1}(0)$ is also an analytic graph. Finally, let us apply Theorem \ref{T6.4.8} to a paradigmatic one-dimensional boundary value problem.
Given $a\in\mc{C}[0,\pi]$ and an integer $p\geq 2$, we consider the boundary value problem
\begin{equation}
	\label{1.12}
	\left\{\begin{array}{l}
		-u''=\lambda u +a(x) u^p \quad \hbox{in}\;\, (0,\pi), \\
		u(0)=u(\pi)=0,
	\end{array}
	\right.
\end{equation}
whose solutions are the zeros of the analytic nonlinear operator
$$
   \mf{F}:\R\times \mc{C}^{2}_{0}[0,\pi]\longrightarrow \mc{C}[0,\pi], \quad \mf{F}(\l,u):=u''+\lambda u +a(x) u^p.
$$
It is straightforward to verify that $\mf{F}$ satisfies hypothesis (F1)--(F3) of this section. Thus, by Theorem \ref{T6.4.8}, the set $\mf{F}^{-1}(0)$ is an analytic graph. Moreover, the set of non-trivial solutions
$$
   \mc{S}=[\mf{F}^{-1}(0)\backslash\mc{T}]\uplus \{(\l,0) : \l\in\Sigma(\mf{L})\},
$$
is also an analytic graph as $\mc{S}\subset\mf{F}^{-1}(0)$.
On the other hand, by the local theorem of Crandall and Rabinowitz \cite{CR} and the global alternative of Rabinowitz \cite{Ra}, it is folklore that, for every integer $n\geq 1$, the set of non-trivial solutions $\mc{S}$
admits a connected component, $\mathscr{C}_n$, with $(\l,u)=(n^2,0)\in \mathscr{C}_n$, which is unbounded in
$\R\times \mc{C}^{2}_{0}[0,\pi]$. Moreover, by the maximum principle, since the number of nodes of
the solutions along $\mathscr{C}_n$ is constant, it turns out that
$
\mathscr{C}_n\cap \mathscr{C}_m=\emptyset$, $ n\neq m$. Note that, for each $n\geq 1$, we have that $\mathscr{C}_{n}\cap \mc{K}_{\mc{S}}\neq \emptyset$ and 
$\mc{S}=\mc{K}_{\mc{S}}\uplus \mc{R}_{\mc{S}}$, where $\mc{K}_{\mc{S}}$ is the skeleton of $\mc{S}$ and
$\mc{R}_{\mc{S}}$ its residual set. Thus, since $\mathscr{C}_{n}$ is a connected component of $\mc{S}$ and both, $\mc{K}_{\mc{S}}$ and $\mc{R}_{\mc{S}}$ are open, it becomes apparent that $\mathscr{C}_{n}\subset \mc{K}_{\mc{S}}$. This shows that actually each of the components $\mathscr{C}_n$, $n\geq 1$, consists of  a discrete set of analytic arcs of curve (edge points) plus a discrete set of branching points (vertex points).

\section{A degenerate one-dimensional problem}\label{SDODP}

\noindent In this section we apply the previous theory to the following
nonlinear one-dimensional boundary value problem
\begin{equation}
\label{vii.1}
\left\{\begin{array}{l}
-u''=\lambda u' +u+(\lambda-u^2)u^2 \quad \hbox{in}\;\, (0,\pi), \\
u(0)=u(\pi)=0.
\end{array}
\right.
\end{equation}
Considering the Hilbert spaces $U \equiv H^{2}(0,\pi)\cap H^{1}_{0}(0,\pi)$ and $V \equiv L^{2}(0,\pi)$,
with the inner products $\langle u,v \rangle_{V}:= \left( \int_0^\pi uv\right)^\frac{1}{2}$, $u, v \in V$, and
$$
   \langle u,v \rangle_{U}:=\langle u, v\rangle_V+\langle u', v'\rangle_V+ \langle u'', v''\rangle_V, \qquad u, v \in U,
$$
the solutions of \eqref{vii.1} can be viewed as the zeroes of the nonlinear operator $\mathfrak{F}:
\mathbb{R}\times U\to V$ defined by
\begin{equation*}
\mathfrak{F}(\lambda,u):=u''+\lambda u'+u+(\lambda-u^2)u^2, \qquad (\lambda,u)\in \mathbb{R}\times U.
\end{equation*}
Since $\mf{F}$ is polynomial in $\l$ and $u$, it is analytic in $(\l,u)\in\mathbb{R}\times U$ (see, e.g., Henry \cite{Henry}). Moreover, for every $(\l_0,u_0)\in \mf{F}^{-1}(0)$ and $u\in U$,
$$
  D_u\mf{F}(\l_0,u_0)u=u''+\l u'+u + 2\l_0 u_0 u -4u_0^3 u=u''+\l u'+u+ W(x)u,
$$
where $W=2\l_0u_0-4u_0^3$. Thus, by the Sturm--Liouville theory,  $N[D_u\mf{F}(\l_0,u_0)]$ is at most
one-dimensional. Moreover, by the Rellich--Kondrachov theorem, the embedding $J: U \hookrightarrow V$ is compact, and, thanks to the Lax--Milgram theorem, for every $\l\in\mathbb{R}$, there exists $\mu_0=\mu_0(\l)\in\mathbb{R}$ such that, setting $D=\frac{d}{dx}$,
\begin{equation}
\label{vii.2}
   D^2+\l D + (1+W)J+ \mu J \in GL(U,V)\quad \hbox{for all}\;\; \mu <\mu_0.
\end{equation}
Thus, since
$$
   D^2+\l D + (1+W)J =D^2+\l D + (1+W)J +\mu J-\mu J,
$$
is a compact perturbation of an invertible operator, it becomes apparent that
$D_u\mf{F}(\l_0,u_0) \in\Phi_{0}(U,V)$ for all $\lambda\in\mathbb{R}$. Therefore, the conditions (F1)--(F3) of Section \ref{SGG} are fulfilled. Consequently, as  a direct consequence of Theorem \ref{T6.4.8},
we have that $\mf{F}^{-1}(0)$ is an analytic graph of $\R\times U$. Note that
\begin{equation*}
\mathfrak{L}(\lambda)u:=D_{u}\mathfrak{F}(\lambda,0)u=u''+\lambda u'+u, \quad u\in U,
\end{equation*}
whose generalized spectrum,
$\Sigma(\mathfrak{L})$, consists of the values $\l\in\mathbb{R}$ for which the problem
\begin{equation}
\label{vii.3}
\left\{\begin{array}{ll}
-u''=\lambda u' +u \quad \hbox{in} \;\; (0,\pi), \\
u(0)=u(\pi)=0,
\end{array}
\right.
\end{equation}
admits a solution $u\neq 0$. Since $\mf{L}(\l)$ is analytic in $\l$, thanks to \eqref{vii.2}, we find from \cite[Th. 4.4.4]{LG01} that $\Sigma(\mf{L})$ is discrete. Obviously, $0\in \Sigma(\mf{L})$ and
$$
      N[\mathfrak{L}(0)]=\mathrm{span}[\varphi_{0}],\qquad \varphi_{0}(x):=\sqrt{\frac{2}{\pi}}\sin x,\quad x\in[0,\pi].
$$
Actually, $\Sigma(\mf{L})=\{0\}$. Indeed, the change of variable $u=e^{-\frac{\l}{2}x}v$ transforms \eqref{vii.3} into
\begin{equation}
\label{vii.4}
\left\{\begin{array}{ll}
-v''=\frac{4-\l^2}{4} v \quad \hbox{in} \;\; (0,\pi), \\
v(0)=v(\pi)=0, \end{array} \right.
\end{equation}
and hence, $\l\in \Sigma(\mf{L})$ if and only if $4-\l^2=4 n^2$ for some integer $n\geq 1$, whose unique real solution is $\l=0$ for the choice $n=1$.
\par
Our main goal is applying the results of Sections \ref{SLBA} and \ref{SAGA} for ascertaining the structure of $\mf{F}^{-1}(0)$. Since $\mathfrak{L}_{1}u\equiv \mathfrak{L}'(0)u=u'$ for all $u \in U$, we have that $\mathfrak{L}_{1}(\varphi_{0})=\phi_{0}$,  where $\phi_{0}(x)=\sqrt{\tfrac{2}{\pi}}\cos x$, $x\in [0,\pi]$. Thus, since $R[\mf{L}_0] = \left\{f\in L^{2}(0,\pi): \; \langle f, \v_0 \rangle_V =0 \right\}$,
it is apparent that $\phi_0 \in R[\mf{L}_0]$, where we are denoting $\mf{L}_0=\mf{L}(0)$. Thus,   $\l_0=0$ is not a $1$-transversal eigenvalue of $\mf{L}(\l)$, i.e., the transversality condition of Crandall and Rabinowitz (see \eqref{ii.4}) fails. Actually, since
$\mathfrak{L}^{(n)}(\lambda)=0$ for all $\l\in\mathbb{R}$ and $n\geq 2$, $0$ cannot be a transversal eigenvalue of $\mf{L}(\l)$ of any order. Indeed, the change of variable $u=e^{-\frac{\l}{2}x}v$, transforms the eigenvalue perturbation problem
\begin{equation*}
\left\{\begin{array}{ll}
u''+\l u'+u=\mu(\l) u \quad \hbox{in} \;\; (0,\pi), \\
u(0)=u(\pi)=0, \end{array} \right.
\end{equation*}
into
\begin{equation*}
\left\{\begin{array}{ll}
v''+\left(1-\frac{\l^{2}}{4}-\mu(\l)\right)v=0 \quad \hbox{in} \;\; (0,\pi), \\
v(0)=v(\pi)=0, \end{array} \right.
\end{equation*}
and provides us with the perturbed eigenvalue from $\l=0$ of $\mf{L}(\l)$, $\mu(\l)=-\frac{\l^{2}}{4}$. Thus, since $\mu(0)=\mu'(0)=0$ and $\mu''(0)\neq 0$, it follows from \cite[th. 4.3.3]{LGMC} that
$\chi[\mf{L},0]=2$.
Since $\chi=2$, Theorem \ref{T6.2.1} cannot guarantee the existence of a continuum emanating from $(0,0)\in\mathbb{R}\times U$. Thus, the techniques developed in
Sections \ref{SLBA} and \ref{SAGA} are imperative to analyze the structure of $\mf{F}^{-1}(0)$.
\subsection{Local structure of the solution set:}\label{SCh143} As in this setting $V=L^{2}(0,\pi)$, we have that $V' =L^{2}(0,\pi)$.
Thus,  the duality pairing $\langle\cdot,\cdot\rangle: V \times V'\to \R$ is given through
$\langle f,g \rangle :=\int_{0}^{\pi} fg \, dx$ for all $f, g\in L^{2}(0,\pi)$.
In this way, we can choose $\varphi_{0}^{\ast}=\varphi_{0}$ and in particular $\langle \varphi_{0}, \varphi_{0}^{\ast}\rangle=1$. Let us consider the pair $\mc{P}=(P,Q)$ of $\mf{L}_{0}$-projections,
\begin{align*}
	 P:L^{2}(0,\pi)\to N[\mathfrak{L}_{0}], \quad P(u):=\langle u,\varphi_{0}\rangle \varphi_{0}, \quad Q:L^{2}(0,\pi)\to R[\mathfrak{L}_0], \quad Q:=I_{V}-P.
\end{align*}
As already described in Section \ref{SectionALS}, performing a Lyapunov--Schmidt reduction to $\mf{F}(\l,u)=0$ at $(0,0)\in \R\times U$ under the pair of $\mf{L}_{0}$-projections $(P,Q)$, it is easily seen that there exist a neighborhood $\mc{U}$ of $(0,0)$ in $\R\times N[\mf{L}_{0}]$, a neighborhood $\mc{O}$ of $(0,0)$ in $\R\times U$, and an analytic operator $\mc{Y}:\mc{U}\to R[\mf{L}_{0}]$ such that
the maps
\begin{equation}
\label{Eq.LS}
\begin{split}
	\Psi:\mathfrak{F}^{-1}(0)\cap \mc{O}\to \mathfrak{G}^{-1}(0), & \quad (\lambda,u)\mapsto (\lambda,\langle u, \varphi_{0}\rangle),\\ \Psi^{-1}:\mathfrak{G}^{-1}(0)\to \mathfrak{F}^{-1}(0)\cap\mc{O}, & \quad (\lambda,z)\mapsto (\lambda,z\varphi_{0}+\mathcal{Y}(\lambda,z\varphi_{0})),
\end{split}
\end{equation}
are inverses of each other,  where the operator
$\mathfrak{G}: \mathscr{U}\subset\mathbb{R}^{2} \to \mathbb{R}$, $\mf{G}=\mf{G}(\l,z)$,
is given by
\begin{equation}
	\label{LSRANA,3} \mf{G}(\lambda,z):=\langle \mathfrak{F}(\lambda,z\varphi_{0}+\mathcal{Y}(\lambda,z\varphi_{0})),\varphi_{0}\rangle,
\end{equation}
with $\mathscr{U}:=\{(\l,z): (\l,z\varphi_{0})\in\mc{U}\}\subset \R^{2}$. Since $\mathfrak{G}(\lambda,0)=0$, there exists an analytic function $g: \mathscr{U}\to \mathbb{R}$ such that $\mathfrak{G}(\lambda,z)=zg(\lambda,z)$ for all $(\lambda,z)\in  \mathscr{U}$. To apply Theorem \ref{th4.2}, it is appropriate to express the operator $\mathfrak{F}$ in the form
\begin{equation*}
	\mathfrak{F}(\lambda,u)=\mathfrak{L}(\lambda)u+\lambda L^{1}_{2}u^{2}+L^{0}_{4}u^{4}, \qquad (\lambda,u)\in \mathbb{R}\times U,
\end{equation*}
where the symmetric operators $L^{1}_{2}\in\mc{S}^{2}(U,V)$, $L^{0}_{4}\in\mc{S}^{4}(U,V)$, are given by
\begin{align*}
	& L^{1}_{2}(u_{1},u_{2}):=u_{1} u_{2}, \quad u_{1},u_{2}\in U, \\
	& L^{0}_{q}(u_{1},u_{2}, u_{3}, u_{4}):=-u_{1}u_{2}u_{3}u_{4}, \quad u_{1},u_{2},u_{3},u_{4}\in U,
\end{align*}
which is consistent with the notations used in the expansion \eqref{iv.3}. In particular,
$$L^{1}_{2}u^{2}=u^{2}, \quad L^{0}_{4}u^{4}=-u^{4}, \quad u\in U.$$
According to Theorem \ref{th4.2}, and taking into account that $\mu(\l)=-\frac{\l^{2}}{4}$, it becomes apparent that
\begin{align*}
\mathfrak{G}(\lambda,z) & =z\Big(-\frac{1}{4}\lambda^{2}+\langle \varphi_{0}^2,\varphi_{0}
   \rangle z\lambda-\langle \varphi_{0}^{4},\varphi_{0}\rangle z^{3}+\sum_{j,k}C_{j,k}\lambda^{j}z^{k}\Big)\\ & =z\Big(-\frac{1}{4}\lambda^{2}+\frac{8}{3\pi}\sqrt{\frac{2}{\pi}}z\lambda -\frac{64}{15\pi^{2}}\sqrt{\frac{2}{\pi}}z^3+\sum_{j,k}C_{j,k}\lambda^{j}z^{k}\Big)= z g(\l,z),
\end{align*}
where the summation of the second sum is taken only on the points $(k,j)$ lying above the polygonal line joining $(0,2)$, $(1,1)$ and $(3,0)$. Owing the Newton--Puiseux algorithm, we obtain the following asymptotic expansion for the solutions of $g(\lambda,z)=0$ close to $(0,0)$,

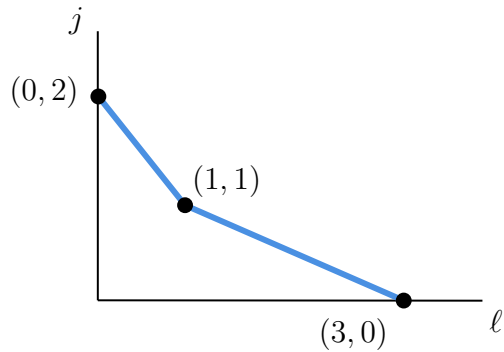
\begin{figure}[h!]
	\begin{center}

		\tikzset{every picture/.style={line width=0.75pt}} 
		
		\begin{tikzpicture}[x=0.75pt,y=0.75pt,yscale=-1,xscale=1]
			
			\draw    (55.98,24.19) -- (56,160) ;
			\draw    (56,160) -- (249.98,160.19) ;
			\draw [color={rgb, 255:red, 74; green, 144; blue, 226 }  ,draw opacity=1 ][line width=2.25]    (56.28,57.13) -- (100,112) ;
			\draw [color={rgb, 255:red, 74; green, 144; blue, 226 }  ,draw opacity=1 ][line width=2.25]    (100,112) -- (210.28,160.13) ;
			\draw  [fill={rgb, 255:red, 0; green, 0; blue, 0 }  ,fill opacity=1 ] (96.5,112) .. controls (96.5,110.07) and (98.07,108.5) .. (100,108.5) .. controls (101.93,108.5) and (103.5,110.07) .. (103.5,112) .. controls (103.5,113.93) and (101.93,115.5) .. (100,115.5) .. controls (98.07,115.5) and (96.5,113.93) .. (96.5,112) -- cycle ;
			\draw  [fill={rgb, 255:red, 0; green, 0; blue, 0 }  ,fill opacity=1 ] (52.78,57.13) .. controls (52.78,55.2) and (54.35,53.63) .. (56.28,53.63) .. controls (58.22,53.63) and (59.78,55.2) .. (59.78,57.13) .. controls (59.78,59.06) and (58.22,60.63) .. (56.28,60.63) .. controls (54.35,60.63) and (52.78,59.06) .. (52.78,57.13) -- cycle ;
			\draw  [fill={rgb, 255:red, 0; green, 0; blue, 0 }  ,fill opacity=1 ] (206.78,160.13) .. controls (206.78,158.2) and (208.35,156.63) .. (210.28,156.63) .. controls (212.22,156.63) and (213.78,158.2) .. (213.78,160.13) .. controls (213.78,162.06) and (212.22,163.63) .. (210.28,163.63) .. controls (208.35,163.63) and (206.78,162.06) .. (206.78,160.13) -- cycle ;
			
			\draw (102,90.4) node [anchor=north west][inner sep=0.75pt]    {$( 1,1)$};
			\draw (10,45.4) node [anchor=north west][inner sep=0.75pt]    {$( 0,2)$};
			\draw (166,167.4) node [anchor=north west][inner sep=0.75pt]    {$( 3,0)$};
			\draw (40,10.4) node [anchor=north west][inner sep=0.75pt]    {$j$};
			\draw (251.98,163.59) node [anchor=north west][inner sep=0.75pt]    {$\ell $};

		\end{tikzpicture}
	\end{center}
	\caption{Newton diagram of $g(\l,z)$}
	\label{F5}
\end{figure}
\noindent
\begin{equation}
	\label{iv.7}
	\begin{split}
		z(\lambda) & =\frac{1}{4\langle \varphi_{0}^{2},\varphi_{0}\rangle}\lambda+O(\lambda)=\frac{3\pi}{32}\sqrt{\frac{\pi}{2}}\lambda+O(\lambda)
		\quad \hbox{as}\;\;\l\to 0, \\
		z(\lambda) & =\pm \sqrt{\frac{\langle \varphi_{0}^{2},\varphi_{0}\rangle}{\langle \varphi_{0}^{4},\varphi_{0}\rangle}}\sqrt{\lambda}
		+O(\sqrt{\lambda})=\pm\sqrt{\frac{5\pi}{8}}\sqrt{\lambda}
		+O(\sqrt{\lambda}) \quad \hbox{as}\;\; \l\downarrow 0,
	\end{split}
\end{equation}
On the other hand, since $\chi[\mathfrak{L},0]=2$,
it follows from Lemma \ref{Lemma12.1.2} that $\chi[\mathfrak{L},0] = \ord_{\l=0}g(\l,0)=2$. Thus,
by the Weierstrass preparation theorem \cite[Th. 5.3.1]{BT}, shortening the neighborhood
$\mathscr{U}= \mathscr{U}_{\lambda}\times \mathscr{U}_{z}\subset\mathbb{R}^{2}$  if necessary,  there exists an analytic function $c: \mathscr{U}\to\mathbb{R}$ such that  $c(0,0)\neq 0$, plus $\chi=2$ analytic functions, $c_{j}: \mathscr{U}_{z}\to\mathbb{R}$ with $c_{j}(0)=0$ for $j=1,2$, such that
\begin{equation*}
	g(\lambda,z)=c(\lambda,z)\left[\lambda^{2}+c_{1}(z)\lambda+c_{2}(z)\right].
\end{equation*}
Hence for every $z\in  \mathscr{U}_{z}$, the equation $g(\lambda,z)=0$ has, at most, two solutions. This shows that indeed, the solutions \eqref{iv.7}, are the unique ones of $g(\lambda,z)=0$ in a neighbourhood of $(0,0)$. Figure \ref{Fig7.4} represents these branches.

\begin{center}
	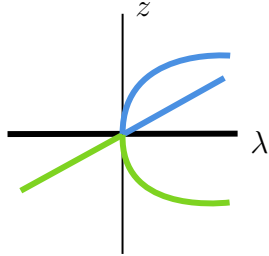
\begin{figure}[h!]

		\tikzset{every picture/.style={line width=0.75pt}} 
		
		\begin{tikzpicture}[x=0.75pt,y=0.75pt,yscale=-1,xscale=1]
			
			\draw    (89.04,18.78) -- (89,140) ;
			\draw [line width=2.25]    (147.04,79.28) -- (31,79.5) ;
			\draw [color={rgb, 255:red, 74; green, 144; blue, 226 }  ,draw opacity=1 ][line width=2.25]    (140.33,51.18) -- (89.02,79.88) ;
			\draw [color={rgb, 255:red, 126; green, 211; blue, 33 }  ,draw opacity=1 ][line width=2.25]    (89.02,79.88) -- (37.7,108.09) ;
			\draw [color={rgb, 255:red, 74; green, 144; blue, 226 }  ,draw opacity=1 ][line width=2.25]    (89.02,79.39) .. controls (89,36) and (141.59,39.77) .. (143.29,39.92) ;
			\draw [color={rgb, 255:red, 126; green, 211; blue, 33 }  ,draw opacity=1 ][line width=2.25]    (89.02,79.88) .. controls (88,118) and (134,115) .. (143,114) ;
			
			\draw (152.04,76.68) node [anchor=north west][inner sep=0.75pt]    {$\lambda $};
			\draw (94,11.4) node [anchor=north west][inner sep=0.75pt]    {$z$};

		\end{tikzpicture}
		\caption{$g^{-1}(0)$ in a neighbourhood of $(0,0)$}
		\label{Fig7.4}
	\end{figure}
\end{center}

\subsection{Global structure of the solution set:} Throughout the rest of this section, we will study the global structure of the set of positive solutions of the problem \eqref{vii.1}. The next result shows that any  small solution of \eqref{vii.1} must be either positive or negative.

\begin{lemma}
\label{levii.1} There exists $\e>0$ such that if $(\l,u)$ solves \eqref{vii.1} with
$|\l|+\|u\|_\infty\leq \e$, then, either $u>0$, or $u<0$.
\end{lemma}
\begin{proof}
On the contrary, suppose that there exists a sequence of solutions of \eqref{vii.1}, $\{(\l_n,u_n)\}_{n\geq 1}$, such that
\begin{equation}
\label{vii.8}
  \lim_{n\to \infty}(\l_n,u_n)=(0,0) \quad \hbox{in}\;\;\mathbb{R}\times \mathcal{C}[0,\pi]
\end{equation}
and $u_n$ changes sign in $(0,\pi)$; in particular, $u_n\neq 0$. Then, for every $n\geq 1$, the functions
$v_n(x):=e^{\frac{\l_n}{2}x}u_n(x)$, $x\in [0,\pi]$, satisfy $v_n(0)=v_n(\pi)=0$ and
\begin{equation}
\label{vii.9}
  v_n = \mathcal{K}\left[ \left(1-\frac{\l_n^2}{4}\right)v_n+\left( \l_n-
  e^{-\l_n x}v_n^2\right) e^{-\frac{\l_n}{2}x}v_n^2\right],
\end{equation}
where, for every $f\in\mathcal{C}[0,\pi]$, we have denoted
$$
  \mathcal{K}[f](x):= \int_{0}^x(s-x)f(s)\,ds-\frac{x}{\pi}\int_0^\pi(s-\pi)f(s)\,ds.
$$
Note that $u\equiv \mathcal{K}[f]$ is the unique solution of $-u''=f$ in $[0,\pi]$ such that $u(0)=u(\pi)=0$.  Thus, setting $\psi_n := \frac{v_n}{\|v_n\|_\infty}$, $n\geq 1$, and dividing \eqref{vii.9}
by $\|v_n\|_\infty$ yields to
\begin{equation}
\label{vii.10}
  \psi_n = \mathcal{K}\left[ \left(1-\frac{\l_n^2}{4}\right)\psi_n+\left( \l_n-
  e^{-\l_n x}v_n^2\right) e^{-\frac{\l_n}{2}x}v_n \psi_n \right],\qquad n\geq 1.
\end{equation}
By \eqref{vii.8} and the definition of the $v_n$'s, the sequence of continuous functions
$$
  f_n:= \left(1-\frac{\l_n^2}{4}\right)\psi_n+\left( \l_n-
  e^{-\l_n x}v_n^2\right) e^{-\frac{\l_n}{2}x}v_n \psi_n,\qquad n\geq 1,
$$
is bounded in $\mathcal{C}[0,\pi]$. Thus, since $\mathcal{K}:\mc{C}[0,\pi]\to \mc{C}^1[0,\pi]$
is a compact operator, along some subsequence, relabeled by $n$, one has that
$\lim_{n\to \infty}\psi_n = \psi$ in $\mc{C}^1[0,\pi]$. In particular, $\|\psi\|_\infty=1$.
On the other hand, letting $n\to \infty$ in \eqref{vii.10}, it becomes apparent that
$\psi= \mathcal{K}\psi$ and hence, either $\psi(x)=\sin x$ for all $x\in [0,\pi]$, or
$\psi(x)=-\sin x$ for all $x\in [0,\pi]$. Therefore, $v_n$ must be positive for sufficiently large $n$ if $\psi(x)=\sin x$, whereas it is negative in $(0,\pi)$ if $\psi(x)=-\sin x$. This contradicts the assumption that $u_n$, and so $v_n$, change sign as $n\to \infty$, and ends the proof.
\end{proof}

In the Lyapunov--Schmidt reduction above, by \eqref{Eq.LS}, we have $z=\langle u, \varphi_{0}\rangle$. Thus,
an integration by parts yields
\begin{equation*}
z=\langle u, \varphi_{0}\rangle=\int_{0}^{\pi}u \varphi_{0} \ dx+\int_{0}^{\pi}u'\varphi'_{0} \ dx+\int_{0}^{\pi}u''\varphi''_{0}=3\int_{0}^{\pi}u\varphi_{0}.
\end{equation*}
Hence, by Lemma \ref{levii.1}, it follows that the solutions of $\mathfrak{G}(\l,z)=0$ are positive for $z>0$ and negative for $z<0$. Therefore, according to the asymptotic expansions \eqref{iv.7}, for $\l>0$ there emanate from $(0,0)$ two branches of positive solutions and one branch of negative solutions, while another branch of negative solutions emanates for $\l<0$, as illustrated by Figure \ref{Fig7.4}. We have represented the positive solutions in blue colour while the negative ones have been represented in green colour.
\par
We conclude this section by analyzing the structure of the set of positive solutions of \eqref{vii.1}. The next result shows that $\l>0$ is necessary for the existence of a positive solution.

\begin{lemma}
\label{levii.2}
The problem \eqref{vii.1} cannot admit any positive solution if $\l\leq 0$.
\end{lemma}
\begin{proof}
Suppose \eqref{vii.1} has a positive solution, $u$. Then, multiplying \eqref{vii.1} by $u$ and integrating in $(0,\pi)$, it follows that
$$
  -\int_0^\pi u''u\,dx=\l\int_0^\pi u'u\,dx +\int_0^\pi u^2\,dx+\int_0^\pi (\l-u^2)u^3\,dx.
$$
Moreover, since $u(0)=u(\pi)=0$,
\begin{align*}
  \int_0^\pi u''u\,dx & =\int_0^\pi(u'u)'\,dx-\int_0^\pi (u')^2\,dx =-\int_0^\pi (u')^2\,dx,\\
  \int_0^\pi u'u\,dx & = \frac{1}{2}\int_0^\pi(u^2)'=0.
\end{align*}
Thus,
$$
   \int_0^\pi(u')^2 \,dx= \int_0^\pi u^2\,dx+\int_0^\pi (\l-u^2)u^3\,dx.
$$
On the other side, it is well known that $\int_0^\pi u^2\,dx \leq \int_0^\pi (u')^2\,dx$ for all
$u \in\mc{C}_0^1[0,\pi]$. Therefore, $\int_0^\pi (\l-u^2)u^3\,dx\geq 0$. Consequently,
$\l>0$, which ends the proof.
\end{proof}

The next result provides us with some useful estimates for the positive solutions of
\eqref{vii.1}.

\begin{lemma}
\label{levii.3}
Let $(\l,u)$ be a positive solution of \eqref{vii.1}. Then, $\l>0$ and
\begin{equation}
\label{vii.11}
  \|u\|_\infty \leq \sqrt{\l}+1.
\end{equation}
\end{lemma}
\begin{proof}
By Lemma \ref{levii.2}, $\l>0$. Let $x_0\in (0,\pi)$ be such that $u(x_0)=\|u\|_\infty = \max_{[0,\pi]}u$.
Then, $u'(x_0)=0$ and  $u''(x_0)\leq 0$. Thus,
$$
  0\leq -u''(x_0)=\l u'(x_0)+u(x_0)+(\l-u^2(x_0))u^2(x_0)=u(x_0) +(\l-u^2(x_0))u^2(x_0).
$$
Hence, since $u(x_0)>0$, $1+ (\l-u^2(x_0))u(x_0)\geq 0$. Consequently, setting
$P(z):=z^3-\l z -1$, $z\in\mathbb{R}$, we have that $P(\|u\|_\infty)\leq 0$. Since $P'(z)=3z^2-\l$, it is easily seen that the cubic polynomial $P(z)$ has a local maximum at $-\sqrt{\l/3}$, a local minimum at
$\sqrt{\l/3}$, and it satisfies $P(0)=-1<0$. Since  $P(\|u\|_{\infty})\leq0$, necessarily $\|u\|_{\infty}\in(0,\tau(\lambda))$, where $\tau(\lambda)$ is the unique positive root of $P$. Moreover, since
\begin{align*}
  P(\sqrt{\l}+1)  =(\sqrt{\l}+1)^3-\l(\sqrt{\l}+1)-1 =2\l+3\sqrt{\l}>0
\end{align*}
for all $\l>0$, then $\tau(\lambda)\leq \sqrt{\l}+1$. Therefore, \eqref{vii.11} holds.
\end{proof}

As a consequence of Lemma \ref{levii.3}, the next result holds.

\begin{lemma}
\label{levii.4}
Let $(\l,u)$ be a positive solution of \eqref{vii.1}. Then,
\begin{equation}
\label{vii.12}
   0<\l\leq 4(1+\sqrt{2})^2.
\end{equation}
\end{lemma}
\begin{proof}
The change of variable $v=e^{\frac{\l}{2}x}u$ transforms the problem \eqref{vii.1} into the next one
\begin{equation}
\label{vii.13}
  \left\{ \begin{array}{l} -v''= \frac{4-\l^2}{4}v+\left(\l-e^{-\l x}v^2\right) e^{-\frac{\l}{2} x}v^2,
  \quad x\in [0,\pi],\\[1ex] v(0)=v(\pi)=0. \end{array} \right.
\end{equation}
Multiplying the $v$-differential equation by $v$ and integrating by parts in $(0,\pi)$ yields to
$$
  \int_0^\pi (v')^2\,dx =\int_{0}^{\pi}\frac{4-\l^2}{4}v^2\,dx+\int_0^\pi  \left(\l-e^{-\l x}v^2\right) e^{-\frac{\l}{2} x}v^3\,dx.
$$
Thus, since $\int_0^\pi v^2 \leq \int_0^\pi (v')^2\,dx$, it is apparent that
$$
  \frac{\l^2}{4}\int_0^\pi v^2\,dx \leq \int_0^\pi \left(\l-e^{-\l x}v^2\right) e^{-\frac{\l}{2} x}v^3\,dx\leq \l \int_0^\pi u v^2\,dx.
$$
Therefore, combining this estimate with \eqref{vii.11}, we find that $\frac{\l}{4} \leq 1+\sqrt{\l}$. From this estimate, the proof is straightforward.
 \end{proof}

Subsequently, we will say that a function  $u\in\mathcal{C}^{1}_0[0,\pi]$ is strongly positive if $u(x)>0$ for every $x\in(0,\pi)$, $u'(0)>0$, and $u'(\pi)<0$. In such case, we simply write $u\gg 0$. Similarly, it is said that a function $u\in\mathcal{C}^{1}_0[0,\pi]$ is strongly negative if $-u\gg 0$, i.e., $u\ll 0$. We claim that any positive (resp. negative) solution of \eqref{vii.1} is strongly positive (resp. negative).  Indeed, suppose that $u\gneq 0$ is a positive solution
of \eqref{vii.1} such that either $u'(0)=0$, or $u'(\pi)=0$, or $u(\eta)=0$ for some $\eta\in (0,\pi)$.
Then, since $u'(\eta)=0$ in the latest case, in any of these cases there exists $x_0\in [0,\pi]$ for which $u$ satisfies the Cauchy problem
\begin{equation}
	\label{eq}
	\left\{\begin{array}{l}
	-u''=\l_0 u'+u+(\l-u^2)u^2 \quad  \text{ in } (0,\pi), \\
	u(x_0)=0,\;\; 	u'(x_0)=0.	\end{array} \right.
\end{equation}
By the Cauchy--Lipschitz theorem, there exists a unique maximal solution $u\in\mathcal{C}^{2}$ of \eqref{eq}. By uniqueness, $u=0$, which contradicts $u\gneq 0$. Therefore, $u\gg 0$. Similarly, any negative solution must be strongly negative. Based on this positivity result, the next result holds.

\begin{lemma}
	\label{levii.5}
	Let $\{(\l_{n},u_{n})\}_{n\in\mathbb{N}}\subset\mathfrak{F}^{-1}(0)$ be a sequence of positive (resp.  negative) solutions of \eqref{vii.1}, such that
$\lim_{n\to\infty}(\l_n,u_n)=(\l_0,u_0)\in\mathfrak{F}^{-1}(0)$ in $\mathbb{R}\times U$. Then, either $u_0\gg 0$, or $u_0=0$. In other words, the unique way to abandone the interior of the positive cone of the ordered Banach space $\mc{C}_0^1[0,\pi]$ is through $u=0$.
\end{lemma}
\begin{proof}
We will detail the proof of the result when $u_n\gg 0$, $n\geq 1$, since the case when $u_n\ll 0$ is analogous. Since $u_{n}\to u_0$ in $U$ as $n\to\infty$ and $U\hookrightarrow \mathcal{C}^{1}[0,\pi]$,  we have that  $u_{n}\to u_{0}$ in $\mathcal{C}^{1}[0,\pi]$ as $n\to\infty$. This implies that $u_0\geq 0$ in $[0,\pi]$. Thus, either $u_0 \gneq 0$, or $u_0\equiv 0$, and, should the first case occurs,  we already know that $u_0\gg 0$. This ends the proof.
\end{proof}

This section finalizes by proving that, there is a loop of positive solutions of \eqref{vii.1} emanating from $u=0$ at $\l=0$. The existence of a connected component of the set of positive solutions
$$\mathscr{S}:=\{(\l,u)\in\mathfrak{F}^{-1}(0): \;u\gg 0\}\subset \mathbb{R}\times U,$$
bifurcating from $(\lambda,u)=(0,0)$ has been already established. More precisely,
we already know that there emanate from $(0,0)$  two analytic arcs of positive solutions $\gamma_{i}:(0,\varepsilon)\to \mathbb{R}\times U$ of the form $\gamma_{i}(\lambda)=(\lambda,u_{i}(\lambda))$, with $\lim_{\lambda\downarrow 0}u_{i}(\lambda)=0$, $i\in\{1,2\}$.
The connected components of the set of positive solutions $\mathscr{S}$ containing to each of the curves $\gamma_{1}$ and $\gamma_{2}$, locally at $(0,0)$, will be called $\mathscr{C}^{+}_{1}$ and $\mathscr{C}^{+}_{2}$, respectively.

\begin{theorem}
	\label{th7.8}
	Under the previous assumptions, $\mathscr{C}^{+}_{1}=\mathscr{C}^{+}_{2}$. Moreover each of the local curves $\gamma_{i}:(0,\varepsilon)\to \mathbb{R}\times U$ can be continued to a global locally injective continuous curve $\Gamma_{i}: (0,T)\to\mathscr{C}^{+}_{i}$ such that $\Gamma_{i}|_{[T-\delta,T)}=\gamma_{j}$ for some $\delta>0$ and  $j\in\{1,2\}\backslash\{i\}$. Thus, there is a loop of positive solutions of \eqref{vii.1} with vertex at $(0,0)$.
\end{theorem}

\begin{proof}
	Once given the local curve $\gamma_{1}:(0,\varepsilon)\to \mathbb{R}\times U$ and the component  $\mathscr{C}^{+}_{1}$, in order to apply Theorem \ref{th.BT}, we should make sure that, for sufficiently small $\varepsilon>0$, the set $\gamma_{1}(0,\varepsilon)\subset \mathfrak{F}^{-1}(0)$
	consists of regular points of $\mathfrak{F}$. By the local analysis already done in Section \ref{SectionALS}, the regular and singular points of $\mathfrak{F}$ in $\mathfrak{F}^{-1}(0)\cap \mathcal{O}$ are in analytical correspondence with those of the reduced map $\mathfrak{G}(\l,z)=zg(\l,z)$, $(\l,z)\in\mathscr{U}$. So, it suffices to prove that, near $(0,0)$, the set $\mathfrak{G}^{-1}(0)$ does not contain any singular point of $\mathfrak{G}$ different form $(0,0)$. By the Weierstrass preparation theorem, shortening the neighborhood $\mathscr{U}=\mathscr{U}_{\lambda}\times\mathscr{U}_{z}\subset\mathbb{R}^{2}$,  if necessary,  there exists an analytic function $c:\mathscr{U}\to\mathbb{R}$ such that  $c(0,0)\neq 0$, plus $\chi=2$ analytic functions, $c_{j}:\mathscr{U}_{z}\to\mathbb{R}$, $c_{j}(0)=0$, $j=1,2$, such that
	\begin{equation*}
		g(\lambda,z)=c(\lambda,z)\left[\lambda^{2}+c_{1}(z)\lambda+c_{2}(z)\right].
	\end{equation*}
Hence, we can express $\mathfrak{G}:\mathscr{U}\to\mathbb{R}$ in the form
	$$
	\mathfrak{G}(\lambda,z)=zc(\lambda,z)\left[\lambda^{2}+c_{1}(z)\lambda+c_{2}(z)\right].
	$$
	By the local analysis already done in Section \ref{SCh143} (see Figure \ref{Fig7.4}), for every $z\in \mathscr{U}_{z}\backslash\{0\}$, the equation $\mathfrak{G}(\l,z)=0$ has two positive different solutions in $\lambda\in\mathscr{U}_{\lambda}$. Thus, there are two analytic maps,  $\varphi_{j}:(-\delta,\delta)\backslash\{0\}\to\mathbb{R}$, $j=1,2$, such that
	\begin{equation*}
		\mathfrak{G}(\l,z)=z c(\l,z)(\lambda-\varphi_{1}(z))(\l-\varphi_{2}(z)), \quad z\in\mathscr{U}_{z}\backslash\{0\}.
	\end{equation*}
	By a direct computation if follows that $(\l,z)\in\mathfrak{G}^{-1}(0)\cap\mathscr{U}$, $(\l,z)\neq(0,0)$,  is a singular point, i.e.,  $D_{z}\mathfrak{G}(\lambda,z)=0$,  if and only if $\varphi_{1}(z)=\varphi_{2}(z)$ or $\varphi'_{j}(z)=0$ for some $j=1,2$. According to
	\eqref{iv.7}, for sufficiently small $\mathscr{U}$, this is not possible. Therefore, $\gamma_{1}:(0,\varepsilon)\to\mathbb{R}\times U$ consists of regular points for
	sufficiently small $\varepsilon>0$. By Theorem \ref{th.BT}, $\gamma_{1}$ admits a prolongation
	to a global locally injective continuous map $\Gamma_{1}:[0,\infty)\to\mathbb{R}\times U$, $\Gamma_{1}([0,\infty))\subset \mathfrak{F}^{-1}(0)$, satisfying one of the alternatives (a) or (b). Due to Lemma \ref{levii.5}, $\Gamma_{1}([0,\infty))\subset\mathscr{C}^{+}_{1}$. Thanks to Lemmas \ref{levii.3} and \ref{levii.4}, $\Gamma_{1}([0,\infty))$ is bounded. Therefore, the alternative (a) cannot occur. Consequently, there exists some $T>0$ such that $\Gamma_{1}(T)=(0,0)$. As in a neighborhood of $(0,0)$ the set of positive solutions consists of
	the graphs of $\gamma_{1}$ and $\gamma_{2}$, being  $\Gamma_{1}$ is locally injective, it follows that, modulus a re-parametrization (if necessary), $\Gamma_{1}|_{(T-\delta,T]}=\gamma_{2}$. This implies, in particular, that $\mathscr{C}^{+}_{1}=\mathscr{C}^{+}_{2}$ and concludes the proof.
\end{proof}

\section{Unilateral bifurcation at geometrically simple eigenvalues}\label{SUBGSE}

\noindent This section is devoted to the study of the unilateral bifurcation problem. Throughout this section, we consider a pair $(U,V)$ of real Banach spaces such that
\begin{enumerate}
	\item[(C)] $U$  is a subspace of $V$ with compact inclusion  $U \hookrightarrow V$.
\end{enumerate}
We consider a map $\mathfrak{F}\in\mathcal{C}^1(\mathbb{R}\times U,V)$ satisfying
the following assumptions:
\begin{enumerate}
	\item[(F1)] $\mathfrak{F}(\lambda,0)=0$ for all $\lambda\in\mathbb{R}$.
	\item[(F2)] $D_{u}\mathfrak{F}(\lambda,u)\in\Phi_{0}(U,V)$ for all $\lambda\in\mathbb{R}$ and $u\in U$.
	\item[(F3)] $\mf{F}$ is proper on closed and bounded subsets of $\R\times U$.
	\item[(F4)] The map $\mf{N}(\l,u):= \mf{F}(\l,u)-D_u\mf{F}(\l,0)u$, $(\l,u)\in \R\times U$,
	admits a continuous extension, also denoted by $\mf{N}$, to $\R \times V$.
	\item[(F5)] $\l_0$ is an isolated eigenvalue of $\mf{L}(\l):=D_u\mf{F}(\l,0)$ such that
	$N[\mf{L}_{0}]=\mathrm{span}[\varphi_{0}]$ for some $\varphi_0\in U$ with  $\|\v_0\|=1$.
\end{enumerate}
Let us consider a closed subspace $Z\subset U$ such that $U=N[\mf{L}_0]\oplus Z$.
Then, by the Hahn--Banach theorem, there exists $\v_0^* \in U'$ such that
\begin{equation*}
	Z=\{u\in U\;:\; \langle\v_0^*,u\rangle=0\} =N[\v_0^*],\qquad \langle\v_0^*,\v_0\rangle =1,
\end{equation*}
where $\langle\cdot,\cdot\rangle$ stands for the $\langle U',U\rangle$-duality. In particular, every $u\in U$ admits a unique decomposition as $u=s\v_0+z$ for some $(s,z)\in\R\times Z$. Necessarily, $s:= \langle \v_0^*,u\rangle$. Let $\psi :U\to \mathbb{R}$ be a continuous functional such that, for some positive constants $0<C_{1}<C_{2}$,
\begin{equation}
	\label{6.6.47}
	\psi(0)=0,\qquad C_{1}\|u\|\leq \psi(u)\leq C_{2} \|u\|\quad \hbox{for all}\;\; u\in U.
\end{equation}
Then, for every $\e>0$ and $\eta \in (0,1/C_{2})$, we consider
\[
Q_{\e,\eta}:=\{(\l,u)\in\R\times U\;:\;|\l-\l_0|<\e,\;\; |\langle \v_0^*,u\rangle|>\eta \psi(u)\}.
\]
Since the mapping $u\mapsto |\langle \v_0^*,u\rangle|-\eta \psi(u)$ is continuous, $Q_{\e,\eta}$ is an open subset of $\R\times U$, and it consists of the two open subsets
\begin{align*}
	Q_{\e,\eta}^+ & :=\{(\l,u)\in\R\times U\;:\; |\l-\l_0|<\e,\;\; \langle \v_0^*,u\rangle >\eta \psi(u)\},\cr
	Q_{\e,\eta}^- & :=\{(\l,u)\in\R\times U\;:\; |\l-\l_0|<\e,\;\; \langle \v_0^*,u\rangle <- \eta \psi(u)\}.
\end{align*}
The following result establishes that, under the conditions of this section, the nontrivial solutions of $\mf{F}(\l,u)=0$ in a neighborhood of $(\l_0,0)$ must lie in $Q_{\e,\eta}$.  Note that $(\l_0,0)$ might not be a bifurcation point of $\mf{F}(\l,u)=0$ from $\mc{T}=\{(\l,0):\l\in\R\}$. We denote by $\mc{S}$ the set of non-trivial solutions of $\mf{F}(\l,u)=0$, which is given by
$\mc{S}=\left[ \mf{F}^{-1}(0)\backslash \mc{T}\right]\cup \{(\l,0):\;\l\in \Sigma(\mathfrak{L})\}$.
The first result of this section reads as follows.
\begin{proposition}
	\label{P6.6.1}
	Let $(U,V)$ be a pair of real Banach spaces satisfying {\rm (C)} and $\mf{F}\in\mc{C}^{1}(\R\times U,V)$ a map satisfying {\rm (F1)--(F5)}. Then, for sufficiently small $\e>0$, there exists $\d_0=\d_0(\eta)>0$ such that, for every $\d \in (0,\d_0)$, $[\mc{S}\setminus\{(\l_0,0)\}]\cap B_\d(\l_0,0)\subset Q_{\e,\eta}$.
\end{proposition}

\noindent The proof of Proposition \ref{P6.6.1} is based on the following lemma of technical nature.

\begin{lemma}
	\label{L6.6.2}
	Let $(U,V)$ be a pair of real Banach spaces satisfying {\rm (C)} and $\mf{F}\in\mc{C}^{1}(\R\times U,V)$ a map satisfying {\rm (F1)--(F5)}. Then, $\mf{N}:\R\times U\to V$ is a compact operator.
\end{lemma}
\begin{proof}
	Let $(\l_n,u_n)\in \R\times U$, $n\geq 1$, be  a bounded sequence. As $\{\l_n\}_{n\geq 1}$ is bounded in $\R$ we can extract a subsequence, relabeled by $n$, such that $\lim_{n\to \infty}\l_n =\l_\o$ for some $\l_\o\in\R$. According to (C), 	we also can extract a subsequence of $\{u_n\}_{n\geq 1}$, labeled again by $n$, such that $\lim_{n\to \infty}u_n=v_\o$ for some 	$v_\o \in V$. Thus, by (F3),
	$\lim_{n\to\infty} \mf{N}(\l_n,u_n)= \mf{N}(\l_\o,v_\o)$.
\end{proof}

\begin{proof}[Proof of Proposition \ref{P6.6.1}]
	By (F5), there is $\e_0>0$ such that $\Sigma(\mathfrak{L})\cap[\l_0-\e_0,\l_0+\e_0]=\{\l_0\}$.
	By \cite[Th. 2.1]{FP2}, there exists a parametrix $\mf{P} :[\l_0-\e_0,\l_0+\e_0]\to GL(V,U)$,
	of the restricted curve $\mf{L}:[\l_{0}-\varepsilon_{0},\l_{0}+\varepsilon_{0}]\to \Phi_{0}(U,V)$, such that
	$$
	\mc{K}(\l)\equiv I_U-\mf{P}(\l)\mf{L}(\l) \in \mc{K}(U)\quad \hbox{if}\;\; |\l-\l_0|\leq \e_0.
	$$
	As for $|\l-\l_0|\leq \e_0$  the equation $\mf{F}(\l,u)=0$ can be equivalently written as
	$\mf{P}(\l)\mf{F}(\l,u)=0$, 	it becomes apparent that $\mf{F}(\l,u)=0$ can be expressed as
	\begin{equation}
		\label{6.6.50}
		u-\mc{K}(\l)[u]+\mf{P}(\l)\mf{N}(\l,u)=0,\qquad |\l-\l_0|\leq \e_0,\quad u\in U.
	\end{equation}
	Since $\mc{K}(\l)\in \mc{K}(U)$ and, due to Lemma \ref{L6.6.2},
	$\mf{P}(\l)\mf{N}(\l,u):\R\times U \to U$, is compact, we have reduced our problem to the compact case. Thus, setting $\mathfrak{G}(\lambda,u):=\mf{P}(\l)\mf{F}(\l,u)$, $\mathfrak{K}(\lambda)u:=u-\mathcal{K}(\l)u$ and $\mathfrak{M}(\lambda,u):=\mf{P}(\l)\mf{N}(\l,u)$, it suffices to show the validity of the proposition for
	\begin{equation}
		\label{6.6.50,2}
		\mathfrak{G}(\l,u)=\mf{K}(\l)[u]+\mf{M}(\l,u)=0, \quad |\l-\l_0|\leq \e_0,\quad u\in U.
	\end{equation}
	To prove the result, we will argue by contradiction. Should not exist $\delta_{0}$ satisfying the desired requirements, there exist two sequences, $\delta_{n}$, $n\geq 1$, and $		 (\lambda_{n},u_{n})\in(\mc{S}\backslash\{(\lambda_{0},0)\})\cap B_{\delta_{n}}(\lambda_{0},0)$, $n\geq 1$,
	such that
	\begin{equation*}
		\lim_{n\to \infty}\d_n=0 \;\;\hbox{and}\;\; (\lambda_{n},u_{n})\notin Q_{\varepsilon,\eta}\;\;
		\hbox{for all}\;\; n\geq 1.
	\end{equation*}
	Since $(\l_n,u_n)\in B_{\delta_{n}}(\lambda_{0},0)$, we have that $|\lambda_{n}-\lambda_{0}|\leq \delta_{n}<\varepsilon$ for sufficiently large $n$, say $n\geq n_0$. Thus, we can infer from  $(\lambda_{n},u_{n})\notin Q_{\varepsilon,\eta}$ that
	\begin{equation}
		\label{6.6.51}
		|\langle \varphi^{\ast}_{0},u_{n}\rangle|\leq \eta\psi(u_{n})\quad \hbox{for all}\;\; n\geq n_0.
	\end{equation}
	Moreover, by construction, $(\lambda_{n},u_{n})\in\mc{S}\backslash\{(\lambda_{0},0)\}$. So,
	$(\lambda_{n},u_{n})\neq(\lambda_{0},0)$. On the other hand, as $\lambda_{0}$ is an isolated eigenvalue of $\Sigma(\mf{L})$, for sufficiently large $n$, $\l_n\notin\Sigma(\mf{L})$ and hence $u_{n}\neq 0$, because $(\lambda_{n},u_{n})\in \mc{S}$. Consequently, since $\psi(u)>0$ for all $u\in U\setminus\{0\}$, it follows from \eqref{6.6.50,2} that
	\begin{equation}
		\label{6.6.52}
		\frac{u_{n}}{\psi(u_{n})}=\mathcal{K}(\lambda_{n})\frac{u_{n}}{\psi(u_{n})}-
		\frac{\mathfrak{M}(\lambda_{n},u_{n})}{\psi(u_{n})}.
	\end{equation}
	Since $\mf{N}(\l,u)=o(\|u\|)$ as $u\to 0$,  by the continuity of $\mf{P}(\lambda)$, it follows from \eqref{6.6.47} and (F4) that
	\begin{equation*}
		\lim_{n\to \infty}\frac{\mathfrak{M}(\lambda_{n},u_{n})}{\psi(u_{n})}=\lim_{n\to \infty}\frac{\mf{P}(\lambda_n)\mathfrak{N}(\lambda_{n},u_{n})}{\psi(u_{n})}= \lim_{n\to \infty}\frac{\mf{P}(\lambda_n)\mathfrak{N}(\lambda_{n},u_{n})}{C \|u_{n}\|}=0,
	\end{equation*}
	for some constant $C>0$. Moreover, by \eqref{6.6.47} and the continuity of $\mc{K}(\l)$,
	\begin{equation*}
		\lim_{n\to \infty}\|[\mathcal{K}(\lambda_{n})-\mathcal{K}(\lambda_{0})]\frac{u_{n}}{\psi(u_{n})}\|=0.
	\end{equation*}
	Thus, letting $n\to \infty$ in \eqref{6.6.52} yields
	\begin{equation}
		\label{6.6.53}
		\lim_{n\to \infty}\|[I_{U}-\mathcal{K}(\lambda_{0})]\frac{u_{n}}{\psi(u_{n})}\|=0.
	\end{equation}
	On the other hand, by \eqref{6.6.47},
	\begin{equation}
		\label{6.6.54}
		\frac{1}{C_{2}}=\frac{\|u_{n}\|}{C_{2}\|u_{n}\|}\leq\left\|\frac{u_{n}}{\psi(u_{n})}\right\|\leq \frac{\|u_{n}\|}{C_{1}\|u_{n}\|}=\frac{1}{C_{1}}.
	\end{equation}
	Thus, the sequence $\frac{u_{n}}{\psi(u_{n})}$, $n\geq 1$,
	is bounded and hence, since $\mathcal{K}(\lambda_{0})$ is compact, along some subsequence, relabeled by $n$, we find that, for some
	$\v\in U$,
	\begin{equation}
		\label{6.6.55}
		\lim_{n\to \infty}\mathcal{K}(\lambda_{0})\left(\frac{u_{n}}{\psi(u_{n})}\right)=\varphi.
	\end{equation}
	Combining \eqref{6.6.53} with \eqref{6.6.55}, it becomes apparent that
	$\lim_{n\to \infty}\frac{u_{n}}{\psi(u_{n})}=\varphi$.
	By \eqref{6.6.54}, we have the bounds $\frac{1}{C_{2}}\leq\|\varphi\|\leq \frac{1}{C_{1}}$.
	Thus, letting $n\to\infty$ in \eqref{6.6.55} gives $\mathcal{K}(\lambda_{0})\varphi=\varphi$.
	Equivalently, $\v -\mf{P}(\l_0)\mf{L}(\l_0)[\v] = \v$, and, since $\mf{P}(\l_0)\in GL(U,V)$, this implies that $\mf{L}(\l_{0})[\varphi]=0$. Consequently, $\varphi=\pm\|\varphi\|\varphi_{0}$, where
	$\v_0$ is the generator of $N[\mf{L}(\l_0)]$ in (F5).  Finally, letting  $n\to\infty$ in \eqref{6.6.51}, we obtain that
	\begin{equation*}
		\frac{1}{C_{2}}\leq \|\varphi\|= |\langle \varphi_{0}^{\ast}, \varphi \rangle|= \lim_{n\to \infty}|\langle \varphi_{0}^{\ast}, \frac{u_{n}}{\psi(u_{n})}\rangle|= \lim_{n\to \infty}\frac{|\langle \varphi_{0}^{\ast}, u_{n}\rangle|}{\psi(u_{n})}\leq  \eta,
	\end{equation*}
	which implies $1/C_{2}\leq \eta$ and contradicts the choice of $\eta\in (0,1/C_{2})$. This contradiction shows the existence of $\delta_{0}$ for which $[\mc{S}\setminus\{(\l_0,0)\}]\cap B_\d(\l_0,0)\subset Q_{\e,\eta}$.
\end{proof}

\noindent The following result establishes the existence of unilateral components.

\begin{theorem}[\textbf{Unilateral components}]
	\label{P6.6.4}
	Let $(U,V)$ be a pair of real Banach spaces satisfying {\rm (C)} and $\mf{F}\in\mc{C}^{1}(\R\times U,V)$ a map satisfying {\rm (F1)--(F5)}. Suppose that for sufficiently small $\rho>0$,
	\begin{equation}
		\label{6.6.57}
		\chi[\mathfrak{L}_{\omega},[\lambda_{-},\lambda_{+}]]\in 2\mathbb{N}+1, \quad \l_{-}=\l_{0}-\rho, \ \l_{+}=\l_{0} + \rho,
	\end{equation}
	where $\mathfrak{L}_{\omega}\in\mathscr{C}^{\omega}([\lambda_{-},\lambda_{+}],\Phi_{0}(U,V))$ is any  analytic curve $\mc{A}$-homotopic to $\mf{L}(\l)$, $\lambda\in[\lambda_{-},\lambda_{+}]$. Then, there exists a component $\mf{C}$ of the set of nontrivial solutions $\mc{S}$, such that $(\l_0,0)\in \mf{C}$ and,
choosing $\varepsilon>0$ and $\delta>0$ as in Proposition \ref{P6.6.1}, there exist two closed connected subsets $\mf{C}_{\delta}^{+},\mf{C}_{\delta}^{-}\subset\mf{C}\cap B_{\delta}(\l_{0},0)$, such that $\mf{C}_{\delta}^{+}\subset Q_{\e,\eta}^+\cup\{(\l_{0},0)\}$,
  $\mf{C}_{\delta}^{-}\subset Q_{\e,\eta}^-\cup\{(\l_{0},0)\}$ and $\mf{C}_{\delta}^{+}\cap\mf{C}_{\delta}^{-}=\{(\l_{0},0)\}$. 	Moreover, $\mf{C}_{\delta}^{+}$ and $\mf{C}_{\delta}^{-}$ links $(\l_0,0)$ to $\p B_\d(\l_0,0)$.
\end{theorem}
\begin{proof}
	The existence of a component $\mf{C}$ of the set of nontrivial solutions $\mc{S}$ such that $(\l_0,0)\in \mf{C}$, follows from Theorem \ref{C6.3.6}.  Choose  $\varepsilon>0$ and $\delta>0$ as in Proposition \ref{P6.6.1}. Suppose no closed and connected subset $\mf{\tilde C}\subset \mf{C}\cap B_{\delta}(\l_{0},0)$ exists such that $\mf{\tilde C}\subset Q_{\e,\eta}^-\cup\{(\l_0,0)\}$, $(\l_0,0)\in \mf{\tilde C}$ and $\mf{\tilde C}\cap\partial B_\d(\l_0,0)\neq \emptyset$. Then, arguing as in the proof of Proposition \ref{P6.6.1}, by taking a parametrix $\mf{P}:[\l_{-},\l_{+}]\to GL(V,U)$ of $\mf{L}:[\l_{-},\l_{+}]\to\Phi_{0}(U,V)$, the equation $\mf{F}(\lambda,u)=0$ can be equivalently written in the form
	\begin{equation*}
		\mathfrak{G}(\l,u):=\mf{P}(\l)\mf{F}(\l,u)=\mf{K}(\l)[u]+\mf{M}(\l,u)=0, \quad \l\in[\l_{-},\l_{+}], \quad u\in U,
	\end{equation*}
	where we are using the same notation as in \eqref{6.6.50,2}. Subsequently, we define
	\begin{equation*}
		\mf{\hat M}(\l,u)= \left\{ \begin{array}{ll} \mf{M}(\l,u) & \quad \hbox{if}\;\; (\l,u)\in Q_{\e,\eta}^-,\\[1ex]  -\frac{\langle\v_0^*,u\rangle}{\eta \psi(u)}\mf{M}(\l,-\eta \psi(u)\v_0+z) &
			\quad \hbox{if}\;\; -\eta \psi(u) \leq
			\langle\v_0^*,u\rangle\leq 0, \ u\neq 0, \\[1ex] -\mf{\hat M}(\l,-u) & \quad \hbox{if}\;\; \langle\v_0^*,u\rangle\geq 0, \end{array} \right.
	\end{equation*}
	where $z$ stands for the projection of $u$ on $Z$ parallel to $\mathrm{span}[\v_0]$. Set
	\begin{equation}
		\label{DefGHA}
		\mf{\hat G}(\l,u):=\mf{K}(\l)[u]+\mf{\hat M}(\l,u),\qquad \l\in[\l_{-},\l_{+}],\quad u \in U.
	\end{equation}
	Clearly,  $\mf{\hat M}$ satisfies the same continuity and compactness properties as $\mf{M}$ and, in addition, it is odd in $u$. Thus, 	$\mf{\hat G}$ also is odd in $u$. However, note that $\mf{M}$ is not necessarily of class $\mc{C}^{1}$ and therefore the bifurcation theory of Sections \ref{SLBT} and \ref{SGBT} cannot be applied to this case. Nevertheless, thanks to the compactness, the classical bifurcation theory for compact operators collected are refined in \cite{LG01} can be applied.
	\par By \eqref{6.6.57}, it follows from Theorem \ref{th3.2} that $\sigma(\mf{L},[\l_{-},\l_{+}])=-1$. 	 Consequently, by the definition of the parity in terms of the parametrix, we obtain that
	$$\deg(\mf{P}(\l_{-})\mf{L}(\l_{-}))\cdot \deg(\mf{P}(\l_{+})\mf{L}(\l_{+}))=-1,$$
	where $\deg$ is the Leray--Schauder degree. Thus, the Leray--Schauder degree $\deg(\mf{K}(\l))$ changes as $\l$ crosses $\l_{0}$. Therefore, the bifurcation theorem for compact operators \cite[Th. 6.2.1]{LG01} yields the existence of a component, $\mf{\hat C}$, of nontrivial solutions of  $\mf{\hat G}(\l,u)=0$, such that $(\l_0,0)\in \mf{\hat C}$. By applying the same procedure as in the proof of Proposition \ref{P6.6.1}, shortening $\delta$ if necessary, we can get  $\mf{\hat C} \cap B_\d(\l_0,0) \subset Q_{\e,\eta}\cup\{(\l_0,0)\}$. Moreover, by applying \cite[Cor. 6.3.2]{LG01}, either $\mf{\hat C}$ is unbounded, or it contains a point $(\l_1,0)$ with $\l_1\neq \l_0$. In either case,
$\mf{\hat C}\cap \p B_\d(\l_0,0)\cap Q_{\e,\eta}\neq \emptyset$. On the other hand, as $\mf{\hat G}(\l,u)$ is odd in $u$, it is apparent that $\mf{\hat C}\cap Q_{\e,\eta}^+ =\{(\l,-u) : (\l,u)\in \mf{\hat C}\cap Q_{\e,\eta}^-\}$. Consequently, necessarily
	\begin{equation}
		\label{6.6.61}
		\mf{\hat C}\cap \p B_\d(\l_0,0)\cap Q_{\e,\eta}^- \neq \emptyset.
	\end{equation}
	Moreover, since $\mathfrak{G}(\lambda,u)=\mathfrak{F}(\lambda,u)$ for all $(\lambda,u)\in Q^{-}_{\varepsilon,\eta}$, it becomes apparent that $\mf{\hat C}\cap Q^{-}_{\varepsilon,\eta}=\mathfrak{C}\cap Q^{-}_{\varepsilon,\eta}$. Therefore, \eqref{6.6.61} can be
	equivalently expressed in the form $\mf{C}\cap \p B_\d(\l_0,0)\cap Q_{\e,\eta}^- \neq \emptyset$,
	which contradicts our first assumption. Similarly, the result follows for $Q^{+}_{\varepsilon,\eta}$.
\end{proof}

\noindent Under the hypothesis of Theorem \ref{P6.6.4}, we consider the closed and connected subsets of $\mc{S}$, $\mf{C}^{+}:=\mathscr{C}^{+}\cup \{(\l_{0},0)\}$ and $\mf{C}^{-}:=\mathscr{C}^{-}\cup \{(\l_{0},0)\}$, where $\mathscr{C}^{+}$ (resp. $\mathscr{C}^{-}$) is the connected component of $\mc{S}\backslash\{(\l_{0},0)\}$ containing $\mf{C}_{\delta}^{+}\backslash\{(\l_{0},0)\}$ (resp. $\mf{C}^{-}_{\delta}\backslash\{(\l_{0},0)\}$). The existence of these components is guaranteed by Theorem \ref{P6.6.4}. In particular $\mf{C}^{\pm}_{\delta}\subset \mf{C}^{\pm}\subset\mf{C}$. The sets $\mf{C}^{\pm}$ are called the \textit{unilateral components} of $\mf{C}$.

\begin{theorem}[\textbf{Unilateral global alternative}]
	\label{T6.6.5}
	Let $(U,V)$ be a pair of real Banach spaces satisfying {\rm (C)} and $\mf{F}\in\mc{C}^{1}(\R\times U,V)$ a map satisfying {\rm (F1)--(F5)}. Suppose that, for sufficiently small $\rho>0$,
	\begin{equation}
		\label{6.6.57,2}
		\chi[\mathfrak{L}_{\omega},[\lambda_{-},\lambda_{+}]]\in 2\mathbb{N}+1, \quad \l_{-}=\l_{0}-\rho, \ \l_{+}=\l_{0} + \rho,
	\end{equation}
	where $\mathfrak{L}_{\omega}\in\mathscr{C}^{\omega}([\lambda_{-},\lambda_{+}],\Phi_{0}(U,V))$ is any  analytic curve $\mc{A}$-homotopic to $\mf{L}(\l)$, $\lambda\in[\lambda_{-},\lambda_{+}]$.
	Then, for each $\nu \in \{-,+\}$, the unilateral component $\mf{C}^\nu$ whose existence has been established in Theorem \ref{P6.6.4}, satisfies some of the following alternatives:
	\begin{enumerate}
		\item[{\rm (i)}] $\mf{C}^\nu$ is not compact in $\R\times U$.
		\item[{\rm (ii)}] There exists $\l_1\neq \l_0$ such that $(\l_1,0)\in \mf{C}^\nu$.
		\item[{\rm (iii)}]  There exist $\l\in\R$ and $z\in Z\setminus\{0\}$ such that $(\l,z)\in \mf{C}^\nu$.
	\end{enumerate}
\end{theorem}

\begin{proof}
	It proceeds by contradiction. Assume, for example, that $\mathfrak{C}^{-}$ does not satisfy any of the three alternatives (i)-(iii).  Then, it satisfies the following properties:
	\begin{enumerate}
		\item[(a)] $\mathfrak{C}^{-}$ is bounded in $\mathbb{R}\times U$.
		\item[(b)] For sufficiently small $\rho>0$, the component $\mathfrak{C}^{-}$ is bounded away from
		$\{(\lambda,0)\in \R\times U:\;  |\lambda-\lambda_{0}|\geq \rho\}$. 
		\item[(c)] $\mathfrak{C}^{-}\cap [\mathbb{R}\times(Z\backslash\{0\})]=\emptyset$.
	\end{enumerate}
	According to (a), since $\mathfrak{F}$ is proper on closed and bounded sets, the component $\mathfrak{C}^{-}$ is compact. We now show that there exists $\eta_{0}\in(0,1/C_{2})$ such that, for every $0<\eta<\eta_{0}$,
	\begin{equation}
		\label{6.6.62}
		\mathfrak{C}^{-}\subset Q^{-}_{\varepsilon,\eta}\cup\{(\lambda_{0},0)\}.
	\end{equation}
	Observe that $Q^{-}_{\varepsilon,\eta}\subset Q^{-}_{\varepsilon,\tilde\eta}$ if $\eta>\tilde\eta$.
	Indeed, for any given $(\lambda,u)\in Q^{-}_{\varepsilon,\eta}$, since $\psi(u)\geq 0$, we have that
$\langle \varphi_{0}^{\ast},u\rangle < -\eta \psi(u) <-\tilde\eta \psi(u)$,
	and hence $(\lambda,u)\in Q^{-}_{\varepsilon,\tilde\eta}$. To prove \eqref{6.6.62} we proceed by
	contradiction. So, suppose \eqref{6.6.62} fails.  Then, there exists a sequence, $\{\eta_n\}_{n\geq 1}$, such that $\lim_{n\to\infty}\eta_n=0$ for which  $\mf{C}^-$ must leave $Q_{\e,\eta_n}^-$. Thus,
	by Proposition \ref{P6.6.1}, there exist $\delta>0$ and a sequence, $(\lambda_{n},u_{n})\in\mathfrak{C}^{-}\backslash B_{\delta}(\lambda_{0},0)$, $n\geq 1$,
	such that, for sufficiently large $n\geq 1$,
	\begin{equation}
		\label{6.6.64}
		\langle \varphi_{0}^{\ast},u_{n}\rangle =-\eta_{n} \psi(u_{n}),
	\end{equation}
	because $\mathfrak{C}^{-}$ leaves $Q^{-}_{\varepsilon,\eta_{n}}$ outside $B_{\delta}(\lambda_{0},0)$ for sufficiently large $n$. Moreover, since $\mathfrak{C}^{-}$ is compact, there exists a subsequence of $(\lambda_{n},u_{n})$, labeled again by $n$, such that $\lim_{n\to \infty} (\lambda_{n},u_{n})= (\lambda_{\omega},u_{\omega})\in\mathfrak{C}^{-}$. Since $(\lambda_{n},u_{n})\in\mathfrak{C}^{-}\backslash B_{\delta}(\lambda_{0},0)$, $n\geq 1$, it is apparent that $(\lambda_{\omega},u_{\omega})\notin B_{\frac{\delta}{2}}(\lambda_{0},0)$. Moreover, thanks to (b), $u_{\omega}\neq 0$. Furthermore, since $\mathfrak{C}^{-}$ is bounded and $\psi$ is continuous, letting $n\to\infty$ in \eqref{6.6.64} yields to $\langle \varphi^{\ast}_{0},u_{\omega}\rangle =0$ and hence $u_{\omega}\in Z$. Therefore, $	 (\lambda_{\omega},u_{\omega})\in\mathfrak{C}^{-}\cap[\mathbb{R}\times(Z\backslash\{0\})]$,
	which contradicts (c). This contradiction shows the existence of $\eta_{0}>0$ such that \eqref{6.6.62} holds for all $\eta\in(0,\eta_{0})$.
	\par
	Now, pick $\eta\in (0,\eta_0)$ and consider the map $\hat{\mathfrak{G}}:[\l_{-},\l_{+}]\times U\to U$
	defined by \eqref{DefGHA} in the proof of Theorem \ref{P6.6.4}. Since the Leray-Schauder degree $\deg(\mf{K}(\l))$ changes as $\l$ crosses $\l_0$, by \cite[Th. 6.2.1]{LG01}, there exists a component $\hat{\mathfrak{C}}$ of the set of nontrivial solutions of $\hat{\mathfrak{G}}(\l,u)=0$ bifurcating from $(\lambda_{0},0)$. Since $\mf{\hat C}\cap Q^{-}_{\varepsilon,\eta}=\mathfrak{C}\cap Q^{-}_{\varepsilon,\eta}$, by \eqref{6.6.62}, 	it follows from the oddness of $\hat{\mf{G}}(\lambda,u)$ in $u$ that
	\begin{equation}
		\label{6.6.65}
		\hat{\mf{C}}=\mf{C}^{-}\cup \{(\l,-u): (\l,u)\in \mf{C}^{-}\}.
	\end{equation}
	Lastly, according to (b), we have that $\hat{\mf{C}}\cap \{(\l,0):\l\in\Sigma\}=\{(\l_{0},0)\}$.
	Therefore, we infer from \cite[Cor. 6.3.2]{LG01} that $\hat{\mf{C}}$ is unbounded in $\mathbb{R}\times U$. By \eqref{6.6.65}, also $\mf{C}^{-}$ must be  unbounded in $\mathbb{R}\times U$, which contradicts (a) and ends the proof.
\end{proof}

When the compactness assumption {\rm (C)} is removed for our list of hypothesis, things do not work
as before, by the failure of the compactness arguments involved  in the proof of Proposition \ref{P6.6.1}. Nevertheless, these technical difficulties can ve overcome if $\lambda_{0}$ is a $1$-transversal eigenvalue and, in addition, $\psi$ is differentiable. In the special case when $\psi(u)=\|u\|$,
our corresponding results provides us with the unilateral theorem of Shi and Wang
\cite{XW}, whose proof follows the same general patterns as the proof of Theorem 6.4.3 of L\'{o}pez-G\'{o}mez \cite{LG01}.
\par
Subsequently, we consider two real Banach spaces, $U, V$,  and an operator $\mathfrak{F}\in\mathcal{C}^1(\mathbb{R}\times U,V)$ satisfying
the following assumptions:
\begin{enumerate}
	\item[(F1)] $\mathfrak{F}(\lambda,0)=0$ for all $\lambda\in\mathbb{R}$.
	\item[(F2)] $\mf{F}$ is orientable.
	\item[(F3)] $D_{u}\mathfrak{F}(\lambda,u)\in\Phi_{0}(U,V)$ for all $\lambda\in\mathbb{R}$ and $u\in U$.
	\item[(F4)] $\mf{F}$ is proper on closed and bounded subsets of $\R\times U$.
	\item[(F5)] $\l_0$ is an isolated eigenvalue of $\mf{L}(\l):=D_u\mf{F}(\l,0)$ such that
	$N[\mathfrak{L}(\lambda_{0})]=\mathrm{span}[\varphi_{0}]$ for some $\varphi_0\in U$ with  $\|\v_0\|=1$ and $\mathfrak{L}'(\lambda_{0})\varphi_{0}\notin R[\mathfrak{L}(\lambda_{0})]$, i.e., $\chi[\mathfrak{L},\l_0]=1$.
\end{enumerate}
We keep the same notations as above for $Z$, $\mc{S}$, $Q_{\varepsilon,\eta}$ and $Q^{\pm}_{\varepsilon,\eta}$. By (F5), one can adapt the proof of the main theorem of Crandall and Rabinowitz \cite{CR},
 to show that there exist $\varepsilon>0$ and two continuous functions, $\mu: (-\varepsilon,\varepsilon)\to \mathbb{R}$ and $\phi: (-\varepsilon,\varepsilon)\to Z$,
such that $\mu(0)=\lambda_{0}$, $\phi(0)=0$, and $\mathfrak{F}^{-1}(0)$ consists, in a neighborhood of $(\lambda_{0},0)$, of the trivial branch $\{(\lambda,0):\lambda\sim\lambda_{0}\}$ and the curve
$(\mu(s),s[\varphi_{0}+\phi(s)])$, $s\sim 0$. In this case, since the component emanating from $(\lambda_{0},0)$ has locally the  form $(\mu(s),s[\varphi_{0}+\phi(s)])$, for sufficiently small $\e>0$, there exists $\d_0=\d_0(\eta)>0$ such that, for every $\d \in (0,\d_0)$,
\begin{equation}
	\label{E3}
	[\mc{S}\setminus\{(\l_0,0)\}]\cap B_\d(\l_0,0)\subset Q_{\e,\eta}.
\end{equation}

\begin{theorem}
	\label{th8.6}
	Suppose $\mf{F}$ satisfies {\rm (F1)--(F5)} and $\psi$ is of class $\mathcal{C}^{1}$. Then, the set of nontrivial solutions, $\mc{S}$, possesses a connected component $\mf{C}$ such that $(\l_0,0)\in \mf{C}$ and, choosing  $\varepsilon>0$ and $\delta>0$ as in \eqref{E3}, there exist two closed connected subsets
		$\mf{C}_{\delta}^{+},\mf{C}_{\delta}^{-}\subset\mf{C}\cap B_{\delta}(\l_{0},0)$ such that
		$\mf{C}_{\delta}^{+}\subset Q_{\e,\eta}^+\cup\{(\l_{0},0)\}$, $\mf{C}_{\delta}^{-}\subset Q_{\e,\eta}^-\cup\{(\l_{0},0)\}$ and $\mf{C}_{\delta}^{+}\cap\mf{C}_{\delta}^{-}=\{(\l_{0},0)\}$. 		 Moreover, $\mf{C}_{\delta}^{+}$ and $\mf{C}_{\delta}^{-}$ links $(\l_0,0)$ to $\p B_\d(\l_0,0)$.
Furthermore, setting $\mf{C}^{+}:=\mathscr{C}^{+}\cup \{(\l_{0},0)\}$ and $\mf{C}^{-}:=\mathscr{C}^{-}\cup \{(\l_{0},0)\}$,  where $\mathscr{C}^{+}$ {\rm(} resp. $\mathscr{C}^{-}${\rm)} is the connected component of $\mc{S}\backslash\{(\l_{0},0)\}$ containing $\mf{C}_{\delta}^{+}\backslash\{(\l_{0},0)\}$ {\rm(}resp. $\mf{C}^{-}_{\delta}\backslash\{(\l_{0},0)\}${\rm)}, then
     for every $\nu \in \{-,+\}$, some of the following alternatives hold:
	\begin{enumerate}
		\item[{\rm (i)}] $\mf{C}^\nu$ is not compact in $\R\times U$.
		\item[{\rm (ii)}] There exists $\l_1\neq \l_0$ such that $(\l_1,0)\in \mf{C}^\nu$.
		\item[{\rm (iii)}]  There exist $\l\in\R$ and $z\in Z\setminus\{0\}$ such that $(\l,z)\in \mf{C}^\nu$.
	\end{enumerate}
\end{theorem}
\begin{proof}
The existence of the connected component $\mf{C}$ is evident from the adaptation of 
the main theorem of \cite{CR}.  Let us choose $\varepsilon>0$ and $\delta>0$ as in \eqref{E3}. Suppose no closed and connected subset $\mf{\tilde C}\subset \mf{C}\cap B_{\delta}(\l_{0},0)$ exists such that $\mf{\tilde C}\subset Q_{\e,\eta}^-\cup\{(\l_0,0)\}$,
$(\l_0,0)\in \mf{\tilde C}$ and $\mf{\tilde C}\cap\partial B_\d(\l_0,0)\neq \emptyset$. Next, setting $\mf{N}(\l,u):= \mf{F}(\l,u)-\mf{L}(\l)u$, for $(\l,u)\in \R\times U$, we define
	\begin{equation}
		\label{viii.14}
		\mf{\hat N}(\l,u):= \left\{ \begin{array}{ll} \mf{N}(\l,u) & \quad \hbox{if}\;\; (\l,u)\in Q_{\e,\eta}^-,\\[1ex]  \xi\left(-\frac{\langle\varphi_{0}^{\ast},u\rangle}{\eta\psi(u)}\right)\mathfrak{N}(\lambda,u) &
			\quad \hbox{if}\;\; -\eta \psi(u) \leq
			\langle\v_0^*,u\rangle\leq 0, \ u\neq 0, \\[1ex] -\mf{\hat N}(\l,-u) & \quad \hbox{if}\;\; \langle\v_0^*,u\rangle\geq 0, \end{array} \right.
	\end{equation}
	where $\xi:\mathbb{R}\to\mathbb{R}$ is a fixed increasing $\mathcal{C}^{1}$ function satisfying $\xi(0)=0, \xi(1)=1$ and $\xi'(1)=0$. Now, set $\mf{\hat F}(\l,u):=\mf{L}(\l)u+\mf{\hat N}(\l,u)$, $(\l,u)\in\R\times U$. 	 The same argument of the proof of \cite[Th.4.4]{XW} shows that  $\hat{\mf{F}}$ is of class $\mathcal{C}^{1}$ and $D_{u}\hat{\mathfrak{F}}(\lambda,u)\in \Phi_{0}(U,V)$ for all  $(\l,u)\in\mathbb{R}\times U$. Moreover, $\mf{\hat F}$ is odd in $u$. Since $\mathfrak{L}'(\lambda_{0})[\varphi_{0}]\notin R[\mathfrak{L}(\lambda_{0})]$, it follows that  $\chi[\mf{L},\l_{0}]=1$. Hence, by Theorem \ref{T6.2.1}, there exists a component,
	$\mf{\hat C}$, of nontrivial solutions of  $\mf{\hat F}(\l,u)=0$ such that $(\l_0,0)\in \mf{\hat F}$. By applying the same argument we follow to establish \eqref{E3}, we can prove that
	\begin{equation}
		\mf{\hat C} \cap B_\d(\l_0,0) \subset Q_{\e,\eta}\cup\{(\l_0,0)\}
	\end{equation}
	for sufficiently small $\d$, say $\d \in (0,\d_0]$. Moreover, since $\chi[\mf{L},\l_{0}]=1$, according to Theorem \ref{C6.3.6}, either $\mf{\hat C}$ is unbounded, or it contains a point $(\l_1,0)$ with $\l_1\neq \l_0$. Note that the hypothesis (F2) is necessary to apply Theorem \ref{C6.3.6}. In either case,
$\mf{\hat C}\cap \p B_\d(\l_0,0)\cap Q_{\e,\eta}\neq \emptyset$. On the other hand, as $\mf{\hat F}(\l,u)$ is odd in $u$, it is apparent that $\mf{\hat C}\cap Q_{\e,\eta}^+ =\{(\l,-u) : (\l,u)\in \mf{\hat C}\cap Q_{\e,\eta}^-\}$. Therefore, necessarily
	\begin{equation}
		\label{viii.16}
		\mf{\hat C}\cap \p B_\d(\l_0,0)\cap Q_{\e,\eta}^- \neq \emptyset.
	\end{equation}
	Moreover, since $\mathfrak{F}(\lambda,u)=\hat{\mathfrak{F}}(\lambda,u)$ for all $(\lambda,u)\in Q^{-}_{\varepsilon,\eta}$, it becomes apparent that $\mf{\hat C}\cap Q^{-}_{\varepsilon,\eta}=\mathfrak{C}\cap Q^{-}_{\varepsilon,\eta}$. Therefore, \eqref{viii.16} can be
	equivalently expressed in the form $\mf{C}\cap \p B_\d(\l_0,0)\cap Q_{\e,\eta}^- \neq \emptyset$,
	which contradicts our first assumption.  The same procedure holds for $Q^{+}_{\varepsilon,\eta}$. This ends the proof of the first of the theorem. The second part follows \textit{mutatis mutandis} from the proof of Threorem \ref{T6.6.5} interchanging the operator $\hat{\mf{G}}$ with $\hat{\mf{F}}$.
\end{proof}

\section{The negative solutions of \eqref{vii.1}}\label{SNS}

\noindent In this section we complete the analysis of Section \ref{SDODP} by studying the global structure of the component of negative solutions of \eqref{vii.1} emanating from $(0,0)$. We have postponed this analysis because we need the unilateral theorems of Section \ref{SUBGSE} to get our result.

\begin{lemma}
\label{le9.1}
$u=0$ is the unique solution of \eqref{vii.1} at $\l=0$.
\end{lemma}

\begin{proof}
Suppose $\l=0$. Then, multiplying  \eqref{vii.1} by $\varphi_{0}=\sin x$ and integrating by parts gives
	\begin{equation*}
\int_{0}^{\pi}u\varphi_{0}\ dx =	-\int_{0}^{\pi}u''\varphi_{0}\ dx=\int_{0}^{\pi}u\varphi_{0}\ dx-\int_{0}^{\pi}u^{4}\varphi_{0}\ dx.
	\end{equation*}
Thus, $\int_{0}^{\pi}u^{4}\varphi_{0}\ dx=0$, and since $\varphi_{0}>0$, it follows that $u=0$. This ends the proof.
\end{proof}

Observe that $(\l,u)$ is a negative solution of \eqref{vii.1} if and only if $(\l,v)=(\l,-u)$ is a positive solution of the problem
\begin{equation}
\label{9.1}
\left\{\begin{array}{l}
-v''=\l v'+v-(\l-v^2)v^2 \quad \text{ in } (0,\pi), \\
v(0)=v(\pi)=0.
\end{array}\right.
\end{equation}
Therefore, the problem of analyzing the negative solutions of \eqref{vii.1} can be solved by studying the positive solutions of \eqref{9.1}.

\begin{theorem}
\label{th9.2}
	Suppose $\l<0$. Then, the problem \eqref{vii.1} admits a negative solution.
\end{theorem}
\begin{proof}
The change of variable $w=e^{\frac{\l}{2}x}v$  transforms \eqref{9.1} into
\begin{equation}
	\label{9.2}
	\left\{\begin{array}{l}
	-w''=\frac{4-\l^{2}}{4}w-\l e^{-\frac{1}{2}\l x}w^{2}+e^{-\frac{3}{2}\l x}w^{4} \quad \text{ in } (0,\pi), \\
	w(0)=w(\pi)=0
	\end{array}\right.
	\end{equation}
and setting $\mu=-\l$, this problem can be expressed as
\begin{equation*}
	\left\{\begin{array}{l}
	-w''=\frac{4-\mu^{2}}{4}w+\mu e^{\frac{1}{2}\mu x}w^{2}+e^{\frac{3}{2}\mu x}w^{4} \quad \text{ in } (0,\pi), \\ 	w(0)=w(\pi)=0, 	\end{array}\right.
\end{equation*}
To study this problem, we will deal with its  generalization
\begin{equation}
	\label{9.3}
	\left\{\begin{array}{l}
	-w''=\rho w+\alpha(x)w^{2}+\beta(x)w^{4} \quad \text{ in } (0,\pi), \\
	w(0)=w(\pi)=0,
	\end{array}\right.
\end{equation}
where $\rho\in\mathbb{R}$ is regarded as a bifurcation parameter and $\alpha,\beta\in\mathcal{C}[0,\pi]$ satisfy $\a(x)>0$ and $\b(x)>0$ for all $x\in[0,\pi]$. The solutions of \eqref{9.3} are the zeroes of the nonlinear operator
\begin{equation*}
	\mathfrak{F}:\mathbb{R}\times U \to V, \quad \mathfrak{F}(\rho,w)=w''+\rho w+\alpha(x)w^{2}+\beta(x)w^{4}.
\end{equation*}
According to the main theorem of Crandall and Rabinowitz \cite{CR}, it is folklore that, for every integer $n\geq 1$,  $\rho_{n}=n^{2}$ is a bifurcation point to a curve of solutions with $n-1$ nodes. In particular, a curve of positive solutions emanates from $\rho=1$.
\par
Suppose that $w$ is a positive solution of \eqref{9.3}. Then, multiplying \eqref{9.3} by $\varphi_{0}(x)=\sin x$ and integrating by parts in $(0,\pi)$ yields
\begin{equation*}
	0=(\rho-1)\int_{0}^{\pi}w\varphi_{0}\ dx+\int_{0}^{\pi}\alpha(x)w^{2}\varphi_{0}\ dx+\int_{0}^{\pi}\beta(x)w^{4}\varphi_{0}\ dx>(\rho-1)\int_{0}^{\pi}w\varphi_{0}\ dx.
\end{equation*}
	Therefore, since $u$ and $\varphi_{0}$ are positive, necessarily $\rho<1$. Consequently,
\eqref{9.3} cannot admit any positive solution if $\r\geq 1$.
\par
Let us denote by $\mathscr{C}$ the component of positive solutions of \eqref{9.3} emanating from $w=0$ at $\rho=1$. We already know that $\mathscr{C}\subset (-\infty,1]\times U$. In particular, $\mathscr{C}$ cannot meet any other bifurcation point $\rho_{n}=n^{2}$. According to Theorem \ref{T6.6.5}, either $\mathscr{C}$ is unbounded in $(-\infty,1]\times U$,  or there exists $(\rho,w)\in\mathscr{C}$, with $w\neq 0$, such that $w\in Z$ where $Z$ stands for the $L^2$-orthogonal of  $\text{span}\left[\varphi_{0}\right]$ in $U$. The second alternative is impossible since $w$ is positive and hence $\int_{0}^{\pi}u\varphi_{0}\ dx>0$. Therefore, $\mathscr{C}$ is unbounded.
\par
Finally, by the \textit{a priori} bounds of Amann and L\'opez-G\'omez \cite{AL}, it is apparent that $\mathcal{P}_{\rho}(\mathscr{C})=(-\infty,1]$ where $\mathcal{P}_{\rho}:\mathbb{R}\times U \to \mathbb{R}$ stands for the $\rho$-projection operator $\mathcal{P}_{\rho}(\rho,w)=\rho$. Consequently, the problem \eqref{9.3} admits a positive solution for each $\rho<1$. Finally, since $ \frac{4-\mu^{2}}{4}<1 $, it is easily realized that, for every  $\mu>0$, the problem \eqref{9.2} admits a positive solution. In other words, \eqref{vii.1} admits a negative solution for all $\l<0$.
\end{proof}

When, instead of negative, $\l>0$, the problem \eqref{9.2} admits the following
generalization
\begin{equation}
	\label{9.4}
	\left\{\begin{array}{l}
	-w''=\rho w-\alpha(x)w^{2}+\beta(x)w^{4} \quad \text{ in } (0,\pi), \\
	w(0)=w(\pi)=0,
	\end{array}\right.
\end{equation}
where $\rho\in\mathbb{R}$ and $\alpha,\beta\in\mathcal{C}[0,\pi]$ satisfy $\a(x)>0$ and $\b(x)>0$ for all $x\in[0,\pi]$. Although the bifurcation to positive solutions from $w=0$ at $\r=1$ is supercritical in this case, because for $w\sim 0$ the problem \eqref{9.4} inherits  a sublinear nature, for sufficiently large $w$ \eqref{9.4} behaves much like a superlinear problem. Thus, the argument given in the proof of Theorem \ref{th9.2} can be also combined with the a priori bounds of \cite{AL} to infer that \eqref{vii.1} possesses a positive solution for all $\l>0$. However, as the argument in this case is more elaborate technically, because we are dealing with a superlinear indefinite problem whose structure is far more intricate, the complete technical details  will be given elsewhere.
\par

\begin{center}
	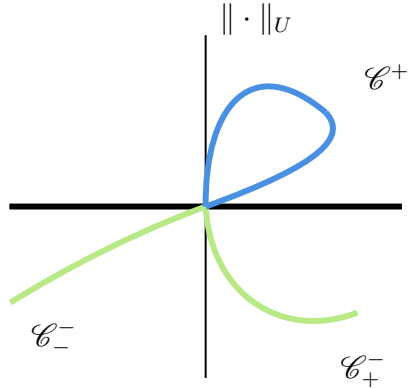
\begin{figure}[h!]
		\tikzset{every picture/.style={line width=0.75pt}} 
		
		\begin{tikzpicture}[x=0.75pt,y=0.75pt,yscale=-1,xscale=1]
			
			\draw    (123.46,31) -- (123.46,204) ;
			\draw [line width=2.25]    (222.46,117.5) -- (24.46,117.5) ;
			\draw [color={rgb, 255:red, 184; green, 233; blue, 134 }  ,draw opacity=1 ][line width=2.25]    (123.46,117.5) .. controls (126.29,163.92) and (164,185) .. (200,171) ;
			\draw [color={rgb, 255:red, 184; green, 233; blue, 134 }  ,draw opacity=1 ][line width=2.25]    (25,166) .. controls (33,161) and (73,136) .. (123.46,117.5) ;
			\draw [color={rgb, 255:red, 74; green, 144; blue, 226 }  ,draw opacity=1 ][line width=2.25]    (123.46,117.5) .. controls (123,66) and (147,39) .. (184,70) ;
			\draw [color={rgb, 255:red, 74; green, 144; blue, 226 }  ,draw opacity=1 ][line width=2.25]    (184,70) .. controls (195,82) and (183,96) .. (123.46,117.5) ;
			
			\draw (129,13.4) node [anchor=north west][inner sep=0.75pt]    {$\| \cdot \| _{U}$};
			\draw (202,44.4) node [anchor=north west][inner sep=0.75pt]    {$\mathscr{C}^{+}$};
			\draw (190,189.4) node [anchor=north west][inner sep=0.75pt]    {$\mathscr{C}_{+}^{-}$};
			\draw (34,173.4) node [anchor=north west][inner sep=0.75pt]    {$\mathscr{C}_{-}^{-}$};

		\end{tikzpicture}
		\caption{Bifurcation diagram of the component emanating at $(0,0)$}
		\label{Fig7.5}
	\end{figure}
\end{center}

According to the analysis carried out in Section \ref{SCh143}, we already know that there are two analytic curves of negative solutions bifurcating from $(0,0)$: One in the direction of $\l>0$ and another in the direction of $\l<0$. Subsequently, we denote by $\mathscr{C}^{-}_{+}$ (resp. $\mathscr{C}^{-}_{-}$)  the connected component of the set of negative solutions
\begin{equation*}
	\mathscr{N}:=\{(\l,u)\in\mathfrak{F}^{-1}(0): \;u\ll 0\}\subset \mathbb{R}\times U,
\end{equation*}
emanating from $(0,0)$ in the direction $\l>0$ (resp. $\l<0$). The next result provides us with their global behavior.

\begin{theorem}
	The components $\mathscr{C}^{-}_{+}$ and  $\mathscr{C}^{-}_{-}$ are unbounded and disjoint, i.e., $\mathscr{C}^{-}_{+}\cap\mathscr{C}^{-}_{-}=\emptyset$.
\end{theorem}

\begin{proof}
	By Lemma \ref{le9.1}, \eqref{vii.1} cannot admit a negative solution at $\l=0$. So,
	$\mathscr{C}^{-}_{+}\cap\mathscr{C}^{-}_{-}=\emptyset$. Let us denote by $\gamma_{+}:(0,\varepsilon)\to\mathbb{R}\times U$ and $\gamma_{-}:(-\varepsilon,0)\to\mathbb{R}\times U$ the two local curves of negative solutions of \eqref{vii.1} that emanate from $(0,0)$ in the direction of $\mathscr{C}^{-}_{+}$ and $\mathscr{C}^{-}_{-}$, respectively. Adapting the argument of the proof of Theorem \ref{th7.8}, it
	is easily seen  that $\gamma_{+}$ and $\gamma_{-}$ consist of regular points for sufficiently small $\varepsilon>0$. Thus, by Theorem \ref{th.BT}, there are two global locally injective continuous curves $\Gamma_{+}:[0,\infty)\to\mathbb{R}\times U$, $\Gamma_{+}([0,\infty))\subset\mf{F}^{-1}(0)$, $\Gamma_{-}:(-\infty,0]\to\mathbb{R}\times U$, $\Gamma_{+}((-\infty,0])\subset\mf{F}^{-1}(0)$, 	that extend $\gamma_{+}$ and $\gamma_{-}$, respectively, and satisfy one of the alternatives (a) and (b). By Lemma \ref{levii.5}, $\Gamma_{+}([0,\infty))\subset\mathscr{C}^{-}_{+}$, $\Gamma_{-}((-\infty,0])\subset \mathscr{C}^{-}_{-}$.
	Since $\mathscr{C}^{-}_{+}\cap\mathscr{C}^{-}_{-}=\emptyset$, the curves $\Gamma_{\pm}$ cannot form a loop. 	 Thus, the  alternative (a) cannot happen. Therefore, $\lim_{t\uparrow\infty}\|\Gamma_{+}(t)\|_{U}=\infty$ and $ \lim_{t\downarrow-\infty}\|\Gamma_{-}(t)\|_{U}=\infty$.
	This entails that $\mathscr{C}^{-}_{+}$ and $\mathscr{C}^{-}_{-}$  are unbounded. The proof is complete.
\end{proof}

As a consequence of the analysis carried out in Sections \ref{SDODP} and \ref{SNS}, the global structure of the component
$\mathscr{C}$ of solutions of \eqref{vii.1} emanating from $(0,0)$ looks like shows
Figure \ref{Fig7.5}. Essentially, it consist of a bounded component of positive solutions $(\l,u)$ with $\l>0$ forming a loop around $(0,0)$ plus two unbounded branches of negative solutions: one for $\l>0$ and another for $\l<0$. It should be emphasized that, thanks to the results of Section \ref{SGG}, the set $\mathfrak{F}^{-1}(0)$ is an analytic graph.
\par
To conclude this section, it should be noted that the loop of positive solutions, $\mathscr{C}^+$, does not satisfies any of the alternatives (i)--(iii) of Theorem \ref{T6.6.5}. This fact shows the relevance that $\chi$ is odd for the validity of these unilateral theorems.

\section{A quasilinear problem of mixed type}\label{SQPMT}

\noindent In this section, $\Omega$ stands for a bounded subdomain of $\mathbb{R}^{N}$ of class $\mathcal{C}^{2}$ such that $\partial\Omega=\Gamma_{0}\uplus\Gamma_{1}$, where $\Gamma_{0}$ and $\Gamma_{1}$ are two open and closed subsets of $\partial\Omega$. Thus, both are of class $\mc{C}^2$. Our main goal is analyzing the existence of regular positive solutions for the multidimensional quasilinear boundary value problem
\begin{equation}
	\label{6.6.71}
	\left\{\begin{array}{ll}
		-\text{div}\left(\frac{\nabla u}{\sqrt{1+|\nabla u|^{2}}}\right)=\lambda a(x) u +g(x,u)u & \text{ in } \Omega, \\
		\mathcal{B}u=0 & \text{ on } \partial\Omega,
	\end{array}
	\right.
\end{equation}
where $a\in\mathcal{C}(\overline{\Omega},\R)$ changes sign in $\O$, $g\in\mathcal{C}^1 (\overline{\Omega}\times\mathbb{R})$ satisfies $g(\cdot,0)=0$, and the boundary operator $\mathcal{B}:\mathcal{C}(\Gamma_{0})\otimes\mathcal{C}(\Gamma_{1})\to\mathbb{R}$ is defined by
\begin{equation*}
	\mathcal{B}u:=\left\{
	\begin{array}{ll}
		u & \text{ on } \Gamma_{0} \\
		\partial_{\mathbf{n}}u+\beta(x)u & \text{ on } \Gamma_{1}
	\end{array}\right.
\end{equation*}
where $\mathbf{n}$ is the exterior normal vector field to $\O$ on $\partial\Omega$, and $\beta\in\mathcal{C}(\Gamma_{1})$. Thus, $\mc{B}$ is a non-classical mixed boundary operator on
$\p\O$; non-classical because $\b$ can change sign on $\G_1$. When $\G_1=\emptyset$, then
$\mc{B}$ coincides with the Dirichlet boundary operator, $\mc{D}$. When, $\G_0=\emptyset$ and
$\b=0$, $\mc{B}$ equals the Neumann boundary operator, $\mc{N}$.
\par
In this section, we will assume that $(-\Delta-\l_0 a(x),\mc{B},\O)$ satisfies the maximum principle for some $\l_0\in\R$, i.e., thanks to \cite[Th. 7.10]{LG13},
\begin{equation}
	\label{6.6.72}
	\s_1[-\D-\l_0 a(x),\mc{B},\O]>0,
\end{equation}
where we are denoting by $\s_1[-\D-\l a,\mc{B},\O]$ the principal eigenvalue of $-\D-\l a$ in $\O$ subject to the boundary operator $\mc{B}$ on $\p\O$. The condition \eqref{6.6.72} holds with $\l_0=0$ if, for example, $\Gamma_{0}=\emptyset$ and $\b \gneq 0$, or $\Gamma_{0}\neq\emptyset$ and $\beta\geq 0$, as in both cases, the constant function $h=1$ provides us with a positive strict supersolution of $(-\Delta,\mc{B},\O)$ as discussed in \cite[Ch. 7]{LG13}.
\par
For every $p>N$, we can define a nonlinear operator $\mathfrak{F}: \mathbb{R}\times W^{2,p}_{\mc{B}}(\Omega) \longrightarrow L^{p}(\Omega)$ through
\begin{equation*}
	\mathfrak{F}(\lambda,u)=-\text{div}\left(\frac{\nabla u}{\sqrt{1+|\nabla u|^{2}}}\right)-\lambda a(x) u -g(x,u)u.
\end{equation*}
Then, $\mathfrak{F}\in\mathcal{C}^{1}(\mathbb{R}\times W^{2,p}_{\mc{B}}(\Omega),L^{p}(\Omega))$ and
\begin{equation*}
	\mathfrak{L}(\lambda)u=D_{u}\mathfrak{F}(\lambda,0)u=-\Delta u-\lambda a(x)u \quad \hbox{for all}\;\;
	u\in W^{2,p}_{\mc{B}}(\Omega).
\end{equation*}
Since
$$
\mu(\l,\g)\equiv \s_1[-\D-\l a(x)+\g,\mc{B},\O]= \s_1[-\D-\l a(x),\mc{B},\O]+\g,
$$
it is apparent that, for sufficiently large $\g \in\R$,  $\mu(\l,\g)>0$. Thus,
$\mathfrak{L}(\lambda)+\gamma J$ lies in $GL(W^{2,p}_{\mc{B}}(\Omega),L^{p}(\Omega))$,
where $J: W^{2,p}_{\mc{B}}(\Omega)\to L^{p}(\Omega)$ stands for the canonical embedding, which is a linear compact operator. Consequently, $\mathfrak{L}(\lambda)=(\mathfrak{L}(\lambda)+\gamma J)-\gamma$ can be expressed as the sum of a compact operator plus an isomorphism. In particular,
$\mf{L}(\l)\in \Phi_0(W^{2,p}_{\mc{B}}(\O),L^p(\O))$. Therefore, $\mathfrak{L}$ is an analytic curve of Fredholm operators of index zero. 
By \eqref{6.6.72},
$$
\mf{L}(\l_0)=-\D-\l_0 a \in GL(W^{2,p}_{\mc{B}}(\Omega),L^{p}(\Omega)).
$$
Thus, thanks to Theorems 4.4.1 and 4.4.3 of \cite{LG01},  $\Sigma(\mf{L})$ is a discrete subset of
$\R$. Moreover, $\Sigma(\mf{L})$ consists of algebraic eigenvalues. Note that $\l \in\Sigma(\mf{L})$ if and only if the linear problem
\begin{equation}
	\label{6.6.73}
	\left\{
	\begin{array}{ll}
		-\Delta u = \lambda a(x)u & \text{ in }\;\; \Omega, \\
		\mc{B}u=0 & \text{ on } \;\;\partial\Omega,
	\end{array}
	\right.
\end{equation}
admits some solution $u\neq 0$ in $W^{2,p}_\mc{B}(\O)$. As $a(x)$ changes of sign in $\O$, it follows from
\cite[Th. 9.4]{LG13} that \eqref{6.6.73} has two (principal) eigenvalues, $\l_-<\l_0<\l_+$, associated with each of them the problem \eqref{6.6.73} possesses a (principal) positive eigenfunction, $\v_\pm \gg 0$, unique up to a positive multiplicative constant. By $\v_\pm \gg 0$ it is meant that $\v_\pm(x)>0$ for all $x\in\O\cup \G_1$, and that $\p_\mathbf{n}\v_\pm(x)<0$ if $x\in \G_0$. Moreover, also based on \cite[Th. 9.4]{LG13}, it becomes apparent that $\l_\pm$ are algebraically simple and that $\mathfrak{L}_{1}(\lambda_{\pm})(\varphi_{\pm})\notin R[\mathfrak{L}(\lambda_{\pm})]$. 
Thus, according to \eqref{ii.3}, $\chi[\mathfrak{L},\lambda_{\pm}]=1$. Therefore,
when, in addition, $g$ is assumed to be of class $\mc{C}^2$, by the Crandall--Rabinowitz theorem \cite{CR},
there exist $\varepsilon>0$ and two $\mathcal{C}^{1}$-functions
\begin{equation*}
	\mu_{\pm}:(-\varepsilon,\varepsilon)\to\mathbb{R}, \quad \psi_{\pm}:(-\varepsilon,\varepsilon)\to Z\equiv \Big\{u\in W^{2,p}_{\mc{B}}(\Omega):\;\int_{\Omega}u\varphi_{\pm}=0\Big\}
\end{equation*}
such that $\mu_\pm(0)=\l_\pm$, $\psi_\pm(0)=0$, and, in a neighborhood of $(\l_\pm,0)$,  the solution set $\mathfrak{F}^{-1}(0)$ consists
of the trivial curve $(\l,0)$ plus the bifurcated $\mc{C}^1$-curve $(\l_\pm(s),u_\pm(s))\equiv (\mu_{\pm}(s),s\varphi_{\pm}+\psi_{\pm}(s))$, $s\sim 0$. At least for sufficiently small $s>0$, the solutions $(\l_\pm(s),u_\pm(s))$ provide us
with positive solutions of \eqref{6.6.71}.
Actually, those are the unique small positive solutions of
\eqref{6.6.71} as no other eigenvalue of \eqref{6.6.73} admits a positive eigenfunction (see, e.g.,
\cite[Ch. 6]{LG13}).
\par
Nevertheless, even when $g$ is simply of class $\mc{C}^1$ regularity, by Theorem
\ref{T6.2.1}, there are two components, $\mf{C}_\pm^+$, of non-trivial solutions of
\eqref{6.6.71} emanating from $(\l,0)$ at $\l_\pm$. These components consist of the  smooth curves $(\l_\pm(s),u_\pm(s))$ for sufficiently small  $s>0$ if, in addition,
$g$ is $\mc{C}^2$. Based on the next result,
a direct adaptation of the argument of \cite[Sec. 6.5]{LG01} shows that $\mf{C}_\pm^+$ actually provides us with the component of positive solutions of \eqref{6.6.71} emanating from $(\l,0)$ at $\l_\pm$.

\begin{lemma}
	Suppose that $u\gneq 0$ is a solution of \eqref{6.6.71} in $W^{2,p}_\mc{B}(\O)$. Then, $u\gg 0$ in the sense
	that $u(x)>0$ for all $x\in\O\cup \G_1$ and $\p_\mathbf{n}u(x)<0$ for all $x\in\G_0$.
\end{lemma}
\begin{proof}
	By the Sobolev embedding theorem, $u\in\mc{C}^1(\bar\O)$. Thus, \eqref{6.6.71} can be written equivalently
	in the form
	$$
	-(1+|\nabla u|^2)^{-\frac{1}{2}}\D u+(1+|\nabla u|^2)^{-\frac{3}{2}} \sum_{i,j=1}^N\frac{\p u}{\p x_i} \frac{\p u}{\p x_j}    \frac{\p^2 u}{\p x_i\p x_j}  =\l a(x)u+g(x,u)u,
	$$
	where the second order differential operator
	$$
	\mc{L}(u,\nabla u)w:= -(1+|\nabla u|^2)^{-\frac{1}{2}}\D w+(1+|\nabla u|^2)^{-\frac{3}{2}} \sum_{i,j=1}^N\frac{\p u}{\p x_i} \frac{\p u}{\p x_j}    \frac{\p^2 w}{\p x_i\p x_j}
	$$
	is uniformly elliptic in $\O$ (see, e.g., Gilbarg--Trudinger \cite[p. 261]{GT}). Thus, since $u\gneq 0$ and \eqref{6.6.71} can be equivalently expressed in the form
	$$
	\left( \mc{L}(u,\nabla u) -\l a(x)-g(x,u)\right) u=0,
	$$
	by the uniqueness of the principal eigenvalue $\s_1[\mc{L}(u,\nabla u) -\l a-g(\cdot,u),\mc{B},\O]=0$ 
	and $u\gneq 0$ must be a principal eigenfunction associated to $0$. Finally, adapting the arguments of \cite[Ch. 7]{LG13}, it follows from the Krein--Rutman theorem that $u\gg 0$.
\end{proof}

Therefore, by Theorem \ref{T6.6.5}, each of the components $\mf{C}_\pm^+$ satisfies some of the following alternatives:
\begin{enumerate}
	\item[{\rm (i)}] $\mf{C}_\pm^+$ is not compact in $\R\times W^{2,p}_\mc{B}(\O)$.
	\item[{\rm (ii)}] There exists $\l_{\pm,1}\neq \l_\pm$ such that $(\l_{\pm,1},0)\in \mf{C}_\pm^+$.
	\item[{\rm (iii)}]  There exist $\l\in\R$ and $z\in Z\setminus\{0\}$ such that $(\l,z)\in \mf{C}_\pm^+$.
\end{enumerate}
Since $z\neq 0$ and  $\v_\pm\gg 0$, it follows from $\int_\O z\v_\pm dx =0$ that $z$ changes sign
in $\O$ and hence $(\l,z)\notin \mf{C}_\pm^+$. Therefore, either (i), or (ii) occurs and the next result holds.

\begin{theorem}
	Either $\mf{C}_+^+=\mf{C}_-^+$, or $\mf{C}_+^+\cap \mf{C}_-^+=\emptyset$ and, in such case, there exist two sequences $(\l_{\pm,n},u_{\pm,n})\in \mf{C}_\pm^+$, $n\geq 1$, such that either
	$\lim_{n\to\infty}|\l_{\pm,n}|=\infty$, or $\lim_{n\to \infty}\|u_{\pm,n}\|_{W^{2,p}_\mc{B}(\O)}=\infty$.
\end{theorem}

Naturally,  $\mf{C}_+^+=\mf{C}_-^+$ occurs when the alternative (ii) holds. Therefore, when
(ii) fails, both components of positive solutions must satisfy the alternative (i). Naturally, the precise global behavior of the components $\mf{C}_\pm^+$ depends on the particular choice of the map $g(x,u)$ in the setting of the problem \eqref{6.6.71}, as well as on the nature of the boundary conditions.
The following result provides us with a sufficient condition so that $\mf{C}_+^+\cap \mf{C}_-^+=\emptyset$
under Dirichlet boundary conditions.

\begin{theorem}
	Suppose $\mc{B}=\mc{D}$ and $g(x,u)\leq 0$ for all $x\in\O$ and $u\geq 0$. Then, $u=0$ is the unique non-negative solution, in $W^{2,p}_0(\Omega)$, of \eqref{6.6.71} for $\l=0$. Thus, $\mf{C}_+^+\cap \mf{C}_-^+=\emptyset$.
\end{theorem}
\begin{proof}
	Let $u\in W^{2,p}_0(\Omega)$ be a positive solution of \eqref{6.6.71} with $\lambda=0$. Then, multiplying the differential equation by $u$ and integrating by parts in $\O$, we obtain that
	\begin{align*}
		\int_{\Omega}\frac{|\nabla u|^{2}}{\sqrt{1+|\nabla u|^{2}}} \, dx & = -\int_{\Omega}\mathrm{div} \left(\frac{\nabla u}{\sqrt{1+|\nabla u|^{2}}}\right)u \, dx\\ & =
		\lambda \int_{\Omega}a(x)u^2 \, dx+\int_{\Omega}g(x,u)u^{2}(x)\, dx\leq 0.
	\end{align*}
	Thus, $\nabla u=0$ in $\O$ and therefore, $u=0$ in $\O$, because $\O$ is connected.
	\par
	On the other hand, since $\s_1[-\D,\mc{D},\O]>0$, the condition  \eqref{6.6.72} holds with
	$\l_0=0$. Hence, $\l_-<0<\l_+$. Consequently, since \eqref{6.6.71} cannot
	admit any positive solution at $\l=0$, $\mf{C}_+^+\cap \mf{C}_-^+=\emptyset$.
	This ends the proof.
\end{proof}

When $\mc{B}=\mc{N}$, $\l=0$ is always a principal eigenvalue of \eqref{6.6.73}. Indeed, any positive constant satisfies \eqref{6.6.73} if $\l=0$. In such case, based on a result of Brown and Lin \cite{BL}, $\l_-=0$ and $\l_+>0$ if $\int_\O a(x)\,dx<0$, while $\l_-<0$ and $\l_+=0$ if $\int_\O a(x)\,dx>0$. When, $\int_\O a(x)\,dx =0$, then $\l=0$ is the unique eigenvalue of \eqref{6.6.73} (see
the discussion on page 312 of \cite{LG13}). Even in the simplest one-dimensional prototype models, the classical solutions of \eqref{6.6.71} can develop singularities for certain values of the parameter $\l$. Consider, for instance, the simplest one dimensional problem
\begin{equation}
	\label{6.6.75}
	\left\{\begin{array}{l}
		-\left(\frac{u'}{\sqrt{1+(u')^{2}}}\right)'=\lambda a(x) u \quad \text{ in } (0,1), \\[2ex]
		u'(0)=u'(1)=0, \end{array} \right.
\end{equation}
with $\int_0^1a(x)\,dx<0$. Then, $\l_-=0$ and $\l_+>0$. Moreover,
$\{(r,0) :\; r\in\R\}\subset \mf{C}_-^+$. In particular, $\mf{C}_-^+$ is unbounded in $L^p(0,1)$.
This particular model fits within the abstract setting of L\'{o}pez-G\'{o}mez and Omari \cite{LGO}, with $p=q=2$, if there exists $z\in (0,1)$ such that $a(x)>0$ for all $x\in (0,z)$ and $a(x)<0$ for all $x\in (z,1)$. Subsequently, we will assume that $a(x)$ satisfies this condition. Then, thanks to \cite[Th. 1.1]{LGO}, for sufficiently small $\l>0$, any positive solution of \eqref{6.6.75} is singular if
\begin{equation}
	\label{6.6.76}
	\left( \int_x^z a(t)\,dt\right)^{-\frac{1}{2}}\in L^1(0,z)\cap L^1(z,1).
\end{equation}
Therefore, if $a(x)$ satisfies  \eqref{6.6.76}, then there exists
$\o>0$ such that $\mc{P}_\l (\mf{C}_+^+)\subset [\o,\infty)$, where $\mc{P}_\l$ stands for the $\l$-projection operator $\mc{P}_\l(\l,u)=\l$. In particular, $\mf{C}_+^+\cap \mf{C}_-^+=\emptyset$.
\par
By having a glance at \eqref{6.6.76} it becomes apparent that it fails whenever the function $a$ is differentiable at the nodal point $z$, whereas it  occurs when $a(x)$ is discontinuous at $z$, like in the special case when it equals a positive constant, $A >0$, in $[z-\e,z)$ and a negative constant, $B<0$, in $(z, z+\d]$, for some $\e, \d>0$. Thus, there is a huge contrast in the nature of the positive solutions of \eqref{6.6.75} according to the integrability properties of $a(x)$ near the node $z$. Figure 1 of L\'{o}pez-G\'{o}mez and Omari \cite{LGO} shows how, under condition  \eqref{6.6.76}, the solutions along
$\mf{C}_+^+$ can develop singularities when $\l$ approximates zero, as their gradients develop
singularities at the node $z$ when they become sufficiently large. Naturally, at these critical values,
$\|u'\|_{L^p}$ becomes unbounded, though $\|u\|_{L^p}$ always stay bounded. Actually, thanks to L\'{o}pez-G\'{o}mez and Omari \cite[Th. 7.2]{LGO}, any positive bounded variation solution of \eqref{6.6.75} must be regular if  the integrability condition \eqref{6.6.76} fails. Therefore, in this case, the solutions along $\mf{C}_\pm^+$ cannot develop singularities. Characterizing whether, or not, the positive solutions of \eqref{6.6.71} can develop singularities in a multidimensional context seems  extremely challenging and remains an open problem.


\begin{thebibliography}{xx}
	
	\bibitem{AL} H. Amann and J. López-Gómez, A priori bounds and multiple solutions for superlinear indefinite elliptic problems, \emph{J. Differential Equations} \textbf{146} (1998), 336--374.
	
	\bibitem{AW} H. Amann and S. A. Weiss, On the uniqueness of the topological degree, \emph{Math. Z.} \textbf{130} (1973), 39--54.
	
	\bibitem{B} T. Banachiewicz, Zur Berechnung der Determinanten, wie auch der Inversen, und zur darauf basierten Auflosung der Systeme linearer Gleichungen, \emph{Acta Astron. Ser. C} \textbf{3} (1937), 41--67.
	
	\bibitem{BF1} P. Benevieri and M. Furi, A simple notion of orientability for Fredholm maps of index zero between Banach manifolds and degree theory, \emph{Ann. Sci. Math. Quebec} \textbf{22} (1998), 131--148.
	
	\bibitem{BF2} P. Benevieri and M. Furi, On the concept of orientability for Fredholm maps between real Banach manifolds, \emph{Topol. Methods Nonl. Anal.} \textbf{16} (2000), 279--306.
	
	\bibitem{BF3} P. Benevieri and M. Furi, On the uniqueness of the degree for nonlinear Fredholm maps of index zero between Banach manifolds, \emph{Comm. Appl. Anal.} \textbf{15}, (2011), 203--216.
	
	\bibitem{B} M. S. Berger, \emph{Nonlinearity and Functional Analysis}, Lectures on Nonlinear Problems in Mathematical Analysis, Academic Press, Inc., 1977.

\bibitem{BL} K. J. Brown and C. C. Lin, On the existence of positive eigenfunctions for an eigenvalue problem with indefinite weight function, \emph{J. Math. Anal. Appn.} \textbf{75} (1980), 112--120.
	
	\bibitem{Br} L. E. J. Brouwer, \"Uber Abbildung von Mannigfaltigkeiten, \emph{Math. Ann.} \textbf{71}
	(1911), 97--115.
	
	\bibitem{RFB} R. F. Brown, \emph{A Topological Introduction to Nonlinear Analysis}, Second edition,
Springer, 2004.
	
\bibitem{BT} B. Buffoni and J. Toland, \emph{Analytic Theory and Global Bifurcation: An introduction}, Princeton Series in Applied Mathematics, Princeton, 2003.
	
\bibitem{CR} M. G. Crandall and P. H. Rabinowitz, Bifurcation from simple eigenvalues,
	\emph{J. Funct. Anal.} \textbf{8} (1971), 321--340.

\bibitem{CRex}  M. G. Crandall and P. H. Rabinowitz, Bifurcation, perturbation from simple eigenvalues
and linearized stability, \emph{Arch. Rat. Mech. Anal.} \textbf{52} (1973), 161--180.

	\bibitem{Da73} E. N. Dancer, Bifurcation Theory for Analytic Operators, \emph{Proc. London Math. Soc.} (3) \textbf{26} (1973), 359--384.

\bibitem{Da732} E. N. Dancer, Global structure of the solutions of non-linear real analytic eigenvalue problems, \emph{Proc. London Math. Soc.} (3) \textbf{27} (1973), 747--765.

\bibitem{Da} E. N. Dancer, Bifurcation from simple eigenvalues and eigenvalues of geometric multiplicity one, \emph{Bull. London Math. Soc.} \textbf{34} (2002), 533--538.
	
	\bibitem{Es} J. Esquinas, Optimal multiplicity in local bifurcation theory, II: General case, \emph{J. Diff. Eqns.} \textbf{75} (1988), 206--215.
	
	\bibitem{ELG} J. Esquinas and J. L\'opez-G\'omez, Optimal multiplicity in local bifurcation theory, I: Generalized generic eigenvalues, \emph{J. Diff. Eqns.} \textbf{71} (1988), 72--92.
	
	\bibitem{FP1} P. M. Fitzpatrick and J. Pejsachowicz, Parity and generalized multiplicity, \emph{Trans. Amer. Math. Soc.} \textbf{326} (1991), 281--305.

\bibitem{FLG} M. Fencl and J. L\'{o}pez-G\'{o}mez, Nodal solutions of weighted indefinite problems, \emph{J. of Evol. Equ.} (2020), https://doi.org/10.1007/s00028-020-00625-7.
	
	\bibitem{FP2} P. M. Fitzpatrick and J. Pejsachowicz, Orientation and the Leray--Schauder theory for fully nonlinear elliptic boundary value problems, \emph{Mem. Amer. Math. Soc.} \textbf{483}, Providence, 1993.
	
	\bibitem{FPRa} P. M. Fitzpatrick, J. Pejsachowicz and P. J. Rabier, Orientability of Fredholm families and topological degree for orientable nonlinear Fredholm mappings, \emph{J. Functional Analysis} \textbf{124} (1994), 1--39.
	
	\bibitem{FPRb} P. M. Fitzpatrick, J. Pejsachowicz and P. J. Rabier, The Degree of Proper $\mathcal{C}^{2}$ Fredholm mappings I, \emph{J. Reine Angew. Math.} \textbf{427} (1992), 1--33
	
	\bibitem{Fr} E. I. Fredholm, Sur une classe d'equations fonctionnelles, \emph{Acta Math} \textbf{27}, (1903), 365--390.
	
\bibitem{OF} O. Forster, \emph{Lectures on Riemann Surfaces}, Springer, Berlin, 1981.
	
	\bibitem{Fu} L. F\"uhrer, \emph{Theorie des Abbildungsgrades in endlichdimensionalen R\"aumen},
	Ph. D. Dissertation, Frei Univ. Berlin, Berlin, 1971.

\bibitem{GT} D. Gilbarg and N. S. Trudinger, \emph{Ellliptic Partial Differential Equations of
Second Order}, Classics in Mathematics, Springer, Berlin, 1988.
	
	\bibitem{GS}  I. C. G\"{o}hberg and E. I. Sigal, An Operator Generalization of the Logarithmic
	Residue Theorem and the Theorem of Rouch\'{e}, \emph{Math. Sbornik} \textbf{84(126)} (1971),
	607--629. English Trans.: \emph{Math. USSR Sbornik} \textbf{13} (1971), 603--625.

\bibitem{GS} M. Golubitsky and D. G. Schaeffer, \emph{Singularitites and Groups in Bifurcation Theory}, Vol. 1, Applied Mathematical Sciences 51, Springer, New York, 1985.

\bibitem{Henry} D. Henry, \emph{Geometric Theory of Semilinear Parabolic Equations}, Lectures Notes in Mathematics 840, Springer, 1981.

	\bibitem{Iz}  J. Ize, Bifurcation Theory for Fredholm Operators, \emph{Mem. Amer. Math. Soc.} \textbf{174}, Providence, 1976.

\bibitem{Ka} T. Kato, \emph{Perturbation Theory for Linear Operators}, Springer, Berlin, 1995.
	
	\bibitem{Ki} H. Kielh\"ofer, Degenerate Bifurcation at simple Eigenvalues and Stability of Bifurcating Solutions, \emph{J. Functional Anal.} \textbf{38}, (1980) 416--441.
	
	\bibitem{Kr} M. A. Krasnoselskij, \emph{Topological Methods in the Theory of Nonlinear Integral Equations}, Pergamon Press, New York, 1964.
	
	\bibitem{K} N. Kuiper, The Homotopy Type of the Unitary Group of Hilbert Space, \emph{Topology} \textbf{3}, (1965), 19--30.
	
	\bibitem{LMA} B. Laloux and J. Mawhin, Coincidence index and multiplicity,
	\emph{Trans. Amer. Math. Soc.} \textbf{217} (1976), 143--162.
	
	\bibitem{LMB} B. Laloux and J. Mawhin, Multiplicity, Leray--Schauder formula and bifurcation,
	\emph{J. Diff. Eqns.} \textbf{24} (1977), 309--322.
	
	\bibitem{LS} J. Leray and J. Schauder, Topologie et \'equations fonctionelles, \emph{Ann. Sci.
		\'Ecole Norm. Sup. S\'er. 3} \textbf{51} (1934), 45--78.
	
	\bibitem{Ll} N. G. Lloyd, \emph{Degree Theory}, Cambridge Tracts in Mathematics, Cambridge University Press, Cambridge, 1978.

	
	\bibitem{LG01} J. L\'opez-G\'omez, \emph{Spectral Theory and Nonlinear Functional Analysis}, CRC Press, Chapman and Hall RNM vol. 426, Boca Raton, 2001.

\bibitem{LG13} J. L\'opez-G\'omez, \emph{Linear Second Order Elliptic Operators}, World Scientific, Singapore, 2013.
	
	\bibitem{LG02} J. L\'{o}pez-G\'{o}mez, Global bifurcation for Fredholm operators, \emph{Rend. Istit. Mat. Univ. Trieste} \textbf{48} (2016), 539--564. DOI: 10.13137/2464-8728/13172.

\bibitem{LGMM} J. L\'{o}pez-G\'{o}mez and M. Molina-Meyer, Bounded components of positive solutions of
abstract fixed point equations: mushrooms, loops and isolas, \emph{J. Diff. Equ.} \textbf{209} (2005),
416--441.
	
	\bibitem{LGMC05} J. L\'opez-G\'omez and C. Mora-Corral, Counting zeroes of $\mathcal{C}^1$-Fredholm maps of index zero, \emph{Bull. London Math. Soc.} \textbf{37} (2005) 778--792.
	
	\bibitem{LGM} J. L\'opez-G\'omez and C. Mora-Corral, Counting solutions of nonlinear abstract equations, \emph{Top. Meth. Nonl. Anal.} \textbf{24}, (2004), 307--335.
	
	\bibitem{LGMC} J. L\'opez-G\'omez and C. Mora-Corral, \emph{Algebraic Multiplicity of Eigenvalues of Linear Operators}, Operator Theory, Advances and Applications vol. 177, Birkh\"auser, Basel, 2007.
	
	\bibitem{LGMCC} J. L\'opez-G\'omez and C. Mora-Corral, Counting solutions in bifurcation problems, \emph{Journal of Mathematical Sciences} \textbf{150}, (2008), 2395--2407.

\bibitem{LGO}  J. L\'{o}pez-G\'{o}mez and P. Omari, Characterizing the formation of singularitites in a
superlinear indefinite problem related to the mean curvature operator, \emph{J. Diff. Eqns.}
\textbf{269} (2020), 1544--1570.

\bibitem{LGOa} J. L\'{o}pez-G\'{o}mez and P. Omari, Regular versus singular solutions in quasilinear indefinite problems with sublinear potentials, \emph{submitted}.
	
	\bibitem{JJ} J. L\'opez-G\'omez and J. C. Sampedro, Algebraic multiplicity and topological degree for Fredholm operators, \emph{Nonl. Anal.} \textbf{201} (2020), 112019, 28 pp.
	
	\bibitem{JJ2} J. L\'opez-G\'omez and J. C. Sampedro, Axiomatization of the Degree of Fitzpatrick, Pejsachowicz and Rabier, \emph{J. Fixed Point Theory Appl.} \textbf{24}, Number 8 (2022).
	
\bibitem{JJ3} J. L\'opez-G\'omez and J. C. Sampedro, New analytical and geometrical aspects of the algebraic multiplicity, \emph{J. Math. Anal. Appns.} \textbf{504}, Issue 1 (2021).

\bibitem{LGT} J. L\'{o}pez-G\'{o}mez and A. Tellini, Generating an arbitrarily large number of isolas in a superlinear indefinite problem, \emph{Nonl. Anal.} \textbf{108} (2014), 223--248.

\bibitem{Ma} R. J. Magnus,  A generalization of multiplicity and the problem of bifurcation,
	\emph{Proc. Lond. Math. Soc.} \textbf{32} (1976), 251--278.

\bibitem{Maw} J. L. Mawhin, Continuation theorems and periodic solutions of ordinary differential equations. Topological methods in differential equations and inclusions (Montreal, PQ, 1994), 291--375, NATO Adv. Sci. Inst. Ser. C Math. Phys. Sci., 472, Kluwer Acad. Publ., Dordrecht, 1995.

	\bibitem{MC} C. Mora-Corral, On the Uniqueness of the Algebraic Multiplicity, \emph{J. London Math. Soc.} \textbf{69} (2004), 231--242.

\bibitem{Ni} L. Nirenberg, \emph{Topics in Nonlinear Functional Analysis}, Lectures Notes of the
Courant Institute of Mathematical Sciences, UNY, New York, 1974.
	
	\bibitem{QS} F. Quinn and A. Sard, Hausdorff Conullity of Critical Images of Fredholm Maps, \emph{Amer. J. Math. } \textbf{94}, (1972), 1101--1110.
	
	\bibitem{PR} J. Pejsachowicz and P. J. Rabier, Degree theory for $\mathcal{C}^{1}$ Fredholm mappings of index 0, \emph{J. Anal. Math.} \textbf{76} (1998), 289.
	
	\bibitem{Ra1} P. J. Rabier, Generalized Jordan chains and two bifurcation theorems of Krasnoselskii,
	\emph{Nonl. Anal. TMA} \textbf{13} (1989), 903--934.
	
	\bibitem{Ra}  P. H. Rabinowitz, Some global results for nonlinear eigenvalue problems, \emph{J.
		Funct. Anal.} \textbf{7} (1971), 487--513.
	
	\bibitem{Sa} A. Sard, The measure of the critical values of differentiable maps, \emph{Bull. Amer. Math. Soc.} \textbf{48} (1942), 883--890.
	
	\bibitem{XW} J. Shi and X. Wang, On global bifurcation for quasilinear elliptic systems on bounded domains, \emph{J. Diff. Eqns. } \textbf{246} (2009), 2788--2812.
	
	\bibitem{Sm} S. Smale, An infinite dimensional version of Sard's theorem, \emph{Amer. J. Math.} \textbf{87} (1965), 861--866.

\bibitem{St} M. Stiefenhofer, Direct sum condition and Artin Approximation in Banach spaces, arXiv: 1905.07583v1[math.AG], 2019.
	
	\bibitem{W} G. T. Whyburn, \emph{Topological Analysis}, Princeton University Press, Princeton, 1958.
	
\end{thebibliography}
\end{document}